\documentclass{article}%
\usepackage{amsfonts}
\usepackage{amsmath}
\usepackage{amssymb}
\usepackage{graphicx}
\usepackage[margin=4.3cm]{geometry}%
\usepackage{color}
\providecommand{\U}[1]{\protect\rule{.1in}{.1in}}
\newtheorem{theorem}{Theorem}[section]

\newtheorem{definition}[theorem]{Definition}

\newtheorem{lemma}[theorem]{Lemma}

\newtheorem{proposition}[theorem]{Proposition}
\newtheorem{remark}[theorem]{Remark}

\newenvironment{proof}[1][Proof]{\noindent\textbf{#1.} }{\phantom{1}\hfill\rule{0.5em}{0.5em}\medskip}

\numberwithin{equation}{section}

\newcommand{\Addresses}{{
  \bigskip
  \footnotesize

(I.~Fonseca) \textsc{Department of Mathematical Sciences, Carnegie Mellon University, Pittsburgh PA 15213-3890, USA}\par\nopagebreak
  \textit{E-mail address}: \texttt{fonseca@andrew.cmu.edu}

  \medskip

(G.~Leoni) \textsc{Department of Mathematical Sciences, Carnegie Mellon University, Pittsburgh PA 15213-3890, USA}\par\nopagebreak
  \textit{E-mail address}: \texttt{giovanni@andrew.cmu.edu}

  \medskip

(M.G.~Mora) \textsc{Dipartimento di Matematica, Universit\`a di Pavia, via Ferrata 1, 27100 Pavia, Italy}\par\nopagebreak
  \textit{E-mail address}: \texttt{mariagiovanna.mora@unipv.it}

}}

\begin{document}

\title{A Second Order Minimality Condition for a Free-Boundary Problem}
\author{Irene Fonseca
\and Giovanni Leoni
\and Maria Giovanna Mora}

\maketitle

\begin{abstract}
\noindent
The goal of this paper is to derive in the two-dimensional case necessary and
sufficient minimality conditions in terms of the second variation for the
functional
\[
v\mapsto\int_{\Omega}\big(|\nabla v|^{2}+\chi_{\{v>0\}}Q^{2}%
\big)\,d\boldsymbol{x},
\]
introduced in a classical paper of Alt and Caffarelli. For a special choice of
$Q$ this includes water waves. The second variation is obtained by computing the second derivative of the functional along suitable variations of the free boundary. It is proved that the strict positivity of the second variation gives a sufficient condition for local minimality. Also, it is shown that smooth critical points are local minimizers in a small tubular neighborhood of the free-boundary. 

\

\end{abstract}

\bigskip

\section{Introduction}

The goal of this paper is to derive a new minimality condition in terms of the
second variation for the functional
\begin{equation}
\mathcal{F}(v):=\int_{\Omega}\big(|\nabla v|^{2}+\chi_{\{v>0\}}Q^{2}%
\big)\,d\boldsymbol{x},\quad v\in\mathcal{A}_{0},\label{functional F}%
\end{equation}
introduced by Alt and Caffarelli in the seminal paper \cite{alt-caffarelli81}
(see also \cite{ACF1}, \cite{ACF2}, \cite{ACF3}, \cite{caffarelli-salsa-book05},
\cite{Friedman book}). Here, $\Omega\subset\mathbb{R}^{N}$ is an open
connected set with locally Lipschitz boundary, the function $Q:\Omega
\rightarrow[0,+\infty)$ is continuous, and
\[
\mathcal{A}_{0}:=\left\{  v\in H_{\operatorname*{loc}}^{1}(\Omega):\,\nabla
v\in L^{2}(\Omega;\mathbb{R}^{N}),~v=v^{\ast}\text{ on }S\right\}  ,
\]
where $S\subset\partial\Omega$ is a measurable set with $\mathcal{H}%
^{n-1}(  S)  >0$, and the Dirichlet datum $v^{\ast}\in
H_{\operatorname*{loc}}^{1}(\Omega)$ is a nonnegative function with $\nabla
v^{\ast}\in L^{2}(\Omega;\mathbb{R}^{N})$. The identity $v=v^{\ast}$ on $S$ is
to be understood in the sense of traces.

In this paper a critical point for \eqref{functional F} is a function $v\in\mathcal{A}_{0}$ such that $\mathcal{F}(v)\in\mathbb{R}$ and 
\begin{equation}\label{critical}
\frac{d \mathcal{F}}{d\varepsilon}(v+\varepsilon\varphi)\Big{\vert}_{\varepsilon=0}=0
\end{equation}
for every $\varphi\in H^1(\Omega)$ with $\varphi=0$ on $S$ in the sense of traces. It can be shown that when $v$ is a smooth critical point, e.g., $v\in C^2(\overline{\Omega})$, and  the free boundary $\Omega\cap\partial\left\{  v>0\right\}  $ is a manifold of class $C^2$, then we are led to a free boundary problem (see \cite{KinderS}). To be precise, 
the Euler-Lagrange equations of (\ref{functional F}) are given by 
\begin{equation}%
\begin{cases}
\Delta v=0 & \text{in }\Omega\cap\left\{  v>0\right\}  ,\\
v=0,\quad\left\vert \nabla v\right\vert =Q & \text{on }\Omega\cap
\partial\left\{  v>0\right\}  ,\\
v=v^{\ast} & \text{on }S,
\end{cases}
\label{E-L}%
\end{equation}
(see the Appendix).

Under the assumptions that $Q$ is H\"{o}lder continuous and%
\begin{equation}
0<Q_{\min}\leq Q(  \mathbf{x})  \leq Q_{\max}<\infty,
\label{Qbounds}%
\end{equation}
Alt and Caffarelli \cite{alt-caffarelli81} proved existence of global
minimizers, full regularity of the free boundary $\Omega\cap\partial\left\{
v>0\right\}  $ of local minimizers for $N=2$ and partial regularity for
$N\geq3$. Using a monotonicity formula, Weiss in \cite{weiss 99} improved the
estimate of the Hausdorff dimension of the singular set, and Caffarelli,
Jerison, and Kenig \cite{CJK} showed full regularity in dimension $N=3$. Note,
however, that in dimension $N=3$ there exist critical points of
(\ref{functional F}) whose free boundary is singular (see
\cite{alt-caffarelli81} and \cite{CJK}).

In this work we prove that in dimension $N=2$ and under the assumption
\eqref{Qbounds}, smooth critical points of (\ref{functional F}) are actually
local minimizers with respect to small $C^{2,\alpha}$ perturbations (see the statement of Theorem \ref{thm stability} for the precise notion of minimality) in a tubular neighborhood of $\partial\{u>0\}\cap\Omega$. The
proof is based on the derivation of a second order variation of the
functional (\ref{functional F}). 

This approach has been successfully applied to several contexts. In particular, in the study of the Mumford-Shah functional the strict positivity of the second variation has been used to obtain local minimality of critical points (see \cite{BM}, \cite{CMM}, \cite{C}, \cite{MM}), including triple junctions, which are at the core of the Mumford-Shah conjecture.
Furthermore, using the diffuse-interface Ohta-Kawasaki energy to model microphase separation in diblock copolymer melts, critical configurations with positive second variation were found to be local minimizers in \cite{AFM} (see also \cite{BC1}, \cite{BC2}). In turn, these results are used to determine global and local
minimality of certain lamellar configurations.
Finally, in  \cite{FM} (see also \cite{B}) the authors analyzed a variational model for epitaxial growth of a thin elastic film over a flat substrate when
a lattice mismatch between the two materials is present. Again using techniques involving the positivity of the second variation, they determined the
critical threshold for local and global minimality of the flat configuration. 

We now present the main results of this paper. We assume $\Omega
:=(-1,1)\times(0,\infty)$, and we consider the functional $\mathcal{F}$ in
(\ref{functional F}) defined on the class%
\begin{align}
\mathcal{A}:=\big\{v\in L_{\operatorname*{loc}}^{1}(\Omega):\ \nabla v\in
L^{2}(\Omega;\mathbb{R}^{2}),\ v(x,0)=u^{\ast}(x)\text{ } &  \text{for }%
x\in(-1,1),\nonumber\\
v(-1,y) &  =v(1,y)\text{ for }y>0\big\},\label{class admissible}%
\end{align}
where $\mathbf{x}=(x,y)
\in\mathbb{R}^{2}$, $u^{\ast}\in C^{1}([-1,1])$ is periodic, and $u^{\ast}>0$, while the function $Q$ satisfies
\begin{equation}
Q\in C^{0,1}(\Omega),\quad 0\leq Q(\boldsymbol{x})\leq Q_{\max}\quad\text{for every
}\boldsymbol{x}\in\Omega.\label{function Q}%
\end{equation}
By Theorem~1.3 in \cite{alt-caffarelli81} (see also \cite{AL}), there exists a minimizer of
$\mathcal{F}$ in $\mathcal{A}$. Moreover, in view of Lemma~2.4 in
\cite{alt-caffarelli81}, for any local minimizer $v$ of $\mathcal{F}$ in
$\mathcal{A}$, the set $\{v>0\}\cap\Omega$ is open and $v$ is harmonic in
$\{v>0\}\cap\Omega$. Let $u\in\mathcal{A}$ be such that the set
\begin{equation}
\Omega_{+}:=\{u>0\}\cap\Omega\label{Omega plus}%
\end{equation}
is open, $u$ is harmonic in $\Omega_{+}$, and
\begin{equation}
\Gamma:=\partial\{u>0\}\cap\Omega\label{Gamma}%
\end{equation}
is a smooth curve. Then $u$ satisfies the elliptic problem%
\begin{equation}%
\begin{cases}
\Delta u=0 & \text{in }\Omega_{+},\\
u=0 & \text{on }\Gamma,\\
u=u^{\ast} & \text{on }\{y=0\}\cap\partial\Omega_{+},
\end{cases}
\label{equation u}%
\end{equation}
together with the periodicity conditions
\begin{equation}
u(-1,y)=u(1,y)\quad\text{for }y>0.\label{periodicity u}%
\end{equation}
We consider a one-parameter family of diffeomorphisms $\{\Phi_{s}%
\}_{s\in\lbrack0,1]}$ that coincide with the identity in a uniform
neighborhood of $\partial\Omega$. We then derive the second derivative of
$\mathcal{F}(u_{s})$ with respect to $s$, where $u_{s}$ is the minimizer of
the Dirichlet energy on $\Phi_{s}(\{u>0\})$ with respect to the given boundary
conditions. Imposing the first derivative of $\mathcal{F}(u_{s})$ to be zero
at $s=0$ gives back the equilibrium condition $\left\vert \nabla u\right\vert
=Q$ on $\Omega\cap\partial\left\{  u>0\right\}  $. The second order derivative
of $\mathcal{F}(u_{s})$ provides a new necessary condition
for minimizers, expressed in terms of a sign condition for a quadratic form
(see Remark~\ref{remark minimality} below). In turn, the strict positivity of
this quadratic form gives a sufficient condition for local minimality. This is
made precise by the following theorem, which is one of the main results of
this paper. 

In what follows, we denote by $\nu:\Gamma\rightarrow\mathbb{S}^{1}$ a smooth normal vector
to $\Gamma$. The curvature $\kappa$ of $\Gamma$ satisfies $\partial_{\tau}%
\nu=\kappa\tau$ and $\partial_{\tau}\tau=-\kappa\nu$, where $\tau
:\Gamma\rightarrow\mathbb{S}^{1}$ is a smooth tangent vector to~$\Gamma$. 

\begin{theorem}
\label{theorem main} 
Assume that $Q\in C^{1,1}(\Omega)$ and satisfies (\ref{function Q}).
Let $u\in C^{2,\alpha}(\Omega_{+}\cup\Gamma)$, $\alpha
>0$, satisfy (\ref{Omega plus}), (\ref{equation u}), (\ref{periodicity u}), and
let the free boundary $\Gamma$ given in (\ref{Gamma}) be the graph of a
$C^{3}$ periodic function. Assume, in addition,
that
\begin{equation}
(\partial_{\nu}u)^{2}=Q^{2}\quad\text{on }\Gamma\label{u stat}%
\end{equation}
and that there exists $C_{0}>0$ such that
\begin{equation}
\int_{\Omega_{+}}2|\nabla u_{\psi}|^{2}\,d\boldsymbol{x}+\int_{\Gamma
}(\partial_{\nu}Q^{2}+2\kappa Q^{2})\psi^{2}\,d\mathcal{H}^{1}\geq
C_{0}\left\Vert \psi\right\Vert _{H^{1/2}(\Gamma)}^{2}\label{coercivity}%
\end{equation}
for every $\psi\in C_{c}^{1}(\Gamma)$, where $u_{\psi}$ is the solution to
\[%
\begin{cases}
\Delta u_{\psi}=0 & \text{in }\Omega_{+},\\
u_{\psi}=Q\psi & \text{on }\Gamma,\\
u_{\psi}=0 & \text{on }\{y=0\}\cap\partial\Omega_{+},%
\end{cases}
\]
with $u_{\psi}(-1,y)=u_{\psi}(1,y)$ for all $y$ such that $(\pm1,y)\in
\overline{\Omega_{+}}$. Then there exists $\delta>0$  such that for every open set
$U\Subset\Omega$ and for every diffeomorphism $\Phi\in C^{2,\alpha
}(\mathbb{R}^{2};\mathbb{R}^{2})$ with
\begin{equation}
\operatorname*{supp}(\Phi-\operatorname*{Id})\subset U\label{m1}%
\end{equation}
and
\begin{equation}
\left\Vert \Phi-\operatorname*{Id}\right\Vert _{C^{2,\alpha}(\mathbb{R}^{2}%
)}\leq\delta,\label{m2}%
\end{equation}
we have%
\[
\mathcal{F}(u)\leq\mathcal{F}(v)
\]
for every $v\in\mathcal{A}$ with $\{v>0\}=\Phi(\{u>0\})$.
\end{theorem}

Although the notion of $C^{2,\alpha}$-minimality established in the previous theorem may be perceived as weak,
it has been shown to lead to a stronger minimality property in several of the contexts mentioned above. To be precise, in the case of epitaxial growth Fusco and Morini \cite{FM} proved that the strict positivity
of the second variation implies local minimality with respect to $W^{2,\infty}$ perturbations and in turn, that this leads to
local minimality with respect to $L^\infty$ perturbations. Similarly, for the diffuse-interface Ohta-Kawasaki energy it is shown in \cite{AFM} that
the strict positivity of the second variation yields local minimality with respect to $W^{2,p}$ perturbations and that $W^{2,p}$-local minimizers
are actually $L^1$-local minimizers. Thus, it is natural to expect that in our setting $C^{2,\alpha}$-minimizers are in fact local minimizers in a much larger
class of competitors. This will be addressed in a forthcoming paper.

We also observe that a different type of second variation for the functional
(\ref{functional F}) has been used by Caffarelli, Jerison, and Kenig in
\cite{CJK} to prove full regularity of global minimizers when $N=3$, and by
Weiss and Zhang in \cite{wz} for a similar functional related to water waves with vorticity. In contrast to
our case, where we perform variations of the free boundary~$\Gamma$, in
\cite{CJK} and \cite{wz} the variations are of the type $u+\varepsilon v$,
where $v$ is harmonic in $\Omega_{+}\cap B$ with boundary datum a given
function $g$ on $\partial(\Omega_{+}\cap B)$ and $B$ is a ball.

In the second main theorem we prove that, if $u$ is a smooth critical point of
$\mathcal{F}$ restricted to $\mathcal{A}$, then $u$ satisfies the minimality
property of Theorem~\ref{theorem main} in a tubular neighborhood of $\Gamma$.

\begin{theorem}
\label{thm stability} Assume that $Q$ satisfies (\ref{function Q}) and $Q\geq Q_{\mathrm{min}}>0$. Let $u\in C^{2,\alpha}(\Omega_{+})$ be as in
(\ref{Omega plus})--(\ref{periodicity u}), and let $\Gamma$ be the graph of a
$C^{3}$ periodic function. Assume, in addition, that
\[
(\partial_{\nu}u)^{2}=Q^{2}\quad\text{on }\Gamma.
\]
Then there exist $\varepsilon>0$ and $c_{\varepsilon}>0$ such that
\begin{equation}
\int_{U_{\varepsilon}\cap\{u>0\}}2|\nabla u_{\psi}|^{2}\,d\boldsymbol{x}%
+\int_{\Gamma}(\partial_{\nu}Q^{2}+2\kappa Q^{2})\psi^{2}\,d\mathcal{H}%
^{1}\geq c_{\varepsilon}\Vert\psi\Vert_{H^{1/2}(\Gamma)}^{2}\label{coerc Ue}%
\end{equation}
for every $\psi\in C_{c}^{1}(\Gamma)$, where $U_{\varepsilon}$ is the
intersection of $\Omega$ with the $\varepsilon$-tubular neighborhood of
$\Gamma$. In particular, if $Q\in C^{1,1}(\Omega)$ then there exists $\delta
_{\varepsilon}>0$ such that for every open set $V_{\varepsilon}\Subset U_{\varepsilon}$
and for every diffeomorphism $\Phi\in C^{2,\alpha}(\mathbb{R}%
^{2};\mathbb{R}^{2})$ with
\[
\operatorname*{supp}(\Phi-\operatorname*{Id})\subset V_{\varepsilon}%
\quad\text{and}\quad\left\Vert \Phi-\operatorname*{Id}\right\Vert
_{C^{2,\alpha}(\mathbb{R}^{2})}\leq\delta_{\varepsilon},
\]
we have%
\[
\int_{U_{\varepsilon}}\big(|\nabla u|^{2}+\chi_{\{u>0\}}Q^{2}%
\big)\,d\boldsymbol{x}\leq\int_{U_{\varepsilon}}\big(|\nabla v|^{2}%
+\chi_{\{v>0\}}Q^{2}\big)\,d\boldsymbol{x}%
\]
for every $v\in L_{\operatorname*{loc}}^{1}(U_{\varepsilon})$ such that
$\nabla v\in L^{2}(U_{\varepsilon};\mathbb{R}^{2})$, $v=u$ on $\partial
U_{\varepsilon}\cap\Omega_{+}$, $v(-1,y)=v(1,y)$ for all $y$ such that
$(\pm1,y)\in\partial U_{\varepsilon}$, and $\{v>0\}=\Phi(\{u>0\})$.
\end{theorem}

The constant $c_{\varepsilon}$ in \eqref{coerc Ue} depends strongly on $Q_{\min}$. This is not surprising, since
the hypothesis $Q_{\min}>0$ is fundamental for the
regularity of local minimizers. When $Q_{\min}=0$ one expects the free
boundary to present singularities at points where $Q(\mathbf{x})
=0$. Indeed, in dimension $N=2$ and when $Q(  x,y)  =\sqrt{(
q-2gy)  _{+}}$, where $q$ is a physical constant related to the hydraulic head and $g$ is the gravitational
acceleration, the free boundary problem (\ref{E-L}) is related to Stokes
waves of greatest height, which are characterized by the
fact that their shape is not regular but has a sharp crest of included angle
$\frac{2}{3}\pi$ (see, e.g., \cite{AL}, \cite{CSS}, \cite{CoStrauss}, \cite{milne-thomson1960book}, \cite{plotnikov-toland05}, \cite{ST}, \cite{stokes1847},
\cite{stokes1880}, \cite{toland14}, \cite{V}, \cite{VarWeiss}, \cite{VW}, \cite{weiss 99}, \cite{wz}, and the references therein). 

This paper is organized as follows. In Section~\ref{section second variation}
we give the precise definition of admissible flows and derive the second
variation of the functional (\ref{functional F}).
In Section~\ref{section diff}, given a small perturbation of $\Gamma$, we
construct an admissible flow (see Definition~\ref{definition admissible flow})
joining $\Gamma$ to the perturbed free boundary and with zero tangential
velocity on the free boundaries. This latter property will play a crucial role
in the proofs of the main theorems, and it leads to a first order partial
differential equation (see (\ref{i equiv}) below) that we solve using the
method of characteristics. One of the main difficulties is that the components
of the flow are given by compositions of functions that are discontinuous.
Thus, proving the regularity of the flow is extremely delicate and it will be
carried out in the appendix. 
In Section~\ref{section proof main} we prove Theorem~\ref{theorem main}. To
control the second variation along the flow we use sharp Schauder estimates
together with the zero tangential velocity of the flow.
Finally, in Section~\ref{section proof stability} we prove Theorem~\ref{thm stability}.

\section{The Second Variation}

\label{section second variation}

In this section we derive the second variation of $\mathcal{F}$ on some
suitable variations of $u$ that are constructed along a family of variations
of $\Gamma$ according to the following definition.

\begin{definition}
\label{definition admissible flow} We say that $\{\Phi_{s}\}_{s\in[0,1]}$ is
an \emph{admissible flow} if it satisfies the following conditions:

\begin{itemize}
\item[(i)] the map $(s,\boldsymbol{x})\mapsto\Phi_{s}(\boldsymbol{x})$ belongs
to $C^{2}([0,1]\times\overline{\Omega};\mathbb{R}^{2})$;

\item[(ii)] for every $s\in[0,1]$, the map $\Phi_{s}$ is a diffeomorphism from
$\overline{\Omega}$ onto itself;

\item[(iii)] $\Phi_{0}=\operatorname*{Id}$ in $\overline{\Omega}$;

\item[(iv)] there exists an open set $U$, compactly contained in $\Omega$,
such that $\operatorname*{supp}(\Phi_{s}-\operatorname*{Id})\subset U$ for all
$s\in[0,1]$.
\end{itemize}
\end{definition}

Let $\{\Phi_{s}\}_{s\in\lbrack0,1]}$ be an admissible flow, and let $u$ be as
in Theorem~\ref{theorem main}. For every $s\in\lbrack0,1]$ we consider the
solution $u_{s}$ of the problem%
\begin{equation}%
\begin{cases}
\Delta u_{s}=0 & \text{in }\Phi_{s}(\Omega_{+}),\\
u_{s}=0 & \text{on }\Phi_{s}(\Gamma),\\
u_{s}=u & \text{on }\partial\Phi_{s}(\{y=0\}\cap\partial\Omega_{+}),
\end{cases}
\label{equation u t}%
\end{equation}
with $u_{s}(-1,y)=u_{s}(1,y)$ for all $y$ such that $(\pm1,y)\in\Phi
_{s}(\overline{\Omega_{+}})$. Note that, in view of property (iv) in
Definition~\ref{definition admissible flow}, we have that
\[
\partial\Phi_{s}(\{y=0\}\cap\partial\Omega_{+})=\{y=0\}\cap\partial\Omega_{+}%
\]
and $(\pm1,y)\in\Phi_{s}(\overline{\Omega_{+}})$ if and only if $(\pm
1,y)\in\overline{\Omega_{+}}$. Moreover, extending $u_{s}$ by $0$ outside
$\Phi_{s}(\Omega_{+})$, we obtain $u_{s}\in\mathcal{A}$.

In what follows, for every $s\in\lbrack0,1]$ and $\boldsymbol{x}\in\Omega$\ we
denote by $\dot{u}_{s}(\boldsymbol{x})$ the partial derivative with respect to
$r$ of the function $(r,\boldsymbol{x})\mapsto u_{r}(\boldsymbol{x})$
evaluated at $(s,\boldsymbol{x})$, that is,
\begin{equation}
\dot{u}_{s}(\boldsymbol{x}):=\frac{\partial u_{r}}{\partial r}(\boldsymbol{x}%
)\Big|_{r=s}. \label{u dot}%
\end{equation}
We define
\begin{equation}
X_{s}:=\dot{\Phi}_{s}\circ\Phi_{s}^{-1},\quad Z_{s}:=\ddot{\Phi}_{s}\circ
\Phi_{s}^{-1} \label{X and Z}%
\end{equation}
for every $s\in\lbrack0,1]$, where%
\begin{equation}
\dot{\Phi}_{s}:=\frac{\partial\Phi_{r}}{\partial r}\Big|_{r=s},\quad\ddot
{\Phi}_{s}:=\frac{\partial^{2}\Phi_{r}}{\partial r^{2}}\Big|_{r=s}.
\label{Phi dot}%
\end{equation}
Moreover, we set $\Gamma_{s}:=\Phi_{s}(\Gamma)$ and denote by $\tau_{s}$ and
$\nu_{s}$ the tangent and normal vector to $\Gamma_{s}$ given by
\begin{equation}
\tau_{s}:=\frac{(D\Phi_{s})\tau}{|(D\Phi_{s})\tau|}\circ\Phi_{s}^{-1},\quad
\nu_{s}:=\frac{(D\Phi_{s})^{-T}\nu}{|(D\Phi_{s})^{-T}\nu|}\circ\Phi_{s}^{-1}.
\label{tangent and normal}%
\end{equation}
Finally, $\kappa_{s}$ denotes the curvature of $\Gamma_{s}$.

The proof of the following proposition follows {the arguments in \cite{CMM}.}

\begin{proposition}
\label{proposition regularity u hat s}Let $u\in C^{2}(\Omega_{+}\cup\Gamma)$
satisfy (\ref{Omega plus})--(\ref{periodicity u}), let $\{\Phi_{s}%
\}_{s\in\lbrack0,1]}$ be an admissible flow, and let $\hat{u}_{s}:=u_{s}%
\circ\Phi_{s}$, where $u_{s}$ solves (\ref{equation u t}). Then the map
\[
s\mapsto\hat{u}_{s}%
\]
belongs to $C^{1}([0,1];H^{1}(\Omega_{+}))$. In particular, the function
$\dot{u}_{s}$ in (\ref{u dot}) is well-defined and is the unique solution to
the boundary value problem
\begin{equation}%
\begin{cases}
\Delta\dot{u}_{s}=0 & \text{in }\Phi_{s}(\Omega_{+}),\\
\dot{u}_{s}=-(X_{s}\cdot\nu_{s})\,\partial_{\nu_{s}}u_{s} & \text{on }%
\Gamma_{s},\\
\dot{u}_{s}=0 & \text{on }\partial\Phi_{s}(\{y=0\}\cap\partial\Omega_{+}),
\end{cases}
\label{udot eq}%
\end{equation}
with $\dot{u}_{s}(-1,y)=\dot{u}_{s}(1,y)$ for all $y$ such that $(\pm
1,y)\in\Phi_{s}(\overline{\Omega_{+}})$.
\end{proposition}

\begin{proof}
For simplicity, we only prove the result in a neighborhood of $s=0$. The
general case can be obtained analogously. In view of (\ref{equation u t}) a
straightforward computation shows that $\hat{u}_{s}$ satisfies
\begin{equation}%
\begin{cases}
\operatorname{div}(A_{s}\nabla\hat{u}_{s})=0 & \text{in }\Omega_{+},\\
\hat{u}_{s}=u & \text{on }\Gamma\cup(\{y=0\}\cap\partial\Omega_{+}),
\end{cases}
\label{s3}%
\end{equation}
with $\hat{u}_{s}(-1,y)=\hat{u}_{s}(1,y)$ for all $(\pm1,y)\in\overline
{\Omega_{+}}$, where
\begin{equation}
A_{s}:=\left(  \frac{D\Phi_{s}^{-1}(D\Phi_{s}^{-1})^{T}}{\det D\Phi_{s}^{-1}%
}\right)  \circ\Phi_{s}.\label{As}%
\end{equation}
Let $V$ be the subspace of all functions $v\in H^{1}(\Omega_{+})$ such that
$v=0$ on $\Gamma\cup(\{y=0\}\cap\partial\Omega_{+})$ and $v(-1,y)=v(1,y)$ for
all $(\pm1,y)\in\overline{\Omega_{+}}$. For every $s\in\lbrack0,1]$ and $v\in
V$ let $\mathcal{H}(s,v)$ be the unique weak solution $w\in V$ of the
Poisson's equation%
\[
\Delta w=\operatorname{div}(A_{s}\nabla(v+u))\quad\text{in }\Omega_{+}.
\]
Then
\[
\mathcal{H}:[0,1]\times V\rightarrow V.
\]
Observe that $A_{0}=I_{2\times2}$ and $u$ is
harmonic in $\Omega_{+}$ by (\ref{equation u}), hence $\mathcal{H}(0,0)=0$. Moreover, (\ref{equation u}) implies that $\mathcal{H}%
(0,v)=v$, thus $\partial_{v}%
\mathcal{H}(0,0)$ is the identity operator from $V$ into $V$. Since the matrix
$A_{s}$ in (\ref{As}) is of class $C^{1}$, by standard elliptic estimates (see
also the proof of Proposition~\ref{proposition estimates us} below), we have
that the map $\mathcal{H}$ is of class $C^{1}$. Hence, we are in a position to
apply the implicit function theorem (see, e.g., \cite[Theorem~2.3]%
{ambrosetti-prodi}) to find $\delta_{0}>0$ and $r_{0}>0$ and a unique
continuous function $g:[0,\delta_{0}]\rightarrow B_{V}(0,r_{0})$ such that
$g(0)=0$ and
\[
\mathcal{H}(s,g(s))=0
\]
for all $s\in\lbrack0,\delta_{0}]$. Moreover, $g$ is of class $C^{1}$.

On the other hand, in view of (\ref{s3}) the function $\hat{u}_{s}-u$ belongs
to $V$ and satisfies
\[
\mathcal{H}(s,\hat{u}_{s}-u)=0
\]
for all $s\in[0,1]$. Since the map $s\mapsto\hat{u}_{s}-u$ is continuous (see,
e.g., the proof of (\ref{s6}) below), it follows by uniqueness that
$g(s)=\hat{u}_{s}-u$ for all $s\in[0,\delta_{0}]$. In particular,
$s\mapsto\hat{u}_{s}-u$ is of class $C^{1}$.

To prove (\ref{udot eq}), let $v\in H_{\operatorname*{loc}}^{1}(\Omega)$ be
such that $\nabla v\in L^{2}(\Omega;\mathbb{R}^{2})$, $v(-1,y)=v(1,y)$ for
$y\in(0,\infty)$, $v(x,0)=0$ for $x\in(-1,1)$, and $\overline{\Gamma}%
\cap\operatorname*{supp}v=\emptyset$. Then $\overline{\Gamma}_{s}%
\cap\operatorname*{supp}v=\emptyset$ for all $s$ sufficiently small. By
(\ref{equation u t}) it follows that there exists an open subset of $\Omega$
containing $\operatorname*{supp}v$ and on which $u_{s}$ is harmonic for all
$s$ sufficiently small; thus,
\[
\int_{\Omega}\nabla u_{s}\cdot\nabla v~d\boldsymbol{x}=0
\]
for all $s$ sufficiently small. Differentiating the previous identity with
respect to $s$ (see (\ref{u dot})) we obtain
\begin{equation}
\int_{\Omega}\nabla\dot{u}_{s}\cdot\nabla v~d\boldsymbol{x}=0. \label{101}%
\end{equation}
By (\ref{equation u t}) we have that $u_{s}(\Phi_{s}(\boldsymbol{x}))=0$ for
$\boldsymbol{x}\in\Gamma$; thus,
\[
\dot{u}_{s}(\Phi_{s}(\boldsymbol{x}))=-\nabla u_{s}(\Phi_{s}(\boldsymbol{x}%
))\cdot\dot{\Phi}_{s}(\boldsymbol{x})\quad\text{ for }\boldsymbol{x}\in
\Gamma,
\]
which by (\ref{X and Z}) is equivalent to
\begin{equation}
\dot{u}_{s}=-\nabla u_{s}\cdot X_{s}\quad\text{ on }\Gamma_{s}.
\label{udot on G}%
\end{equation}
On the other hand, since $\nabla u_{s}=\partial_{\nu_{s}}u_{s}\,\nu_{s}$ on
$\Gamma_{s}$ by (\ref{equation u t}), we have that $\nabla u_{s}\cdot
X_{s}=(X_{s}\cdot\nu_{s})\,\partial_{\nu_{s}}u_{s}$ on $\Gamma_{s}$. In
conclusion,
\begin{equation}
\dot{u}_{s}=-(X_{s}\cdot\nu_{s})\,\partial_{\nu_{s}}u_{s}\quad\text{ on
}\Gamma_{s}. \label{102}%
\end{equation}

Let now $v\in H_{\operatorname*{loc}}^{1}(\Omega)$ be such that $\nabla v\in
L^{2}(\Omega;\mathbb{R}^{2})$, $v(-1,y)=v(1,y)$ for $y\in(0,\infty)$,
$v(x,0)=0$ for $x\in(-1,1)$, and decompose $v=v_{1}+v_{2}$, where
$\overline{\Gamma}\cap\operatorname*{supp}v_{1}=\emptyset$. Then by
(\ref{101}) and (\ref{102}), integrating by parts we get
\begin{align*}
\int_{\Omega}\nabla\dot{u}_{s}\cdot\nabla v~d\boldsymbol{x}  &  =\int_{\Omega
}\nabla\dot{u}_{s}\cdot\nabla v_{2}~d\boldsymbol{x}=\int_{\Gamma_{s}}%
(X_{s}\cdot\nu_{s})\,\partial_{\nu_{s}}u_{s}\,\partial_{\nu_{s}}%
v_{2}\,d\mathcal{H}^{1}\\
&  =\int_{\Gamma_{s}}(X_{s}\cdot\nu_{s})\,\partial_{\nu_{s}}u_{s}%
\,\partial_{\nu_{s}}v\,d\mathcal{H}^{1}%
\end{align*}
for all $s$ sufficiently small. This proves that the function $\dot{u}_{s}$ is
a solution to (\ref{udot eq}).
\end{proof}

In view of Proposition~\ref{proposition regularity u hat s} we can now derive
the second derivative of $\mathcal{F}(u_{s})$.

\begin{theorem}
\label{thm:vars}Let $u\in C^{2}(\Omega_{+}\cup\Gamma)$ satisfy
(\ref{Omega plus})--(\ref{periodicity u}), let $Q$ satisfy (\ref{function Q}), and let $\{\Phi_{s}\}_{s\in\lbrack
0,1]}$ be an admissible flow. Then
\begin{equation}
\frac{d}{ds}\mathcal{F}(u_{s})=\int_{\Gamma_{s}}(Q^{2}-|\nabla u_{s}%
|^{2})(X_{s}\cdot\nu_{s})\,d\mathcal{H}^{1} \label{1var}%
\end{equation}
and
\begin{align}
\frac{d^{2}}{ds^{2}}\mathcal{F}(u_{s})  &  =\int_{\Phi_{s}(\Omega_{+}%
)}2|\nabla\dot{u}_{s}|^{2}\,d\boldsymbol{x}+\int_{\Gamma_{s}}\big(\partial
_{\nu_{s}}Q^{2}+2\kappa_{s}(\partial_{\nu_{s}}u_{s})^{2}\big)(X_{s}\cdot
\nu_{s})^{2}\,d\mathcal{H}^{1}\nonumber\\
&  \quad+\int_{\Gamma_{s}}(Q^{2}-|\nabla u_{s}|^{2})\big(Z_{s}\cdot\nu
_{s}-2(X_{s}\cdot\tau_{s})\,\partial_{\tau_{s}}(X_{s}\cdot\nu_{s})+\kappa
_{s}|X_{s}|^{2}\big)\,d\mathcal{H}^{1}, \label{2var}%
\end{align}
where $\dot{u}_{s}$ is given in (\ref{udot eq}).
\end{theorem}

\begin{remark}
\label{remark minimality}If $u$ is a minimizer of $\mathcal{F}$, then the
expression in (\ref{1var}) is equal to zero at $s=0$; since this is true for
any choice of the admissible flow, this implies
\begin{equation}
|\nabla u|^{2}=Q^{2}\quad\text{on }\Gamma.\label{EL}%
\end{equation}
In turn, the second variation at $s=0$ reduces to
\begin{equation}
\frac{d^{2}}{ds^{2}}\mathcal{F}(u_{s})\Big|_{s=0}=\int_{\Omega_{+}}2|\nabla
u_{X_{0}\cdot\nu}|^{2}\,d\boldsymbol{x}+\int_{\Gamma}\big(\partial_{\nu}%
Q^{2}+2\kappa Q^{2}\big)(X_{0}\cdot\nu)^{2}\,d\mathcal{H}^{1}%
,\label{2var simple}%
\end{equation}
where $u_{X_{0}\cdot\nu}$ is the solution to
\begin{equation}%
\begin{cases}
\Delta u_{X_{0}\cdot\nu}=0 & \text{in }\Omega_{+},\\
u_{X_{0}\cdot\nu}=Q(X_{0}\cdot\nu) & \text{on }\Gamma,\\
u_{X_{0}\cdot\nu}=0 & \text{on }\{y=0\}\cap\partial\Omega_{+},%
\end{cases}
\label{100}%
\end{equation}
with $u_{X_{0}\cdot\nu}(-1,y)=u_{X_{0}\cdot\nu}(1,y)$ for all $y$ such that
$(\pm1,y)\in\overline{\Omega_{+}}$. Indeed, on $\Gamma$ we have $(\partial
_{\nu}u)^{2}=|\nabla u|^{2}=Q^{2}$ and $\partial_{\nu}u<0$ by the Hopf Lemma.
Moreover, the expression in (\ref{2var simple}) is nonnegative. 

Note that every minimizer $u$ satisfies the necessary condition
\begin{equation}\label{mg101}
\int_{\Omega_{+}}2|\nabla u_{\psi}|^{2}\,d\boldsymbol{x}+\int_{\Gamma
}\big(\partial_{\nu}Q^{2}+2\kappa Q^{2}\big)\psi^{2}\,d\mathcal{H}^{1}\geq0
\end{equation}
for every $\psi\in C_{c}^{2}(\Gamma)$, where
$u_{\psi}$ solves (\ref{100}) with $\psi$ in place of $X_{0}\cdot\nu$.
In fact, for every $\psi\in C_{c}^{2}(\Gamma)$ with small $C^{2}$ norm it is
possible to construct an admissible flow $\{\Phi_{s}\}_{s\in\lbrack0,1]}$ such
that $X_{0}\cdot\nu=\psi$ on $\Gamma$. To see this, it is enough to consider
\[
\Phi_{s}(x,y):=(x,y)+\lambda(y)s\psi(x,w(x))\nu(x,y),
\]
where the normal $\nu$ to $\Gamma$ has been extended smoothly, and $\lambda$ is
a cut-off function (see (\ref{phi s}) below for more details). Hence, from
(\ref{2var simple}) we deduce that (\ref{mg101}) holds for every $\psi\in C_{c}^{2}(\Gamma)$ with small $C^{2}$ norm. In turn, given an arbitrary $\psi\in C_{c}^{2}(\Gamma)$, using a scaling argument, it can be shown that (\ref{mg101})
continues to hold.
\end{remark}

\begin{remark}
Observe that if $u$ is a critical point of $\mathcal{F}$, that is, $u$
satisfies (\ref{EL}) in addition to (\ref{Omega plus})--(\ref{periodicity u}),
then (\ref{2var simple}) holds.
\end{remark}

\begin{proof}
[Proof of Theorem~\ref{thm:vars}]In view of (\ref{equation u t}), we have that
$u_{s}>0$ in $\Phi_{s}(\Omega_{+})$; thus,
\begin{align*}
\mathcal{F}(u_{s})  &  =\int_{\Phi_{s}(\Omega_{+})}\big(|\nabla u_{s}%
|^{2}+Q^{2}(\boldsymbol{x})\big)\,d\boldsymbol{x}\\
&  =\int_{\Omega_{+}}\big(|\nabla u_{s}(\Phi_{s}(\boldsymbol{y}))|^{2}%
+Q^{2}(\Phi_{s}(\boldsymbol{y}))\big)\det D\Phi_{s}(\boldsymbol{y}%
)\,d\boldsymbol{y}.
\end{align*}
Differentiating the previous identity with respect to $s$ we obtain
\begin{align*}
\frac{d}{ds}\mathcal{F}(u_{s})  &  =\int_{\Omega_{+}}2\nabla u_{s}(\Phi
_{s}(\boldsymbol{y}))\cdot\big(\nabla\dot{u}_{s}(\Phi_{s}(\boldsymbol{y}%
))+D^{2}u_{s}(\Phi_{s}(\boldsymbol{y}))\dot{\Phi}_{s}(\boldsymbol{y})\big)\det
D\Phi_{s}(\boldsymbol{y})\,d\boldsymbol{y}\\
&  \quad+\int_{\Omega_{+}}\nabla Q^{2}(\Phi_{s}(\boldsymbol{y}))\cdot\dot
{\Phi}_{s}(\boldsymbol{y})\det D\Phi_{s}(\boldsymbol{y})\,d\boldsymbol{y}\\
&  \quad+\int_{\Omega_{+}}\big(|\nabla u_{s}(\Phi_{s}(\boldsymbol{y}%
))|^{2}+Q^{2}(\Phi_{s}(\boldsymbol{y}))\big)\,\frac{d}{ds}(\det D\Phi
_{s}(\boldsymbol{y}))\,d\boldsymbol{y},
\end{align*}
where we used (\ref{u dot}) and (\ref{Phi dot}).
By \cite[Chapter~III, Section~10]{G} we have
\[
\frac{d}{ds}(\det D\Phi_{s})=\big[\operatorname{div}(\dot{\Phi}_{s}\circ
\Phi_{s}^{-1})\circ\Phi_{s}\big]\det D\Phi_{s},
\]
thus, recalling that $X_{s}=\dot{\Phi}_{s}\circ\Phi_{s}^{-1}$ (see
(\ref{X and Z})),
\begin{align*}
\frac{d}{ds}\mathcal{F}(u_{s})  &  =\int_{\Phi_{s}(\Omega_{+})}2\nabla
u_{s}\cdot\nabla\dot{u}_{s}\,d\boldsymbol{x}+\int_{\Phi_{s}(\Omega_{+})}%
2D^{2}u_{s}\nabla u_{s}\cdot X_{s}\,d\boldsymbol{x}\\
&  \quad+\int_{\Phi_{s}(\Omega_{+})}|\nabla u_{s}|^{2}\,\operatorname{div}%
X_{s}\,d\boldsymbol{x}+\int_{\Phi_{s}(\Omega_{+})}\big(\nabla Q^{2}\cdot
X_{s}+Q^{2}\operatorname{div}X_{s}\big)\,d\boldsymbol{x}.
\end{align*}
Integrating by parts, from (\ref{udot eq}) and the fact that
$\operatorname*{supp}(\Phi_{s}-\operatorname*{Id})\subset U$ for all
$s\in\lbrack0,1]$, we deduce that
\begin{align*}
\frac{d}{ds}\mathcal{F}(u_{s})  &  =\int_{\Gamma_{s}}2\dot{u}_{s}\partial
_{\nu_{s}}u_{s}\,d\mathcal{H}^{1}+\int_{\Phi_{s}(\Omega_{+})}%
\operatorname{div}\big((|\nabla u_{s}|^{2}+Q^{2})X_{s}\big)\,d\boldsymbol{x}\\
&  =\int_{\Gamma_{s}}\big(-2(X_{s}\cdot\nu_{s})(\partial_{\nu_{s}}u_{s}%
)^{2}+(|\nabla u_{s}|^{2}+Q^{2})(X_{s}\cdot\nu_{s})\big)\,d\mathcal{H}^{1}.
\end{align*}
Since $(\partial_{\nu_{s}}u_{s})^{2}=|\nabla u_{s}|^{2}$ on $\Gamma_{s}$, we
obtain (\ref{1var}).

We now derive the second derivative of $\mathcal{F}(u_{s})$ with respect to
$s$ at $s=0$. First, by the area formula we can write the first derivative as
\[
\frac{d}{ds}\mathcal{F}(u_{s})=\int_{\Gamma}\big(Q^{2}(\Phi_{s}(\boldsymbol{y}%
))-|\nabla u_{s}(\Phi_{s}(\boldsymbol{y}))|^{2}\big)\,\dot{\Phi}%
_{s}(\boldsymbol{y})\cdot\nu_{s}(\Phi_{s}(\boldsymbol{y}))\,J_{\Phi_{s}%
}(\boldsymbol{y})\,d\mathcal{H}^{1}(\boldsymbol{y}),
\]
where $J_{\Phi_{s}}:=|(D\Phi_{s})^{-T}\nu|\det D\Phi_{s}$ is the
one-dimensional Jacobian of $\Phi_{s}$. Differentiating with respect to $s$
yields
\begin{align}
\frac{d^{2}}{ds^{2}}\mathcal{F}(u_{s})\Big|_{s=0}  &  =\int_{\Gamma
}\big(\nabla Q^{2}\cdot\dot{\Phi}_{0}\big)\,\dot{\Phi}_{0}\cdot\nu
\,d\mathcal{H}^{1}\nonumber\\
&  \quad-\int_{\Gamma}2\big(\nabla u\cdot\nabla\dot{u}_{0}+D^{2}u\nabla
u\cdot\dot{\Phi}_{0}\big)\,\dot{\Phi}_{0}\cdot\nu\,d\mathcal{H}^{1}\nonumber\\
&  \quad+\int_{\Gamma}(Q^{2}-|\nabla u|^{2})\,\frac{d}{ds}\big[\dot{\Phi}%
_{s}\cdot(\nu_{s}\circ\Phi_{s})\,J_{\Phi_{s}}\big]\Big|_{s=0}\,d\mathcal{H}%
^{1}. \label{2der}%
\end{align}
The first integral in the above expression can be written as
\[
\int_{\Gamma}\big(\nabla Q^{2}\cdot\dot{\Phi}_{0}\big)\,\dot{\Phi}_{0}\cdot
\nu\,d\mathcal{H}^{1}=\int_{\Gamma}\big(\partial_{\tau}Q^{2}\,(\dot{\Phi}%
_{0}\cdot\tau)(\dot{\Phi}_{0}\cdot\nu)+\partial_{\nu}Q^{2}\,(\dot{\Phi}%
_{0}\cdot\nu)^{2}\big)\,d\mathcal{H}^{1}.
\]
Since $\nabla u=\partial_{\nu}u\,\nu$ on $\Gamma$, the first term in the
second line of (\ref{2der}) becomes
\begin{align*}
-\int_{\Gamma}2\big(\nabla u\cdot\nabla\dot{u}_{0}\big)\,\dot{\Phi}_{0}%
\cdot\nu\,d\mathcal{H}^{1}  &  =-\int_{\Gamma}2\partial_{\nu}\dot{u}%
_{0}\partial_{\nu}u\,\dot{\Phi}_{0}\cdot\nu\,d\mathcal{H}^{1}\\
&  =\int_{\Gamma}2\dot{u}_{0}\partial_{\nu}\dot{u}_{0}\,d\mathcal{H}^{1},
\end{align*}
where we used the fact that $\dot{u}_{0}=-(X_{0}\cdot\nu)\partial_{\nu}u$ by
(\ref{udot on G}) and $X_{0}=\dot{\Phi}_{0}$. We now focus on the term
\[
-\int_{\Gamma}2\big(D^{2}u\nabla u\cdot\dot{\Phi}_{0}\big)\,\dot{\Phi}%
_{0}\cdot\nu\,d\mathcal{H}^{1}.
\]
Using again the fact that $\nabla u\cdot\tau=0$ on $\Gamma$ and that $u$ is
harmonic, we obtain
\[
0=\partial_{\tau}(\nabla u\cdot\tau)=D^{2}u\,\tau\cdot\tau+\nabla
u\cdot\partial_{\tau}\tau=-D^{2}u\,\nu\cdot\nu-\kappa\partial_{\nu}%
u\quad\text{on }\Gamma,
\]
that is,
\[
D^{2}u\,\nu\cdot\nu=-\kappa\partial_{\nu}u\quad\text{on }\Gamma.
\]
Thus,
\begin{align*}
\lefteqn{-\int_{\Gamma}2\big(D^{2}u\nabla u\cdot\dot{\Phi}_{0}\big)\,\dot
{\Phi}_{0}\cdot\nu\,d\mathcal{H}^{1}}\\
&  =-\int_{\Gamma}2\big(D^{2}u\nabla u\cdot\tau\big)(\dot{\Phi}_{0}\cdot
\tau)(\dot{\Phi}_{0}\cdot\nu)\,d\mathcal{H}^{1}-\int_{\Gamma}2\partial_{\nu
}u\big(D^{2}u\,\nu\cdot\nu\big)(\dot{\Phi}_{0}\cdot\nu)^{2}\,d\mathcal{H}%
^{1}\\
&  =-\int_{\Gamma}\partial_{\tau}\big(|\nabla u|^{2}\big)(\dot{\Phi}_{0}%
\cdot\tau)(\dot{\Phi}_{0}\cdot\nu)\,d\mathcal{H}^{1}+\int_{\Gamma}%
2\kappa(\partial_{\nu}u)^{2}(\dot{\Phi}_{0}\cdot\nu)^{2}\,d\mathcal{H}^{1}.
\end{align*}
Finally, by \cite[Lemma~3.8]{CMM} we have
\[
\frac{d}{ds}\big[\dot{\Phi}_{s}\cdot(\nu_{s}\circ\Phi_{s})\,J_{\Phi_{s}%
}\big]\Big|_{s=0}=\ddot{\Phi}_{0}\cdot\nu-2(\dot{\Phi}_{0}\cdot\tau
)\partial_{\tau}(\dot{\Phi}_{0}\cdot\nu)+\kappa(\dot{\Phi}_{0}\cdot\tau
)^{2}+\partial_{\tau}\big[(\dot{\Phi}_{0}\cdot\nu)\dot{\Phi}_{0}\big]\cdot
\tau.
\]
Combining the previous equalities, we deduce that
\begin{align*}
\frac{d^{2}}{ds^{2}}\mathcal{F}(u_{s})\Big|_{s=0}  &  =\int_{\Gamma}2\dot
{u}_{0}\partial_{\nu}\dot{u}_{0}\,d\mathcal{H}^{1}+\int_{\Gamma}%
\big(\partial_{\nu}Q^{2}+2\kappa(\partial_{\nu}u)^{2}\big)(\dot{\Phi}_{0}%
\cdot\nu)^{2}\,d\mathcal{H}^{1}\\
&  \quad+\int_{\Gamma}(Q^{2}-|\nabla u|^{2})\big(\ddot{\Phi}_{0}\cdot
\nu-2(\dot{\Phi}_{0}\cdot\tau)\partial_{\tau}(\dot{\Phi}_{0}\cdot\nu
)+\kappa(\dot{\Phi}_{0}\cdot\tau)^{2}\big)\,d\mathcal{H}^{1}\\
&  \quad+\int_{\Gamma}(Q^{2}-|\nabla u|^{2})\,\partial_{\tau}\big[(\dot{\Phi
}_{0}\cdot\nu)\dot{\Phi}_{0}\big]\cdot\tau\,d\mathcal{H}^{1}\\
&  \quad+\int_{\Gamma}\partial_{\tau}\big(Q^{2}-|\nabla u|^{2}\big)(\dot{\Phi
}_{0}\cdot\tau)(\dot{\Phi}_{0}\cdot\nu)\,d\mathcal{H}^{1}.
\end{align*}
Now, using the fundamental theorem of calculus on curves, the last two
integrals in the formula above satisfy
\begin{align*}
\int_{\Gamma}(Q^{2}-|\nabla u|^{2})\,  &  \partial_{\tau}\big[(\dot{\Phi}%
_{0}\cdot\nu)\dot{\Phi}_{0}\big]\cdot\tau\,d\mathcal{H}^{1}+\int_{\Gamma
}\partial_{\tau}\big(Q^{2}-|\nabla u|^{2}\big)(\dot{\Phi}_{0}\cdot\tau
)(\dot{\Phi}_{0}\cdot\nu)\,d\mathcal{H}^{1}\\
&  =\int_{\Gamma}\partial_{\tau}\big[(Q^{2}-|\nabla u|^{2})\,(\dot{\Phi}%
_{0}\cdot\nu)\dot{\Phi}_{0}\big]\cdot\tau\,d\mathcal{H}^{1}\\
&  =\int_{\Gamma}\kappa(Q^{2}-|\nabla u|^{2})\,(\dot{\Phi}_{0}\cdot\nu
)^{2}\,d\mathcal{H}^{1}.
\end{align*}
Thus, we conclude that
\begin{align}\nonumber
\frac{d^{2}}{ds^{2}}\mathcal{F}(u_{s})\Big|_{s=0}=  &  \int_{\Gamma}2\dot
{u}_{0}\partial_{\nu}\dot{u}_{0}\,d\mathcal{H}^{1}+\int_{\Gamma}%
\big(\partial_{\nu}Q^{2}+2\kappa(\partial_{\nu}u)^{2}\big)(X_{0}\cdot\nu
)^{2}\,d\mathcal{H}^{1}\\
&  \quad+\int_{\Gamma}(Q^{2}-|\nabla u|^{2})\big(Z_{0}\cdot\nu-2(X_{0}%
\cdot\tau)\,\partial_{\tau}(X_{0}\cdot\nu)+\kappa|X_{0}|^{2}%
\big)\,d\mathcal{H}^{1},\label{ir01}
\end{align}
where we used the fact that $\dot{\Phi}_{0}=X_{0}$ and $\ddot{\Phi}_{0}=Z_{0}$.

Let us now fix $r\in(0,1)$. We observe that the family of diffeomorphisms
$\{\tilde{\Phi}_{h}\}_{h\in\lbrack0,1]}$ defined as
\[
\tilde{\Phi}_{h}:=\Phi_{r+h}\circ\Phi_{r}^{-1}%
\]
is still an admissible flow (we can always reparametrize the variable $h$ away
from $0$ so that $\tilde{\Phi}_{h}$ is defined for all $h\in\lbrack0,1]$), and
that
\[
\dot{{\tilde{\Phi}}}_{0}=X_{r},\qquad\ddot{\tilde{\Phi}}_{0}=Z_{r}.
\]
Applying (\ref{ir01}), we deduce that
\begin{align*}
\frac{d^{2}}{ds^{2}} &  \mathcal{F}(u_{s})\Big|_{s=r}=\frac{d^{2}}{dh^{2}%
}\mathcal{F}(u_{r+h})\Big|_{h=0}\\
= &  \int_{\Gamma_{r}}2\dot{u}_{r}\partial_{\nu_{r}}\dot{u}_{r}\,d\mathcal{H}%
^{1}+\int_{\Gamma_{r}}\big(\partial_{\nu_{r}}Q^{2}+2\kappa_{r}(\partial
_{\nu_{r}}u_{r})^{2}\big)(X_{r}\cdot\nu_{r})^{2}\,d\mathcal{H}^{1}\\
&  \quad+\int_{\Gamma_{r}}(Q^{2}-|\nabla u_{r}|^{2})\big(Z_{r}\cdot\nu
_{r}-2(X_{r}\cdot\tau_{r})\,\partial_{\tau_{r}}(X_{r}\cdot\nu_{r})+\kappa
_{r}|X_{r}|^{2}\big)\,d\mathcal{H}^{1}.
\end{align*}

To conclude the proof of (\ref{2var}), it remains to show that
\begin{equation}
\int_{\Gamma_{s}}2\dot{u}_{s}\partial_{\nu_{s}}\dot{u}_{s}\,d\mathcal{H}%
^{1}=\int_{\Phi_{s}(\Omega_{+})}2|\nabla\dot{u}_{s}|^{2}\,d\boldsymbol{x}.
\label{variant}%
\end{equation}
Indeed, by (\ref{udot eq}) and the divergence theorem
\begin{align*}
\int_{\Gamma_{s}}2\dot{u}_{s}\partial_{\nu_{s}}\dot{u}_{s}\,d\mathcal{H}^{1}
&  =\int_{\partial\Phi_{s}(\Omega_{+})}2\dot{u}_{s}\partial_{\nu_{s}}\dot
{u}_{s}\,d\mathcal{H}^{1}\\
&  =\int_{\Phi_{s}(\Omega_{+})}2\big(\dot{u}_{s}\Delta\dot{u}_{s}+|\nabla
\dot{u}_{s}|^{2}\big)\,d\boldsymbol{x}\\
&  =\int_{\Phi_{s}(\Omega_{+})}2|\nabla\dot{u}_{s}|^{2}\,d\boldsymbol{x}.
\end{align*}
Hence, (\ref{2var}) holds and the proof is complete.
\end{proof}

\section{Construction of the Family $\Phi_{s}$}

\label{section diff}

Let $\Omega$, $u$, and $\Gamma$ be as in
Section~\ref{section second variation}, and assume that
\[
\Gamma=\left\{  (x,w(x)):\,x\in(-1,1)\right\}  ,
\]
where $w$ is a periodic function with $w\in C^{3}(\mathbb{R})$ and
\begin{equation}
w(x)>0\quad\text{for all }x\in\lbrack-1,1]. \label{w positive}%
\end{equation}
Let
\[
-1<a<b<1
\]
and consider a polynomial $\varphi:[a,b]\rightarrow\mathbb{R}$ satisfying
\begin{equation}
\varphi(a)=\varphi^{\prime}(a)=\varphi^{\prime\prime}(a)=\varphi^{\prime
\prime\prime}(a)=0,\quad\varphi(b)=\varphi^{\prime}(b)=\varphi^{\prime\prime
}(b)=\varphi^{\prime\prime\prime}(b)=0,
\label{f and itz derivatives zero at endpoints}%
\end{equation}
and such that $\left\Vert \varphi\right\Vert _{C^{2,\alpha}(a,b)}<<1$. Extend
$\varphi$ to be zero outside $[a,b]$. In this section we construct an
admissible flow (see Definition~\ref{definition admissible flow}) joining
$\Gamma$ to $\operatorname*{graph}(w+\varphi)$. To estimate the second
variation along the flow it is essential to have the condition $X_{s}\cdot
\tau_{s}=0$ on $\Gamma_{s}$ for every $s$. This leads to a first order partial
differential equation (see (\ref{i equiv}) below), that we solve using the
method of characteristics. One of the main difficulties is that the components
of the flow are given by compositions of functions that are discontinuous.
Thus, proving the regularity of the flow is extremely delicate and it will be
carried out in the appendix. The construction of the flow is the central part
of this paper and will require several preliminary results.

\begin{theorem}
\label{theorem diffeomorphism}Let $\varphi$ and $w$ be as above. Then there
exists an admissible flow $\{\Phi_{s}\}_{s\in\lbrack0,1]}$ such that
\begin{equation}
\Phi_{s}(\Gamma)=\left\{  (x,w(x)+s\varphi(x)):\,x\in(-1,1)\right\}  
\label{Phi s Gamma}%
\end{equation}
and
\begin{equation}
X_{s}\cdot\tau_{s}=0\quad\text{on }\Gamma_{s}
\label{tangential part equal zero}%
\end{equation}
for every $s\in\lbrack0,1]$.
\end{theorem}

For every $x\in\mathbb{R}$, we consider the initial value problem%
\begin{equation}%
\begin{cases}
\dfrac{d\xi}{dt}=-w^{\prime}(\xi)\varphi(\xi)-(\eta-w(\xi))\varphi^{\prime
}(\xi),\smallskip\\
\dfrac{d\eta}{dt}=\varphi(\xi),\smallskip\\
\xi(0)=x,\ \eta(0)=w(x).
\end{cases}
\label{IVP}%
\end{equation}
Since $\varphi\in C_{c}^{3}(\mathbb{R})$, $w\in C^{3}(\mathbb{R})$, the
function
\[
(\xi,\eta)\in\mathbb{R}^{2}\mapsto(-w^{\prime}(\xi)\varphi(\xi)-(\eta
-w(\xi))\varphi^{\prime}(\xi),\varphi(\xi))
\]
is $C^{2}$, globally Lipschitz, and satisfies%
\[
\left\vert (-w^{\prime}(\xi)\varphi(\xi)-(\eta-w(\xi))\varphi^{\prime}%
(\xi),\varphi(\xi))\right\vert \leq C(1+\left\vert \eta\right\vert )
\]
for all $(\xi,\eta)\in\mathbb{R}^{2}$ and for some constant $C>0$. Hence, the
initial value problem (\ref{IVP}) admits a unique global solution, which
depends smoothly on the initial datum, and thus on $x$. We will denote by
$(\xi(t,x),\eta(t,x))$, $t\in\mathbb{R}$, the solution of (\ref{IVP}).

Observe that if $\varphi(x)=0$, then
\begin{equation}
(\xi(t,x),\eta(t,x))\equiv(x,w(x)). \label{solution when u=0}%
\end{equation}

\begin{remark}
\label{remark uniqueness}Note that if $\varphi^{\prime}(x_{0})=w^{\prime
}(x_{0})=0$ or $\varphi(x_{0})=\varphi^{\prime}(x_{0})=0$ for some $x_{0}%
\in\mathbb{R}$, then for every $y_{0}\in\mathbb{R}$ the unique solution of the
initial value problem
\begin{equation}
\label{ivp2}%
\begin{cases}
\dfrac{d\xi}{dt}=-w^{\prime}(\xi)\varphi(\xi)-(\eta-w(\xi))\varphi^{\prime
}(\xi), \smallskip\\
\dfrac{d\eta}{dt}=\varphi(\xi), \smallskip\\
\xi(0)=x_{0}, \ \eta(0)=y_{0},
\end{cases}
\end{equation}
is given by
\begin{equation}
\xi(t,x_{0})\equiv x_{0},\quad\eta(t,x_{0})=y_{0}+t\varphi(x_{0}).
\label{solution f'(x)=0}%
\end{equation}
Hence, if for some $\alpha<\beta$ we have $\varphi(\alpha)=\varphi^{\prime
}(\alpha)=0$ and $\varphi(\beta)=\varphi^{\prime}(\beta)=0$, then for every
$\alpha<x<\beta$ the curve $(\xi(\cdot,x),\eta(\cdot,x))$ cannot leave the
vertical strip $(\alpha,\beta)\times\mathbb{R}$, otherwise uniqueness for the
initial value problem (\ref{ivp2}) would be violated. In particular, in view
of (\ref{f and itz derivatives zero at endpoints}), if $a<x<b$ then the curve
$(\xi(\cdot,x),\eta(\cdot,x))$ cannot leave the vertical strip $(a,b)\times
\mathbb{R}$.
\end{remark}

\begin{theorem}
\label{theorem existence t0}Let $\varphi$ and $w$ be as above. Given
$s\in[0,1]$ and $x\in\mathbb{R}$, there exists a first time $t_{0}%
=t_{0}(s,x)\geq0$ such that the solution $(\xi(\cdot,x),\eta(\cdot,x))$ of
(\ref{IVP}) intersects the graph of the function $w+s\varphi$ at time $t_{0}$.
Moreover, if $s=0$ or $\varphi(x)=0$ then
\begin{equation}
t_{0}(s,x)=0, \label{t0 when u=0}%
\end{equation}
if $\varphi(x)\neq0$ and $\varphi^{\prime}(x)=w^{\prime}(x)=0$ then
\begin{equation}
t_{0}(s,x)=s, \label{t0 when u'=0}%
\end{equation}
while in all the other cases%
\begin{equation}
0\leq t_{0}(s,x)\leq s. \label{t0 bounded}%
\end{equation}
Finally, $t_{0}$ is of class $C^{2}$ in $[0,1]\times\left\{  x\in
\mathbb{R}:\,\varphi(x)\neq0\right\}  $ and if $\varphi(x)\neq0$, then%
\begin{equation}
\varphi(\xi(t,x))\varphi(x)>0\quad\text{for all }0\leq t\leq t_{0}(s,x).
\label{u(x)u(xi)>0}%
\end{equation}

\end{theorem}

\begin{proof}
We begin by proving the existence of $t_{0}$. If $s=0$ or $\varphi(x)=0$, then
$t_{0}(s,x)=0$ by (\ref{IVP}). Property (\ref{t0 when u'=0}) follows from
(\ref{ivp2}) and (\ref{solution f'(x)=0}) with $x_{0}=x$ and $y_{0}=w(x)$.
Thus, in what follows assume that $\varphi(x)\neq0$, $s>0$, and that at least
one of $\varphi^{\prime}(x)$ and $w^{\prime}(x)$ is different from zero. By
(\ref{f and itz derivatives zero at endpoints}) and
Remark~\ref{remark uniqueness}, the curve $(\xi(\cdot,x),\eta(\cdot,x))$
cannot leave the vertical strip $(a,b)\times\mathbb{R}$.\smallskip

\noindent\textbf{Step 1:} Assume that $\varphi(x)>0$ and let $T>0$ be the
first time, if it exists, such that $\varphi(\xi(T,x))=0$, otherwise set
$T:=\infty$. Then $\varphi(\xi(t,x))>0$ for all $0\leq t<T$, and so by
(\ref{IVP}), $\eta(\cdot,x)$ is strictly increasing in $[0,T)$ and there
exists
\begin{equation}
\lim_{t\rightarrow T^{-}}\eta(t,x)=\ell\in(w(x),\infty]. \label{limit eta}%
\end{equation}
If $\ell=\infty$ (and hence $T=\infty$), then the solution $(\xi(\cdot
,x),\eta(\cdot,x))$ of (\ref{IVP}) intersects the graph of the function
$w+s\varphi$, and so $t_{0}$ exists. Thus, in what follows it suffices to
consider the case $\ell<\infty$. Assume, by contradiction, that%
\begin{equation}
\eta(t,x)<w(\xi(t,x))+s\varphi(\xi(t,x))\quad\text{for all }0<t\leq T.
\label{eta below graph}%
\end{equation}

\noindent\textbf{Substep 1a:} We claim that the curve $(\xi(\cdot
,x),\eta(\cdot,x))$ stays above the graph of the function $w$ for all $0<t\leq
T$. Consider the function%
\begin{equation}\label{ir02}
G(\xi,\eta):=\frac{\eta-w(\xi)}{\varphi(\xi)}%
\end{equation}
defined for all $\varphi(\xi)\neq0$ and $\eta\in\mathbb{R}$. Then
\[
\partial_{\xi}G(\xi,\eta)=\frac{-w^{\prime}(\xi)\varphi(\xi)-(\eta
-w(\xi))\varphi^{\prime}(\xi)}{\varphi^{2}(\xi)},\quad\partial_{\eta}%
G(\xi,\eta)=\frac{1}{\varphi(\xi)},
\]
and so by (\ref{IVP}), for all $0<t<T$,
\begin{align*}
\partial_{t}\xi(t,x)  &  =\varphi^{2}(\xi(t,x))\partial_{\xi}G(\xi
(t,x),\eta(t,x)),\\
\partial_{t}\eta(t,x)  &  =\varphi^{2}(\xi(t,x))\partial_{\eta}G(\xi
(t,x),\eta(t,x)).
\end{align*}
It follows that for all $0<t<T$,
\begin{align*}
\frac{\left(  \partial_{t}\xi(t,x)\right)  ^{2}+\left(  \partial_{t}%
\eta(t,x)\right)  ^{2}}{\varphi^{2}(\xi(t,x))}  &  =\partial_{\xi}%
G(\xi(t,x),\eta(t,x))\partial_{t}\xi(t,x)\\
&  \quad+\partial_{\eta}G(\xi(t,x),\eta(t,x))\partial_{t}\eta(t,x)\\
&  =\partial_{t}(G(\xi(t,x),\eta(t,x)).
\end{align*}
Integrating between $0$ and $t$, and using (\ref{IVP}) once more, gives
\begin{equation}
\int_{0}^{t}\frac{\left(  \partial_{t}\xi(r,x)\right)  ^{2}+\left(
\partial_{t}\eta(r,x)\right)  ^{2}}{\varphi^{2}(\xi(r,x))}\,dr=\frac
{\eta(t,x)-w(\xi(t,x))}{\varphi(\xi(t,x))}. \label{conservation law}%
\end{equation}
Since the integrand on the left-hand side is positive for all $0<t<T$, it
follows that $\eta(t,x)>w(\xi(t,x))$ for all $0<t<T$. If $T<\infty$, then also
by (\ref{eta below graph}) we have that
\[
w(\xi(t,x))<\eta(t,x)<w(\xi(t,x))+s\varphi(\xi(t,x)),
\]
and letting $t\rightarrow T^{-}$ gives that $w(\xi(T,x))=\eta(T,x)=w(\xi
(T,x))+s\varphi(\xi(T,x))$, where we have used the fact that $\varphi
(\xi(T,x))=0$. This contradicts (\ref{eta below graph}), and thus establishes the
existence of $t_{0}$ in the case $T<\infty$.

It remains to study the case $T=\infty$. Since $\frac{\left(
\partial_{t}\eta(r,x)\right)  ^{2}}{\varphi^{2}(\xi(r,x))}=1$ by (\ref{IVP}),
it follows from (\ref{conservation law}) that%
\begin{equation}
t\leq\frac{\eta(t,x)-w(\xi(t,x))}{\varphi(\xi(t,x))} \label{bound time}%
\end{equation}
for all $t>0$. Let $\alpha<x<\beta$ be such that $\varphi>0$ in $(\alpha
,\beta)$ and $\varphi(\alpha)=\varphi(\beta)=0$. We claim that there exists%
\begin{equation}
\lim_{t\rightarrow\infty}\xi(t,x)=l\in\{\alpha,\beta\}.
\label{xi goes to endpoints}%
\end{equation}
To see this, note that since $\varphi(\xi(t,x))>0$ for all $t$, we have that%
\[
\alpha\leq l_{1}:=\liminf_{t\rightarrow\infty}\xi(t,x),\quad l_{2}%
:=\limsup_{t\rightarrow\infty}\xi(t,x)\leq\beta.
\]
Assume, by contradiction, that $l_{1}<l_{2}$. Then there exists a sequence
$t_{i}\rightarrow\infty$ such that $\xi(t_{i},x)\rightarrow c\in(\alpha
,\beta)$. Taking $t=t_{i}$ in (\ref{bound time}), and using (\ref{limit eta})
and the fact that $\ell<\infty$, gives
\[
\infty=\lim_{i\rightarrow\infty}t_{i}=\lim_{i\rightarrow\infty}\frac
{\eta(t_{i},x)-w(\xi(t_{i},x))}{\varphi(\xi(t_{i},x))}=\frac{\ell
-w(c)}{\varphi(c)}<\infty,
\]
which is a contradiction. Hence, $l_{1}=l_{2}$. Note that the previous
argument also shows that $l_{1}$ cannot belong to $(\alpha,\beta)$. Hence,
either $l_{1}=\alpha$ or $l_{2}=\beta$.\smallskip

\noindent\textbf{Substep 1b:} We prove the existence of $t_{0}$. Without loss
of generality, assume that $l=\alpha$ (the case $l=\beta$ is similar). Then by
(\ref{eta below graph}) and Substep~1a, we have%
\[
w(\xi(t,x))<\eta(t,x)<w(\xi(t,x))+s\varphi(\xi(t,x))
\]
for all $t>0$. Hence,
\[
0<\frac{\eta(t,x)-w(\xi(t,x))}{\varphi(\xi(t,x))}<s.
\]
Letting $t\rightarrow\infty$ we obtain a contradiction from (\ref{bound time}%
). Therefore, we have proved that condition (\ref{eta below graph}) fails. This
asserts the existence of $t_{0}$.\smallskip

\noindent\textbf{Substep 1c:} We prove (\ref{u(x)u(xi)>0}). It follows from
Substeps~1a and~1b that $t_{0}(s,x)\leq T$, so that
\[
\varphi(\xi(t,x))\varphi(x)>0\quad\text{for all }0\leq t<t_{0}(s,x).
\]
To prove (\ref{u(x)u(xi)>0}), it remains to show that $\varphi(\xi
(t_{0}(s,x),x))>0$. Let $t_{0}:=t_{0}(s,x)$. Assume, by contradiction, that
$\varphi(\xi(t_{0},x))=0$. Then by the definition of $t_{0}(s,x)$ we have
that
\[
\eta(t_{0},x)=w(\xi(t_{0},x)),
\]
which contradicts the fact that the unique solution of the initial value
problem
\[%
\begin{cases}
\dfrac{d\xi}{dt}=-w^{\prime}(\xi)\varphi(\xi)-(\eta-w(\xi))\varphi^{\prime
}(\xi), \smallskip\\
\dfrac{d\eta}{dt}=\varphi(\xi), \smallskip\\
\xi(t_{0})=\xi(t_{0},x), \ \eta(t_{0})=w(\xi(t_{0},x)),
\end{cases}
\]
is given by
\[
\xi_{1}(t)\equiv\xi(t_{0},x),\quad\eta_{1}(t,x_{0})\equiv w(\xi(t_{0},x)).
\]

\noindent\textbf{Step 2:} The case $\varphi(x)<0$ is similar and we omit
it.\smallskip

\noindent\textbf{Step 3:} We prove the regularity of $t_{0}$. Fix
$(s_{0},x_{0})\in\lbrack0,1]\times\mathbb{R}$, with $\varphi(x_{0})\neq0$, and
let $t_{0}:=t_{0}(s_{0},x_{0})$. Assume that $\varphi(x_{0})>0$ (the case
$\varphi(x_{0})<0$ is similar), and let $\alpha<x_{0}<\beta$ be such that
$\varphi>0$ in $(\alpha,\beta)$ and $\varphi(\alpha)=\varphi(\beta)=0$.
Consider the $C^{2}$ function
\[
F(s,t,x):=\eta(t,x)-w(\xi(t,x))-s\varphi(\xi(t,x))
\]
defined in the set
\[
V:=\mathbb{R}\times\mathbb{R}\times(\alpha,\beta).
\]
Then $F(s_{0},t_{0},x_{0})=0$. By (\ref{IVP}) and (\ref{u(x)u(xi)>0}), we have%
\begin{align*}
\partial_{t}F(s_{0},t_{0},x_{0})  &  =\partial_{t}\eta(t_{0},x_{0})-\left[
w^{\prime}(\xi(t_{0},x_{0}))+s_{0}\varphi^{\prime}(\xi(t_{0},x_{0}))\right]
\partial_{t}\xi(t_{0},x_{0})\\
&  =\varphi(\xi(t_{0},x_{0}))\left[  1+\left(  w^{\prime}(\xi(t_{0}%
,x_{0}))+s_{0}\varphi^{\prime}(\xi(t_{0},x_{0}))\right)  ^{2}\right]  >0.
\end{align*}
Thus, we can apply the implicit function theorem to conclude
that there exist $0<r<\min\{\beta-x_{0},x_{0}-\alpha\}$, $\delta>0$, a
function $t_{1}:B((s_{0},x_{0});r)\rightarrow\lbrack t_{0}-\delta,t_{0}%
+\delta]$ of class $C^{2}$ such that $t_{1}(s_{0},x_{0})=t_{0}$ and
\[
F(s,t_{1}(s,x),x)=0\quad\text{for all }(s,x)\in B((s_{0},x_{0});r).
\]
Note that, in view of (\ref{conservation law}), which continues to hold for
$t<0$ small, and the fact that $\varphi>0$, for $t<0$ sufficiently small,
\[
\eta(t,x)<w(\xi(t,x))\leq w(\xi(t,x))+s\varphi(\xi(t,x)),
\]
and so the function $t_{1}$ must be nonnegative. Hence, by the definition of
$t_{0}(s,x)$, we have%
\[
t_{0}(s,x)\leq t_{1}(s,x)\quad\text{for all }(s,x)\in B((s_{0},x_{0});r).
\]
We claim that $t_{1}(s,x)=t_{0}(s,x)$ for all $(s,x)\in B((s_{0},x_{0}%
);r_{1})$ for some $0<r_{1}<r$. Since $t_{0}$ is the first time that the
solution $(\xi(\cdot,x_{0}),\eta(\cdot,x_{0}))$ of (\ref{IVP}) intersects the
graph of the function $w+s_{0}\varphi$, if $t_{0}>0$ we have that
\[
\eta(t,x_{0})<w(\xi(t,x_{0}))+s_{0}\varphi(\xi(t,x_{0}))\quad\text{for all
}0\leq t<t_{0}.
\]
Fix $0<\varepsilon<\delta$ and let
\[
c_{\varepsilon}:=\min_{0\leq t\leq t_{0}-\varepsilon}\left(  w(\xi
(t,x_{0}))+s_{0}\varphi(\xi(t,x_{0}))-\eta(t,x_{0})\right)  >0.
\]
By the regularity of $w$ and $\varphi$ and the continuity of $\xi$ and $\eta$
with respect to initial data, there exists $0<r_{1}<r$ such that
\[
\left\vert w(\xi(t,x))+s\varphi(\xi(t,x))-\eta(t,x)-\left(  w(\xi
(t,x_{0}))+s_{0}\varphi(\xi(t,x_{0}))-\eta(t,x_{0})\right)  \right\vert
\leq\frac{1}{2}c_{\varepsilon}%
\]
for all $(s,x)\in B((s_{0},x_{0});r_{1})$ and for all $t\in\lbrack
0,t_{0}-\varepsilon]$. Hence,
\[
\eta(t,x)<w(\xi(t,x))+s\varphi(\xi(t,x))\quad\text{for all }0\leq
t<t_{0}-\varepsilon
\]
for all $(s,x)\in B((s_{0},x_{0});r_{1})$. This implies that%
\begin{equation}
t_{0}-\varepsilon\leq t_{0}(s,x) \label{t0 below}%
\end{equation}
for all $(s,x)\in B((s_{0},x_{0});r_{1})$. If $t_{0}=0$, then (\ref{t0 below})
continues to hold since $t_{0}(s,x)\geq0$. On the other hand, since $t_{1}$ is
continuous and $t_{1}(s_{0},x_{0})=t_{0}$, by taking $r_{1}$ smaller if
necessary, we have that $t_{0}-\varepsilon\leq t_{1}(s,x)\leq t_{0}%
+\varepsilon$ for all $(s,x)\in B((s_{0},x_{0});r_{1})$. Because $t_{0}(s,x)\leq
t_{1}(s,x)$, also by (\ref{t0 below}), we have that
\[
t_{0}-\varepsilon\leq t_{0}(s,x)\leq t_{0}+\varepsilon
\]
for all $(s,x)\in B((s_{0},x_{0});r_{1})$. Using the fact that $\varepsilon<\delta$, it
follows from the uniqueness of the implicit function that $t_{0}%
(s,x)=t_{1}(s,x)$ for all $(s,x)\in B((s_{0},x_{0});r_{1})$. In turn,
\[
F(s,t_{0}(s,x),x)=0\quad\text{for all }(s,x)\in B((s_{0},x_{0});r_{1}),
\]
and so, by (\ref{IVP}) and the definition of $t_{0}$, we have%
\begin{align}
\partial_{x}t_{0}(s,x)  &  =-\frac{\partial_{x}F(s,t_{0}(s,x),x)}{\partial
_{t}F(s,t_{0}(s,x),x)}\nonumber\\
&  =\frac{\left[  w^{\prime}(\xi(t_{0}(s,x),x))+s\varphi^{\prime}(\xi
(t_{0}(s,x),x))\right]  \partial_{x}\xi(t_{0}(s,x),x)-\partial_{x}\eta
(t_{0}(s,x),x)}{\varphi(\xi(t_{0}(s,x),x))\left[  1+\left(  w^{\prime}%
(\xi(t_{0}(s,x),x))+s\varphi^{\prime}(\xi(t_{0}(s,x),x))\right)  ^{2}\right]
}. \label{partial x of t0}%
\end{align}

\noindent\textbf{Step 4:} It remains to prove (\ref{t0 bounded}). By
(\ref{bound time}) and the definition of $t_{0}$, we have that\emph{ }%
\[
0\leq t_{0}(s,x)\leq\frac{\eta(t_{0}(s,x),x)-w(\xi(t_{0}(s,x),x))}{\varphi
(\xi(t_{0}(s,x),x))}=s.
\]
This concludes the proof.
\end{proof}

\begin{remark}
By (\ref{IVP}), if $x\in\mathbb{R}$ and $0\leq t\leq1$, then%
\[
|\eta(t,x)-w(x)|=\left\vert \int_{0}^{t}\varphi(\xi(r,x))\,dr\right\vert
\leq\left\Vert \varphi\right\Vert _{C\left(  \left[  a,b\right]  \right)  }%
\]
and, in turn,
\begin{align*}
\left\vert \xi(t,x)-x\right\vert  &  =\left\vert \int_{0}^{t}\left[
-w^{\prime}(\xi(r,x))\varphi(\xi(r,x))-(\eta(r,x)-w(\xi(r,x)))\varphi^{\prime
}(\xi(r,x))\right]  \,dr\right\vert \\
&  \leq3\left\Vert w\right\Vert _{C^{1}([a,b])}\left\Vert \varphi\right\Vert
_{C^{1}([a,b])}+\left\Vert \varphi\right\Vert _{C([a,b])}\left\Vert
\varphi^{\prime}\right\Vert _{C([a,b])}.
\end{align*}
Since $0\leq t_{0}(s,x)\leq s\leq1$ by Theorem~\ref{theorem existence t0}, it
follows that
\begin{eqnarray}
& \left\vert \xi(t,x)-x\right\vert    \leq3\left(  \left\Vert w\right\Vert
_{C^{1}([a,b])}+\left\Vert \varphi\right\Vert _{C([a,b])}\right)  \left\Vert
\varphi\right\Vert _{C^{1}([a,b])}, \label{bound xi-x and eta in terms of f}\smallskip\\
& \left\vert \eta(t,x)-w(x)\right\vert    \leq\left\Vert \varphi\right\Vert
_{C([a,b])}\nonumber
\end{eqnarray}
for all $s\in[0,1]$, $x\in\mathbb{R}$, and $0\leq t\leq t_{0}(s,x)$.
\end{remark}

Given $s\in\lbrack0,1]$ and $x\in\mathbb{R}$, we define the function
$g:[0,1]\times\mathbb{R}\rightarrow\mathbb{R}$ by
\begin{equation}
g(s,x):=\xi(t_{0}(s,x),x), \label{g}%
\end{equation}
where $t_{0}(s,x)$ is given by Theorem~\ref{theorem existence t0}. Note that
by the definition of $t_{0}(s,x)$, if $s=0$ or $\varphi(x)=0$, then
$t_{0}(s,x)=0$, and so%
\begin{equation}
g(s,x)=\xi(0,x)=x. \label{g at s=0}%
\end{equation}
Moreover, since $(\xi(t_{0}(s,x),x),\eta(t_{0}(s,x),x))$ belongs to the graph
of $w+s\varphi$, we have that
\begin{align}
\eta(t_{0}(s,x),x)  &  =w(\xi(t_{0}(s,x),x))+s\varphi(\xi(t_{0}%
(s,x),x))\label{eta at time t0}\\
&  =w(g(s,x))+s\varphi(g(s,x)).\nonumber
\end{align}
We will use this property in the sequel.

The following theorem states that the function $g$ is of class $C^{2}$. As
part of the proof we will actually show that the function $t_{0}$ is
\emph{discontinuous at all points} $(s,x_{0})$ with $\varphi(x_{0})=0$ and
$\varphi\neq0$ near $x_{0}$. Hence, establishing the regularity of $g$ is far
from trivial. The proof of Theorem~\ref{theorem regularity g} is rather
lengthy and will be presented in the appendix.

\begin{theorem}
\label{theorem regularity g}Let $\varphi$ and $w$ be as above. Then the
function $g:[0,1]\times\mathbb{R}\rightarrow\mathbb{R}$ defined in (\ref{g})
is of class $C^{2}$.
\end{theorem}

Define the function $h:[0,1]\times\mathbb{R}\rightarrow\mathbb{R}$ by%
\begin{equation}
h(s,x):=\eta(t_{0}(s,x),x), \label{h definition}%
\end{equation}
where $t_{0}(s,x)$ is given by Theorem~\ref{theorem existence t0}. Note that,
by (\ref{eta at time t0}),
\begin{equation}
h(s,x)=w(g(s,x))+s\varphi(g(s,x)). \label{h}%
\end{equation}
Thus in view of Theorem~\ref{theorem regularity g}, the function $h$ is of
class $C^{2}$. Moreover, by (\ref{g at s=0}), if $s=0$ or $\varphi(x)=0$, then%
\begin{equation}
h(s,x)=w(x). \label{h at s=0}%
\end{equation}

\begin{theorem}
\label{theorem graph}Let $\varphi$ and $w$ be as above. Then for every
$s\in\lbrack0,1]$,%
\[
\left\{  (g(s,x),h(s,x)):\,x\in\lbrack a,b]\right\}  =\left\{
(x,w(x)+s\varphi(x)):\,x\in\lbrack a,b]\right\}  .
\]

\end{theorem}

\begin{proof}
Given $s_{0}\in\lbrack0,1]$ and $x_{0}\in\lbrack a,b]$, we want to find
$x\in\lbrack a,b]$ such that%
\[
(\xi(t_{0}(s_{0},x),x),\eta(t_{0}(s_{0},x),x))=(x_{0},w(x_{0})+s_{0}%
\varphi(x_{0})).
\]
If $s_{0}=0$ or $\varphi(x_{0})=0$, then by (\ref{g at s=0}) and
(\ref{h at s=0}), $g(s_{0},x_{0})=x_{0}$ and $h(s_{0},x_{0})=w(x_{0})$ and so
there is nothing to prove. Therefore, also by
(\ref{f and itz derivatives zero at endpoints}), in what follows we assume
that $s_{0}>0$, $x_{0}\in(a,b)$, and $\varphi(x_{0})\neq0$. Assume further that
$\varphi(x_{0})>0$ (the case $\varphi(x_{0})<0$ is similar).

Consider the initial value problem
\begin{equation}%
\begin{cases}
\dfrac{d\xi}{dt}=w^{\prime}(\xi)\varphi(\xi)+(\eta-w(\xi))\varphi^{\prime}%
(\xi),\smallskip\\
\dfrac{d\eta}{dt}=-\varphi(\xi),\smallskip\\
\xi(0)=x_{0},\ \eta(0)=w(x_{0})+s_{0}\varphi(x_{0}).
\end{cases}
\label{IVP 2}%
\end{equation}
Reasoning as for (\ref{IVP}), we have that (\ref{IVP 2}) admits a unique
solution $(\xi_{0},\eta_{0})$ defined for all $t\in\mathbb{R}$. We claim that
$(\xi_{0},\eta_{0})$ intersects the graph of $w$ at some time $t_{1}>0$.

For every $y\in\mathbb{R}$ the functions
\begin{align*}
\xi_{1}(t)  &  \equiv a,\quad\eta_{1}(t) \equiv y,\\
\xi_{2}(t)  &  \equiv b,\quad\eta_{2}(t) \equiv y
\end{align*}
are solutions of the differential system in (\ref{IVP 2}) with $\xi(0)=a$,
$\eta(0)=y$, and $\xi(0)=b$, $\eta(0)=y$, respectively. Hence, by uniqueness
of (\ref{IVP 2}), we conclude that the curve $(\xi_{0},\eta_{0})$ cannot leave
the vertical strip $(a,b)\times\mathbb{R}$.

Let $T>0$ be the first time, if it exists, such that $\varphi(\xi_{0}(T))=0$,
otherwise set $T:=\infty$. Then by (\ref{IVP 2}) the function $\eta_{0}$ is
strictly decreasing in $[0,T)$, and so there exists
\[
\lim_{t\rightarrow T}\eta_{0}(t)=l_{2}\in[-\infty,w(x_{0}) +s_{0}\varphi
(x_{0})).
\]
If $l_{2}=-\infty$ (and hence $T=\infty$), then there exists a time $t_{1}>0$
such that $(\xi_{0},\eta_{0})$ intersects the graph of $w$. Thus, assume that
$l_{2}\in\mathbb{R}$ and that
\begin{equation}
w(\xi_{0}(t)) <\eta_{0} ( t) \quad\text{for all }0\leq t\leq T. \label{5000}%
\end{equation}
Reasoning as in Substep~1a of the proof of Theorem~\ref{theorem existence t0},
with $G$ in (\ref{ir02}) replaced by
\[
G( \xi,\eta) :=\frac{w( \xi) -\eta}{\varphi(\xi)},
\]
we have that
\begin{align*}
\int_{0}^{t}\frac{( \partial_{t}\xi_{0}( r ))^{2}+ (\partial_{t}\eta_{0} ( r
))^{2}}{\varphi^{2}( \xi_{0}(r ) )}\,dr  &  =G ( \xi_{0} ( t ) ,\eta_{0} ( t )
) -G ( x_{0},w ( x_{0}) +s_{0}\varphi(x_{0}))\\
&  =\frac{w ( \xi_{0}( t)) -\eta_{0}( t)}{\varphi( \xi_{0}( t)) }+s_{0}%
\end{align*}
for all $0\leq t<T$. Since the integrand on the left-hand side is positive for
$t>0$, it follows that $\eta_{0}(t) < w( \xi_{0}( t)) +s_{0}\varphi( \xi
_{0}(t))$ for all $0< t<T$. As in Substep~1a of the proof of
Theorem~\ref{theorem existence t0}, if $T<\infty$ then we obtain a contradiction to
(\ref{5000}). Thus, we can assume that $T=\infty$. As in Subtep 1a of the
proof of Theorem~\ref{theorem existence t0}, the inequality
\[
t\leq\frac{w (\xi_{0} ( t )) -\eta_{0} (t) }{\varphi( \xi_{0}( t )) }+s_{0}
\]
for all $t>0$ implies (\ref{xi goes to endpoints}). The existence of $t_{1}$
follows exactly as in Substep~1b of the proof of
Theorem~\ref{theorem existence t0}.

This shows that $(\xi_{0},\eta_{0})$ intersects the graph of $w$ at some time
$t_{1}>0$. Assume that $t_{1}$ is the first such time. Define $x:=\xi
_{0}(t_{1})$. Then the function $(\xi_{0}(t_{1}-\cdot),\eta_{0}(t_{1}-\cdot))$
is the solution $(\xi(\cdot,x),\eta(\cdot,x))$ of the initial value problem
(\ref{IVP}), and at time $t=t_{1}$ it touches the graph of $w+s_{0}\varphi$ at
the point $(x_{0},w( x_{0}) +s_{0}\varphi(x_{0}))$. Hence, $t_{0}%
(s_{0},x)=t_{1}$ and
\[
(\xi(t_{0}(s,x),x),\eta(t_{0}(s,x),x))=(x_{0},w( x_{0}) +s_{0}\varphi
(x_{0})).
\]
This completes the proof.
\end{proof}

To estimate the norm of $\partial_{x}g$ and $\partial_{x}h$ we need the
following preliminary result.

\begin{proposition}
\label{proposition partial x} Let $\varphi$ and $w$ be as above with
$\left\Vert \varphi\right\Vert _{C^{2}([a,b])}<1$. Then
\[
|\partial_{x}\xi( t,x) -1|\leq C\left\Vert \varphi\right\Vert _{C^{2}%
([a,b])},\quad\left\vert \partial_{x}\eta( t,x) -w^{\prime}( x) \right\vert
\leq C\left\Vert \varphi\right\Vert _{C^{2}([a,b])}%
\]
for all $x\in\mathbb{R}$ and $0\leq t\leq t_{0}(s,x) $, where $C>0$ depends on
$\left\Vert w\right\Vert _{C^{2}([a,b])}$.
\end{proposition}

\begin{proof}
Differentiating (\ref{IVP}) with respect to $x$, we have that%
\[
\partial_{t}\left(  \left\vert \partial_{x}\xi\right\vert ^{2}+\left\vert
\partial_{x}\eta\right\vert ^{2}\right)  =-2[w^{\prime\prime}(\xi)\varphi
(\xi)+ ( \eta-w(\xi)) \varphi^{\prime\prime}(\xi)]\left\vert \partial_{x}%
\xi\right\vert ^{2}.
\]
If $\varphi( x) =0$, then by (\ref{solution when u=0}) the right-hand side of
the previous equality is identically equal to zero. If $\varphi( x) \neq0$ then
assume that $\varphi( x) >0$ (the case $\varphi( x) <0$ is similar). Using the fact that
\[
w(\xi( t,x) )\leq\eta( t,x) \leq w(\xi( t,x) )+s\varphi(\xi( t,x) )
\]
for all $0\leq t\leq t_{0}( s,x)$ (see Step~1 of the proof of
Theorem~\ref{theorem existence t0}), we obtain%
\begin{align*}
&  \left\vert \partial_{x}\xi( t,x) \right\vert ^{2}+\left\vert \partial
_{x}\eta( t,x) \right\vert ^{2}\leq1+\left(  w^{\prime}( x) \right)  ^{2}\\
&  \quad+2\left\Vert \varphi\right\Vert _{C([a,b])}\left(  \left\Vert
w^{\prime\prime}\right\Vert _{C([a,b])}+\left\Vert \varphi^{\prime\prime
}\right\Vert _{C([a,b])}\right)  \int_{0}^{t}\left(  \left\vert \partial
_{x}\xi( r,x) \right\vert ^{2}+\left\vert \partial_{x}\eta( r,x) \right\vert
^{2}\right)  \,dr
\end{align*}
for all $0\leq t\leq t_{0}( s,x)$. By Gronwall's inequality and the facts that
$t_{0}\leq1$ by Theorem~\ref{theorem existence t0} and $\left\Vert
\varphi\right\Vert _{C^{2}([a,b])}<1$, we deduce that
\begin{align*}
&  \left\vert \partial_{x}\xi(t,x) \right\vert ^{2}+\left\vert \partial
_{x}\eta(t,x) \right\vert ^{2}\\
&  \quad\leq\left(  1+\left\Vert w^{\prime}\right\Vert _{C([a,b])}^{2}\right)
\exp\left(  2\left\Vert w^{\prime\prime}\right\Vert _{C([a,b])}+2\right)
\end{align*}
for $0\leq t\leq t_{0}(s,x)$. In turn,
\begin{align}
&  \left\vert \partial_{x}\eta(t,x) -w^{\prime}(x) \right\vert \leq\int%
_{0}^{t}\left\vert \varphi^{\prime}(\xi(r,x) )\partial_{x}\xi(r,x) \right\vert
\,dr\label{T0}\\
&  \quad\leq\left\Vert \varphi^{\prime}\right\Vert _{C([a,b])}\left(
1+\left\Vert w^{\prime}\right\Vert _{C([a,b])}\right)  \exp\left(  \left\Vert
w^{\prime\prime}\right\Vert _{C([a,b])}+1\right)  .\nonumber
\end{align}
This implies that%
\begin{align*}
&  \left\vert \partial_{x}\xi(t,x) -1\right\vert \\
&  \leq\int_{0}^{t} \big|[w^{\prime\prime}(\xi(r,x) )\varphi(\xi(r,x) )+ (
\eta(r,x) -w(\xi(r,x) )) \varphi^{\prime\prime}(\xi(r,x) )]\partial_{x}%
\xi(r,x)\\
&  \quad+\varphi^{\prime}(\xi(r,x) )\partial_{x}\eta( r,x) \big|\,dr\\
&  \leq\left\Vert \varphi\right\Vert _{C([a,b])}\left(  \left\Vert
w^{\prime\prime}\right\Vert _{C([a,b])}+\left\Vert \varphi^{\prime\prime
}\right\Vert _{C([a,b])}\right) \\
&  \quad\times\left(  1+\left\Vert w^{\prime}\right\Vert _{C([a,b])}\right)
\exp\left(  \left\Vert w^{\prime\prime}\right\Vert _{C([a,b])}+1\right) \\
&  \quad+\left\Vert \varphi^{\prime}\right\Vert _{C([a,b])}\left(
1+\left\Vert w^{\prime}\right\Vert _{C([a,b])}\right)  \exp\left(  \left\Vert
w^{\prime\prime}\right\Vert _{C([a,b])}+1\right)  .
\end{align*}
Since $\left\Vert \varphi\right\Vert _{C^{2}([a,b])}<1$, this concludes the proof.
\end{proof}

\begin{theorem}
\label{theorem norm of g and h}Let $\varphi$ and $w$ be as above with
$\left\Vert \varphi\right\Vert _{C^{2}([a,b])}<1$. Then%
\begin{align}
&  \left\vert \partial_{x}g(s,x) -1\right\vert \leq C_{0}\left\Vert
\varphi\right\Vert _{C^{2}([a,b])},\label{g x estimate}\\
&  \left\vert \partial_{x}h(s,x) -w^{\prime}(x) \right\vert \leq
C_{0}\left\Vert \varphi\right\Vert _{C^{2}([a,b])} \label{h x estimate}%
\end{align}
for all $(s,x) \in[0,1]\times\mathbb{R}$, where $C_{0}>0$ depends on
$\left\Vert w\right\Vert _{C^{2}([a,b])}$.
\end{theorem}

\begin{proof}
The proof is subdivided into three steps.\smallskip

\noindent\textbf{Step 1:} By (\ref{partial x of g (new)}) in the Appendix, for
$s\geq0$ and $\varphi(x)>0$ (the case $\varphi(x)<0$ is similar), we have%
\begin{align}
&  \partial_{x}g(s,x)-1=\frac{\partial_{x}\xi(t_{0}(s,x),x)-1}{1+\left[
w^{\prime}(g(s,x))+s\varphi^{\prime}(g(s,x))\right]  ^{2}}\nonumber\\
&  \quad+\frac{\left[  w^{\prime}(g(s,x))+s\varphi^{\prime}(g(s,x))\right]
\left[  w^{\prime}(x)-w^{\prime}(g(s,x))-s\varphi^{\prime}(g(s,x))\right]
}{1+\left[  w^{\prime}(g(s,x))+s\varphi^{\prime}(g(s,x))\right]  ^{2}%
}\label{300}\\
&  \quad+\frac{w^{\prime}(g(s,x))+s\varphi^{\prime}(g(s,x))}{1+\left[
w^{\prime}(g(s,x))+s\varphi^{\prime}(g(s,x))\right]  ^{2}}[\partial_{x}%
\eta(t_{0}(s,x),x)-w^{\prime}(x)].\nonumber
\end{align}
By the mean value theorem, (\ref{IVP}), (\ref{g}), the facts that
$t_{0}(s,x)\leq1$ and $w(\xi)\leq\eta\leq w(\xi)+s\varphi(\xi)$ (see Theorem~\ref{theorem existence t0}), we obtain%
\begin{align}
|w^{\prime}(x)-w^{\prime}(g(s,x))|  &  =|w^{\prime\prime}%
(c)(x-g(s,x))|\nonumber\\
&  \leq\left\Vert w^{\prime\prime}\right\Vert _{C([a,b])}\int_{0}^{t_{0}%
(s,x)}|w^{\prime}(\xi(r,x))\varphi(\xi(r,x))\label{302}\\
&  \quad+\left(  \eta(r,x)-w(\xi(r,x))\right)  \varphi^{\prime}(\xi
(r,x))|\,dr\nonumber\\
&  \leq\left\Vert w^{\prime\prime}\right\Vert _{C([a,b])}(\left\Vert
w^{\prime}\right\Vert _{C([a,b])}+\left\Vert \varphi^{\prime}\right\Vert
_{C([a,b])})\left\Vert \varphi\right\Vert _{C([a,b])}.\nonumber
\end{align}
Hence, from (\ref{300}) and Proposition~\ref{proposition partial x}, we deduce that%
\begin{align*}
|\partial_{x}g(s,x)-1|  &  \leq|\partial_{x}\xi(t_{0}(s,x),x)-1|+|w^{\prime
}(x)-w^{\prime}(g(s,x))|+|\varphi^{\prime}(g(s,x))|\\
&  \quad+|\partial_{x}\eta(t_{0}(s,x),x)-w^{\prime}(x)|\leq C\left\Vert
\varphi\right\Vert _{C^{2}([a,b])}.
\end{align*}

\noindent\textbf{Step 2:} If $s\geq0$ and $\varphi(x)=0$, then by
(\ref{partial x of g at x=a}) in the Appendix,
\begin{align*}
&  |\partial_{x}g(s,x)-1|=1-\frac{\sqrt{1+\left[  w^{\prime}(x)\right]  ^{2}}%
}{\sqrt{1+\left[  w^{\prime}(x)+s\varphi^{\prime}(x)\right]  ^{2}}}\\
&  \quad=\frac{\left[  w^{\prime}(x)+s\varphi^{\prime}(x)\right]  ^{2}-\left[
w^{\prime}(x)\right]  ^{2}}{\left(  \sqrt{1+\left[  w^{\prime}(x)\right]
^{2}}+\sqrt{1+\left[  w^{\prime}(x)+s\varphi^{\prime}(x)\right]  ^{2}}\right)
\sqrt{1+\left[  w^{\prime}(x)+s\varphi^{\prime}(x)\right]  ^{2}}}\\
&  \quad\leq2\left\Vert w^{\prime}\right\Vert _{C([a,b])}\left\Vert
\varphi^{\prime}\right\Vert _{C([a,b])}+\left\Vert \varphi^{\prime}\right\Vert
_{C([a,b])}^{2},
\end{align*}
so that (\ref{g x estimate}) holds even in this case.\smallskip

\noindent\textbf{Step 3:} To conclude the proof, note that by (\ref{h}),%
\[
\partial_{x}h(s,x)=[w^{\prime}(g(s,x))+s\varphi^{\prime}(g(s,x))]\partial
_{x}g(s,x),
\]
and so by (\ref{g x estimate}) and (\ref{302}), we deduce that
\begin{align*}
|\partial_{x}h(s,x)-w^{\prime}(x)|  &  \leq|w^{\prime}(g(s,x))+s\varphi
^{\prime}(g(s,x))||\partial_{x}g(s,x)-1|\\
&  \quad+|w^{\prime}(g(s,x))-w^{\prime}(x)|+|\varphi^{\prime}(g(s,x))|\leq
C\left\Vert \varphi\right\Vert _{C^{2}([a,b])},
\end{align*}
which proves (\ref{h x estimate}).
\end{proof}

We are now ready to construct the family of diffeomorphisms.\medskip

\begin{proof}
[Proof of Theorem~\ref{theorem diffeomorphism}]For every $s\in\lbrack0,1]$, we
define $\Psi_{s}:\mathbb{R}^{2}\rightarrow\mathbb{R}^{2}$ by%
\begin{equation}
\Psi_{s}(x,y):=(g(s,x),h(s,x)+y-w(x)), \label{psi s}%
\end{equation}
where $g$ and $h$ are the functions given in (\ref{g}) and (\ref{h definition}%
), respectively. By Theorem~\ref{theorem regularity g} and (\ref{h}),
$\Psi_{s}$ is of class $C^{2}\left(  [0,1]\times\mathbb{R}^{2}\right)  $.
Moreover, by (\ref{g at s=0}) and (\ref{h at s=0}),%
\[
\Psi_{0}(x,y)=(x,y),
\]
which implies, in particular, that $\Psi_{0}(\Gamma)=\Gamma$, while by Theorem~\ref{theorem graph} and the fact that $g(s,x)=x$ and $h(s,x)=w(x)$ for all
$x\in\mathbb{R}\setminus\lbrack a,b]$ by (\ref{g at s=0}) and (\ref{h at s=0}%
), it follows that
\[
\Psi_{s}(\Gamma)=\left\{  \Psi_{s}(x,w(x)):\,x\in(-1,1)\right\}  =\left\{
(x,w(x)+s\varphi(x)):\,x\in(-1,1)\right\}
\]
for every $s\in\lbrack0,1]$.

Since $\min_{[-1,1]}w>0$ by (\ref{w positive}), let
\begin{equation}
0<2L<\min_{[-1,1]}w,\quad M>\max_{[-1,1]}w. \label{L and M}%
\end{equation}
We now modify $\Psi_{s}$ to obtain a diffeomorphism in $\mathbb{R}^{2}$ which
coincides with the identity outside the open set $U:=\left(  a,b\right)
\times\left(  L,M+2\right)$.
Given
\begin{equation}
0<\delta_{0}<\min\left\{  1,L\right\}  , \label{delta 0}%
\end{equation}
construct a function $\lambda\in C_{c}^{\infty}\left(  \mathbb{R}\right)  $
such that $0\leq\lambda\leq1$, $\lambda\left(  y\right)  =1$ if $L+\delta
_{0}\leq y\leq M+2-\delta_{0}$, $\lambda\left(  y\right)  =0$ if $y\geq M+2$
or $y\leq L$, and $\left\vert \lambda^{\prime}\left(  y\right)  \right\vert
\leq2/\delta_{0}$ for all $y\in\mathbb{R}$. For every $s\in\left[  0,1\right]
$ and $(x,y)\in\mathbb{R}^{2}$, define
\begin{equation}
\Phi_{s}(x,y):=\lambda(y)\Psi_{s}(x,y)+(1-\lambda(y))(x,y). \label{phi s}%
\end{equation}
Then
\begin{equation}
D\Phi_{s}(x,y)=I_{2\times2}+\lambda(y)(D\Psi_{s}(x,y)-I_{2\times2})+(\Psi
_{s}(x,y)-(x,y))\otimes(0,\lambda^{\prime}(y)), \label{D phi s}%
\end{equation}
and so%
\[
\left\vert D\Phi_{s}(x,y)-I_{2\times2}\right\vert \leq\left\vert D\Psi
_{s}(x,y)-I_{2\times2}\right\vert +\frac{2}{\delta_{0}}\left\vert \Psi
_{s}(x,y)-\left(  x,y\right)  \right\vert .
\]
By Theorem~\ref{theorem norm of g and h} and (\ref{psi s}), we have
\[
\left\vert D\Psi_{s}(x,y)-I_{2\times2}\right\vert \leq C_{0}\left\Vert
\varphi\right\Vert _{C^{2}([a,b])},
\]
while by (\ref{g at s=0}), (\ref{h at s=0}), and Theorem~\ref{theorem norm of g and h}, for $x\in\left(  a,b\right)  $,%
\begin{align*}
\left\vert \Psi_{s}(x,y)-\left(  x,y\right)  \right\vert  &  \leq2\left\vert
g(s,x)-x\right\vert +2\left\vert h(s,x)-w(x)\right\vert \\
&  =2\left\vert \int_{a}^{x}\left(  \partial_{x}g\left(  s,r\right)
-1\right)  \,dr\right\vert +2\left\vert \int_{a}^{x}\left(  \partial
_{x}h\left(  s,r\right)  -w^{\prime}(r)\right)  \,dr\right\vert \\
&  \leq C_{0}\left\Vert \varphi\right\Vert _{C^{2}([a,b])},
\end{align*}
while for $x\notin\left(  a,b\right)  $, $\Psi_{s}(x,y)=\left(  x,y\right)  $
by (\ref{g at s=0}) and (\ref{h at s=0}), since $\varphi=0$ outside $\left(
a,b\right)  $. Hence, for all $(x,y)\in\mathbb{R}^{2}$, we deduce that%
\begin{equation}
\left\vert D\Phi_{s}(x,y)-I_{2\times2}\right\vert \leq C_{0}\left(  1+\frac
{1}{\delta_{0}}\right)  \left\Vert \varphi\right\Vert _{C^{2}([a,b])}<1,
\label{D phi s -I}%
\end{equation}
provided $\left\Vert \varphi\right\Vert _{C^{2}([a,b])}<\delta_{0}^{2}$ and
\[
0<\delta_{0}<\frac{1}{2C_{0}}.
\]
This implies that $\Phi_{s}$ is invertible in $\mathbb{R}^{2}$. It follows by
the inverse function theorem that $\Phi_{s}\left(  \mathbb{R}^{2}\right)  $ is
open and $\left(  \Phi_{s}\right)  ^{-1}$ is of class $C^{2}$.

Moreover, we have already seen that $\Psi_{s}(x,y)=\left(  x,y\right)  $ for
$x\notin\left(  a,b\right)  $, and so, again by (\ref{phi s}), $\Phi
_{s}(x,y)=\left(  x,y\right)  $ for $x\notin\left(  a,b\right)  $. This shows
that $\Phi_{s}$ is the identity outside $U$. In particular, $\Phi_{s}\left(
\partial U\right)  =\partial U$ and $\Phi_{s}\left(  U\right)  \subset U$.

Finally, we observe that
\[
\Phi_{s}=\Psi_{s}\quad\text{on }\Gamma_{s}%
\]
since, by (\ref{L and M}) and (\ref{delta 0}),
\[
L+\delta_{0}\leq2L-\left\Vert \varphi\right\Vert _{C^{0}([a,b])}\leq
w(x)+s\varphi(x)\leq M+\left\Vert \varphi\right\Vert _{C^{0}([a,b])}\leq
M+2-\delta_{0},
\]
provided
\[
\left\Vert \varphi\right\Vert _{C^{0}([a,b])}\leq\min\{1,L-\delta_{0}\}.
\]

To conclude the proof, it remains to show (\ref{tangential part equal zero}).
By (\ref{phi s}), (\ref{D phi s}), and the fact that $\lambda\left(  y\right)
=1$ if $L+\delta_{0}\leq y\leq M+2-\delta_{0}$, we have that
\begin{align*}
&  \dot{\Phi}_{s}(x,w(x))\cdot(D\Phi_{s}(x,w(x))\tau(x,w(x)))\\
&  =\dot{\Psi}_{s}(x,w(x))\cdot(D\Psi_{s}(x,w(x))\tau(x,w(x)))\\
&  =\left(  \partial_{s}g(s,x),\partial_{s}h(s,x)\right)  \cdot\left(
\begin{array}
[c]{cc}%
\partial_{x}g(s,x) & 0\\
\partial_{x}h(s,x)-w^{\prime}(x) & 1
\end{array}
\right)  \frac{\left(  1,w^{\prime}(x)\right)  }{\sqrt{1+\left(  w^{\prime
}(x)\right)  ^{2}}},
\end{align*}
provided $\left\Vert \varphi\right\Vert _{C^{2,\alpha}(a,b)}$ is sufficiently
small. Hence, by (\ref{X and Z}), (\ref{tangential part equal zero}) is
equivalent to
\begin{equation}
\partial_{x}g(s,x)\partial_{s}g(s,x)+\partial_{x}h(s,x)\partial_{s}h(s,x)=0
\label{i equiv}%
\end{equation}
for every $(s,x)\in[0,1]\times[ a,b]$.

Differentiating (\ref{h}) with respect to $x$ and $s$, respectively, yields
\begin{equation}
\partial_{x}h=(w^{\prime}(g)+s\varphi^{\prime}(g))\partial_{x}g,\quad
\partial_{s}h=(w^{\prime}(g)+s\varphi^{\prime}(g))\partial_{s}g+\varphi(g),
\label{partial h}%
\end{equation}
so that
\[
\partial_{x}g\partial_{s}g+\partial_{x}h\partial_{s}h=[(1+(w^{\prime
}(g)+s\varphi^{\prime}(g))^{2})\partial_{s}g+(w^{\prime}(g)+s\varphi^{\prime
}(g))\varphi(g)]\partial_{x}g,
\]
which is equal to $0$ by (\ref{partial s of g (new)}),
(\ref{partial g outside [a,b]}), and (\ref{partial s of g at a}) in the Appendix.
\end{proof}

\section{Proof of Theorem~\ref{theorem main}}

\label{section proof main}

To prove Theorem~\ref{theorem main} we first establish a minimality property with
respect to special variations of the domain $\Omega_{+}$. To be precise, we
will show the following result.

\begin{theorem}
\label{theorem second variation positive} Under the assumptions of
Theorem~\ref{theorem main}, there exists $\delta_{1}>0$ such that for all
polynomials $\varphi:[a,b]\rightarrow\mathbb{R}$ satisfying
(\ref{f and itz derivatives zero at endpoints}), extended to be zero outside
$[a,b]$ and with $\left\Vert \varphi\right\Vert _{C^{2,\alpha}(a,b)}\leq
\delta_{1}$,%
\[
\mathcal{F}(u)\leq\mathcal{F}(v)
\]
for every $v\in\mathcal{A}$ such that $\{v>0\}=\Phi_{1}(\{u>0\})$, where
$\{\Phi_{s}\}_{s\in\lbrack0,1]}$ is the admissible flow given in Theorem~\ref{theorem diffeomorphism}.
\end{theorem}

We begin with some preliminary estimates.

\begin{proposition}
\label{proposition estimates us}
Let $Q\in C^{1,1}(\Omega)$, let $\varphi$ be as in Theorem~\ref{theorem second variation positive}, and let $u_{s}$ be the solution to
problem (\ref{equation u t}), where $\Phi_{s}$ is given by (\ref{phi s}).
Then
\[
\big|Q^{2}(x,w(x)+s\varphi(x))-|\nabla u_{s}(x,w(x)+s\varphi(x))|^{2}\big|\leq
C\left\Vert \varphi\right\Vert _{C^{2,\alpha}(a,b)}%
\]
and%
\[
\big|\partial_{x}\big(Q^{2}(x,w(x)+s\varphi(x))-|\nabla u_{s}(x,w(x)+s\varphi
(x))|^{2}\big)\big|\leq C\left\Vert \varphi\right\Vert _{C^{2,\alpha}(a,b)}%
\]
for all $x\in\lbrack-1,1]$, where $C$ depends only on $\left\Vert w\right\Vert
_{C^{2,\alpha}(-1,1)}$ and $\left\Vert u\right\Vert _{C^{2,\alpha}(\Omega
_{+})}$.
\end{proposition}

\begin{proof}
The proof is subdivided into three steps.\smallskip

\noindent\textbf{Step 1:} Recall that the function $\hat{u}_{s}:=u_{s}%
\circ\Phi_{s}$ satisfies the boundary value problem (\ref{s3}) with
coefficients $A_{s}$ given by (\ref{As}). Using the matrix expansion
\[
(I_{2\times2}+B)^{-1}=I_{2\times2}-B+o(|B|),
\]
it follows from (\ref{D phi s -I}) that the matrix $B_{s}:=A_{s}-I_{2\times2}$
satisfies%
\begin{equation}
|B_{s}|\leq C\left\Vert \varphi\right\Vert _{C^{2}([a,b])}\leq\frac{1}{2},
\label{s4}%
\end{equation}
provided $\left\Vert \varphi\right\Vert _{C^{2}([a,b])}$ is sufficiently
small. In turn, the matrix $A_{s}$ is positive definite uniformly with respect
to $s$. Using (\ref{s4}), by (\ref{s3}) and Poincar\'{e}
inequality in the Lipschitz domain $\Omega_{+}$ we obtain
\begin{equation}
\left\Vert \hat{u}_{s}\right\Vert _{H^{1}(\Omega_{+})}\leq C\left\Vert
u\right\Vert _{H^{1}(\Omega_{+})}, \label{s5}%
\end{equation}
where $C>0$ depends on $\Omega_{+}$ but not on $s$.
On the other hand, by (\ref{equation u}) and (\ref{s3}) we have
\[
\left\{
\begin{array}
[c]{ll}%
\operatorname{div}(A_{s}\nabla(\hat{u}_{s}-u))=-\operatorname{div}(B_{s}\nabla
u) & \text{in }\Omega_{+},\\
\hat{u}_{s}-u=0 & \text{on }\Gamma\cup(\{y=0\}\cap\partial\Omega_{+}),
\end{array}
\right.
\]
with $(\hat{u}_{s}-u)(-1,y)=(\hat{u}_{s}-u)(1,y)$ for all $(\pm1,y)\in
\overline{\Omega_{+}}$. Hence, with similar estimates, it follows from
(\ref{s4}) that
\begin{equation}
\left\Vert \hat{u}_{s}-u\right\Vert _{H^{1}(\Omega_{+})}\leq C\left\Vert
B_{s}\right\Vert _{C^{0}(\Omega_{+})}\left\Vert u\right\Vert _{H^{1}%
(\Omega_{+})}\leq C\left\Vert \varphi\right\Vert _{C^{2}([a,b])}\left\Vert
u\right\Vert _{H^{1}(\Omega_{+})}. \label{s6}%
\end{equation}
Using the fact that $\hat{u}_{s}(x,y)=u_{s}(x,y)$ for all $y<L$, where $L$ is
given in (\ref{L and M}), by (\ref{s5}) and (\ref{s6}) we have
\begin{equation}
\left\Vert u_{s}\right\Vert _{H^{1}((-1,1)\times(0,L))}\leq C,\quad\left\Vert
u_{s}-u\right\Vert _{H^{1}((-1,1)\times(0,L))}\leq C\left\Vert \varphi
\right\Vert _{C^{2}([a,b])}, \label{s7}%
\end{equation}
where $C$ depends on $\left\Vert w\right\Vert _{C^{1}(-1,1)}$ and $\left\Vert
u\right\Vert _{H^{1}(\Omega_{+})}$. By \cite[Theorem~9.13]{GilbargTrud} and
(\ref{s7}),%
\[
\left\Vert u_{s}\right\Vert _{H^{2}((-1,1)\times(\varepsilon_{0}%
,6\varepsilon_{0}))}\leq C,\quad\left\Vert u_{s}-u\right\Vert _{H^{2}%
((-1,1)\times(\varepsilon_{0},6\varepsilon_{0}))}\leq C\left\Vert
\varphi\right\Vert _{C^{2}([a,b])}%
\]
for $0<\varepsilon_{0}<L/6$. Since $u_{s}$ and $u$ are periodic in the $x$
variable, they are still harmonic in $\mathbb{R}\times(\varepsilon
_{0},6\varepsilon_{0})$ and satisfy
\[
\left\Vert u_{s}\right\Vert _{H^{2}((a^{\prime},b^{\prime})\times
(\varepsilon_{0},6\varepsilon_{0}))}\leq C,\quad\left\Vert u_{s}-u\right\Vert
_{H^{2}((a^{\prime},b^{\prime})\times(\varepsilon_{0},6\varepsilon_{0}))}\leq
C\left\Vert \varphi\right\Vert _{C^{2}([a,b])}%
\]
for some $a^{\prime}<-1<1<b^{\prime}$. Using \cite[Theorem~2.10]{GilbargTrud}
in the set $(a^{\prime},b^{\prime})\times(\varepsilon_{0},6\varepsilon_{0})$
we obtain that%
\begin{equation}
\left\Vert u_{s}\right\Vert _{C^{3}((-1,1)\times(2\varepsilon_{0}%
,5\varepsilon_{0}))}\leq C,\quad\left\Vert u_{s}-u\right\Vert _{C^{3}%
((-1,1)\times(2\varepsilon_{0},5\varepsilon_{0}))}\leq C\left\Vert
\varphi\right\Vert _{C^{2}([a,b])}, \label{s8}%
\end{equation}
where we invoked the continuous immersion of $H^{2}((a^{\prime},b^{\prime}%
)\times(\varepsilon_{0},6\varepsilon_{0}))$ into $C^{0}((a^{\prime},b^{\prime
})\times(\varepsilon_{0},6\varepsilon_{0}))$.\smallskip

\noindent\textbf{Step 2:} Let $\left\Vert \varphi\right\Vert _{C^{2}%
([a,b])}<\varepsilon_{0}$. By Theorem~\ref{theorem diffeomorphism} the
function%
\begin{equation}
v_{s}(x,y):=u_{s}(x,y+s\varphi(x)) \label{s0}%
\end{equation}
is well-defined in the set%
\begin{equation}
\Omega_{0}:=\Omega_{+}\cap((-1,1)\times(3\varepsilon_{0},\infty)),
\label{Omega 0}%
\end{equation}
and by (\ref{equation u t}) it satisfies the elliptic equation
\begin{equation}
\partial_{x}^{2}v_{s}+(1+(s\varphi^{\prime})^{2})\partial_{y}^{2}%
v_{s}-2s\varphi^{\prime}\partial_{xy}^{2}v_{s}-s\varphi^{\prime\prime}%
\partial_{y}v_{s}=0\quad\text{in }\Omega_{0}. \label{s0a}%
\end{equation}
Moreover, since $\varphi=0$ outside $[a,b]\subset(-1,1)$, we have
$v_{s}(-1,y)=u_{s}(-1,y)=u_{s}(1,y)=v_{s}(1,y)$. Hence, $v_{s}$ satisfies the
previous equation in $((a^{\prime},b^{\prime})\times(3\varepsilon_{0}%
,\infty))\cap\{u>0\}$, where $u$ has been extended periodically and
$a^{\prime}<-1<1<b^{\prime}$. Moreover, $v_{s}=0$ on $\Gamma$ by
(\ref{equation u t}) and (\ref{Phi s Gamma}), while $v_{s}(x,3\varepsilon
_{0})=u_{s}(x,3\varepsilon_{0}+s\varphi(x))$. Since $\left\Vert \varphi
\right\Vert _{C^{2}([a,b])}<\varepsilon_{0}$, we have that $(x,3\varepsilon
_{0}+s\varphi(x))\in(a^{\prime},b^{\prime})\times(2\varepsilon_{0}%
,4\varepsilon_{0})$.

By (\ref{D phi s -I}) and (\ref{s5}) we have that
\[
\left\Vert u_{s}\right\Vert _{H^{1}(\Phi_{s}(\Omega_{+}))}\leq C,
\]
where $C$ depends on $\Omega_{+}$ and $\left\Vert u\right\Vert _{H^{1}%
(\Omega_{+})}$. By the lateral periodicity of $u_{s}$, the same estimate holds
with $\Phi_{s}(\Omega_{+})$ replaced by $\Phi_{s}(((a^{\prime},b^{\prime
})\times(3\varepsilon_{0},\infty))\cap\{u>0\})$. In turn, by (\ref{s0}) and
the chain rule
\[
\left\Vert v_{s}\right\Vert _{H^{1}(((a^{\prime},b^{\prime})\times
(3\varepsilon_{0},\infty))\cap\{u>0\})}\leq C.
\]
It follows from \cite[Theorem~9.13]{GilbargTrud}, with $T$ the graph of $w$
restricted to $(a^{\prime},b^{\prime})$, that
\[
\left\Vert v_{s}\right\Vert _{H^{2}(((a^{\prime\prime},b^{\prime\prime}%
)\times(4\varepsilon_{0},\infty))\cap\{u>0\})}\leq C
\]
for $a^{\prime}<a^{\prime\prime}<-1<1<b^{\prime\prime}<b^{\prime}$. By the
continuous immersion of $H^{2}(((a^{\prime\prime},b^{\prime\prime}%
)\times(4\varepsilon_{0},\infty))\cap\{u>0\})$ into $C^{0,\alpha}%
(((a^{\prime\prime},b^{\prime\prime})\times(4\varepsilon_{0},\infty
))\cap\{u>0\})$, we have
\[
\left\Vert v_{s}\right\Vert _{C^{0,\alpha}(((a^{\prime\prime},b^{\prime\prime
})\times(4\varepsilon_{0},\infty))\cap\{u>0\})}\leq C.
\]
By \cite[Corollary~6.7]{GilbargTrud}, with $T$ the graph of $w$ restricted to
$(a^{\prime\prime},b^{\prime\prime})$, and using a covering argument, we obtain that
there exists an $\varepsilon_{1}$-neighborhood $\Gamma_{1}$ of $\Gamma$ such that
\begin{equation}
\left\Vert v_{s}\right\Vert _{C^{2,\alpha}(\Gamma_{1}\cap\Omega_{0})}\leq C
\label{s9}%
\end{equation}
for some $0<\varepsilon_{1}<\varepsilon_{0}$. By (\ref{s8}) and the chain rule,
we have that
\begin{equation}
\left\Vert v_{s}\right\Vert _{C^{2,\alpha}((-1,1)\times(3\varepsilon
_{0},4\varepsilon_{0}))}\leq C. \label{s10}%
\end{equation}
In the remaining set we can now use the interior Schauder's estimate in
\cite[Corollary~6.3]{GilbargTrud} to conclude, also by (\ref{s9}) and
(\ref{s10}), that there exists a constant $C$ depending only on $\left\Vert
w\right\Vert _{C^{2,\alpha}(-1,1)}$ and $\left\Vert u\right\Vert
_{H^{1}(\Omega_{+})}$ such that
\begin{equation}
\left\Vert v_{s}\right\Vert _{C^{2,\alpha}(\Omega_{0})}\leq C \label{s1}%
\end{equation}
for all $s\in(0,1)$.

By (\ref{equation u}) and (\ref{s0a}),
\[
\Delta(v_{s}-u)=-(s\varphi^{\prime})^{2}\partial_{y}^{2}v_{s}+2s\varphi
^{\prime}\partial_{xy}^{2}v_{s}+s\varphi^{\prime\prime}\partial_{y}v_{s}%
\quad\text{in }\Omega_{0}.
\]
Since $v_{s}-u=0$ on $\Gamma$, we can argue as in Step~1, and from standard estimates,
Poincar\'{e} inequality, (\ref{s1}), and the fact that $\left\Vert
\varphi\right\Vert _{C^{2}([a,b])}<\varepsilon_{0}$, we obtain 
\begin{align*}
\left\Vert v_{s}-u\right\Vert _{H^{1}(\Omega_{0})}  &  \leq C\left\Vert
\varphi\right\Vert _{C^{2}([a,b])}+C\left\Vert v_{s}-u\right\Vert
_{C^{1}((-1,1)\times\{3\varepsilon_{0}\})}\\
&  \leq C\left\Vert \varphi\right\Vert _{C^{2}([a,b])}%
\end{align*}
where the last inequality follows from the chain rule and (\ref{s8}). By the
lateral periodicity of $v_{s}$ and $u$, the same estimate holds with
$\Omega_{0}$ replaced by $((a^{\prime},b^{\prime})\times(3\varepsilon
_{0},\infty))\cap\{u>0\}$. Again by \cite[Theorem~9.13]{GilbargTrud}, with $T$
the graph of $w$ restricted to $(a^{\prime},b^{\prime})$, we deduce that
\begin{align*}
\left\Vert v_{s}-u\right\Vert _{H^{2}(((a^{\prime\prime},b^{\prime\prime
})\times(4\varepsilon_{0},\infty))\cap\{u>0\})}  &  \leq C\left\Vert
v_{s}-u\right\Vert _{H^{1}(((a^{\prime},b^{\prime})\times(3\varepsilon
_{0},\infty))\cap\{u>0\})}\\
&  \quad+C\left\Vert \varphi\right\Vert _{C^{2}([a,b])}\left\Vert
v_{s}\right\Vert _{H^{2}(((a^{\prime},b^{\prime})\times(3\varepsilon
_{0},\infty))\cap\{u>0\})}\\
&  \leq C\left\Vert \varphi\right\Vert _{C^{2}([a,b])}%
\end{align*}
for $a^{\prime}<a^{\prime\prime}<-1<1<b^{\prime\prime}<b^{\prime}$, and where
we have used the previous inequality and (\ref{s1}), which holds in
$((a^{\prime\prime},b^{\prime\prime})\times(3\varepsilon_{0},\infty
))\cap\{u>0\}$ by lateral periodicity.

By \cite[Corollary~6.7]{GilbargTrud}, with $T$ the graph of $w$ restricted to
$(a^{\prime\prime},b^{\prime\prime})$, and a covering argument, we have that
\begin{align*}
\left\Vert v_{s}-u\right\Vert _{C^{2,\alpha}(\Gamma_{1}\cap\Omega_{0})}  &
\leq C\left\Vert v_{s}-u\right\Vert _{C^{0}(((a^{\prime\prime},b^{\prime
\prime})\times(3\varepsilon_{0},\infty))\cap\{u>0\})}\\
&  \quad+C\left\Vert \varphi\right\Vert _{C^{2,\alpha}([a,b])}\left\Vert
v_{s}\right\Vert _{C^{2,\alpha}(((a^{\prime\prime},b^{\prime\prime}%
)\times(3\varepsilon_{0},\infty))\cap\{u>0\})}\\
&  \leq C\left\Vert \varphi\right\Vert _{C^{2}([a,b])},
\end{align*}
where we used the fact that the estimate (\ref{s1}) holds in $((a^{\prime
\prime},b^{\prime\prime})\times(3\varepsilon_{0},\infty))\cap\{u>0\}$ by
lateral periodicity. We can now continue as before using (\ref{s8}) and
\cite[Corollary~6.3]{GilbargTrud} to conclude that
\begin{equation}
\left\Vert v_{s}-u\right\Vert _{C^{2,\alpha}(\Omega_{0})}\leq C\left\Vert
\varphi\right\Vert _{C^{2,\alpha}(a,b)}. \label{s2b}%
\end{equation}

\noindent\textbf{Step 3:} Since
\begin{equation}
Q^{2}(x,w(x))-|\nabla u(x,w(x))|^{2}=0 \label{s2a}%
\end{equation}
for all $x\in\lbrack-1,1]$, by (\ref{equation u}) and (\ref{u stat}) it
follows that
\begin{align*}
\big|Q^{2}  &  (x,w(x)+s\varphi(x))-|\nabla u_{s}(x,w(x)+s\varphi
(x))|^{2}\big|\\
&  \leq|Q^{2}(x,w(x)+s\varphi(x))-Q^{2}(x,w(x))|\\
&  \quad+\big||\nabla u_{s}(x,w(x)+s\varphi(x))|^{2}-|\nabla u(x,w(x))|^{2}%
\big|\\
&  \leq\left\Vert Q^{2}\right\Vert _{C^{1}}\left\Vert \varphi\right\Vert
_{C^{0}}+\big||\nabla u_{s}(x,w(x)+s\varphi(x))|^{2}-|\nabla u(x,w(x))|^{2}%
\big|.
\end{align*}
By (\ref{s0}), (\ref{s1}), (\ref{s2b}), and the chain rule, the last term on
the right-hand side can be estimated from above by%
\begin{align*}
&  C(\left\Vert \nabla u_{s}\right\Vert _{C^{0}}+\left\Vert \nabla
u\right\Vert _{C^{0}})|\nabla u_{s}(x,w(x)+s\varphi(x))-\nabla u(x,w(x))|\\
&  \leq C(\left\Vert \nabla u_{s}\right\Vert _{C^{0}}+\left\Vert \nabla
u\right\Vert _{C^{0}})\big(|\nabla v_{s}(x,w(x))-\nabla u(x,w(x))|+\left\Vert
\nabla v_{s}\right\Vert _{C^{0}}|\varphi^{\prime}(x)|\big)\\
&  \leq C\left\Vert \varphi\right\Vert _{C^{2,\alpha}(a,b)},
\end{align*}
where, as before, $C$ depends only on $\left\Vert w\right\Vert _{C^{2,\alpha
}(-1,1)}$ and $\left\Vert u\right\Vert _{C^{1}(\Omega_{+})}$.
On the other hand by (\ref{s2a}),
\[
\partial_{x}\big(Q^{2}(x,w(x))-|\nabla u(x,w(x))|^{2}\big)=0
\]
for all $x\in\lbrack-1,1]$, and so%
\begin{align*}
\big|\partial_{x}\big(Q^{2}  &  (x,w(x)+s\varphi(x))-|\nabla u_{s}%
(x,w(x)+s\varphi(x))|^{2}\big)\big|\\
&  \leq\big|\partial_{x}\big(Q^{2}(x,w(x)+s\varphi(x))-Q^{2}%
(x,w(x))\big)\big|\\
&  \quad+\big|\partial_{x}\big(|\nabla u_{s}(x,w(x)+s\varphi(x))|^{2}-|\nabla
u(x,w(x))|^{2}\big)\big|\\
&  \leq C\left\Vert Q^{2}\right\Vert _{C^{1,1}}\left\Vert \varphi\right\Vert
_{C^{1}}+\big|\partial_{x}\big(|\nabla u_{s}(x,w(x)+s\varphi(x))|^{2}-|\nabla
u(x,w(x))|^{2}\big)\big|,
\end{align*}
where $C$ depends only on $\left\Vert w\right\Vert _{C^{1}(-1,1)}$. The last
term on the right-hand side can be estimated from above by%
\begin{align*}
&  C\left\Vert \nabla u_{s}\right\Vert _{C^{0}}|\nabla^{2}u_{s}%
(x,w(x)+s\varphi(x))-\nabla^{2}u(x,w(x))|\\
&  +C\left\Vert \nabla^{2}u\right\Vert _{C^{0}}|\nabla u_{s}(x,w(x)+s\varphi
(x))-\nabla u(x,w(x))|+C\left\Vert u\right\Vert _{C^{2}}^{2}|\varphi^{\prime
}(x)|,
\end{align*}
where, as before, $C$ depends only on $\left\Vert w\right\Vert _{C^{2,\alpha
}(-1,1)}$ and $\left\Vert u\right\Vert _{C^{2}(\Omega_{+})}$. By (\ref{s0})
and the chain rule, we have that
\begin{align*}
&  |\nabla^{2}u_{s}(x,w(x)+s\varphi(x))-\nabla^{2}u(x,w(x))|\\
&  \leq|\nabla^{2}v_{s}(x,w(x))-\nabla^{2}u(x,w(x))|+C\left\Vert
v_{s}\right\Vert _{C^{2,\alpha}}\left\Vert \varphi\right\Vert _{C^{2}%
([a,b])}\leq C\left\Vert \varphi\right\Vert _{C^{2,\alpha}(a,b)},
\end{align*}
where in the last inequality we used (\ref{s1}) and (\ref{s2b}). A similar
estimate holds for $|\nabla u_{s}(x,w(x)+s\varphi(x))-\nabla u(x,w(x))|$. This
concludes the proof.
\end{proof}

\begin{remark}
The proof of the previous proposition could be significantly simplified if we
could show that the diffeomorphism $\Phi_{s}$ is of class $C^{2,\alpha}$
rather than just $C^{2}$, and if we had uniform estimates on the $C^{2,\alpha}$
norm of $\Phi_{s}$ in terms of $\left\Vert w\right\Vert _{C^{2,\alpha}(-1,1)}$
and $\left\Vert \varphi\right\Vert _{C^{2,\alpha}(a,b)}$.
Indeed, the $C^{2,\alpha}$ bounds on $u_s$ and $v_s$ would follow in this case
from standard elliptic estimates. 
\end{remark}

Next we estimate the second integral on the right-hand side of (\ref{2var}).

\begin{proposition}
\label{proposition estimate 1}Let $Q\in C^{1,1}(\Omega)$, let $\varphi$ be as in Theorem~\ref{theorem second variation positive}, and let $u_{s}$ be the solution to
problem (\ref{equation u t}), where $\Phi_{s}$ is given by (\ref{phi s}). Then
there exists $C>0$, depending only on $\left\Vert w\right\Vert _{C^{2,\alpha
}(-1,1)}$ and $\left\Vert u\right\Vert _{C^{2}(\Omega_{+})}$, such that for
every $s\in\lbrack0,1]$ and every $\psi\in C(\Gamma_{s})$,%
\begin{equation}
\Big|\int_{\Gamma_{s}}\kappa_{s}(\partial_{\nu_{s}}u_{s})^{2}\psi
^{2}d\mathcal{H}^{1}-\int_{\Gamma}\kappa(\partial_{\nu}u)^{2}\psi^{2}\circ
\Phi_{s}\,d\mathcal{H}^{1}\Big|\leq C\left\Vert \varphi\right\Vert
_{C^{2}([a,b])}\int_{\Gamma_{s}}\psi^{2}d\mathcal{H}^{1} \label{stima 100}%
\end{equation}
and%
\begin{equation}
\Big|\int_{\Gamma_{s}}\partial_{\nu_{s}}Q^{2}~\psi^{2}d\mathcal{H}^{1}%
-\int_{\Gamma}\partial_{\nu}Q^{2}~\psi^{2}\circ\Phi_{s}\,d\mathcal{H}%
^{1}\Big|\leq C\left\Vert \varphi\right\Vert _{C^{1}([a,b])}\int_{\Gamma_{s}%
}\psi^{2}\,d\mathcal{H}^{1}. \label{stima 101}%
\end{equation}

\end{proposition}

\begin{proof}
Let $v_{s}$ and $\Omega_{0}$ be defined as in (\ref{s0}) and (\ref{Omega 0}).
Then, by (\ref{s1}) and (\ref{s2b}),
\begin{equation}
\left\Vert v_{s}\right\Vert _{C^{2,\alpha}(\Omega_{0})}\leq C,\quad\left\Vert
v_{s}-u\right\Vert _{C^{2,\alpha}(\Omega_{0})}\leq C\left\Vert \varphi
\right\Vert _{C^{2,\alpha}(a,b)} \label{s11}%
\end{equation}
for some constant $C>0$ depending only on $\left\Vert w\right\Vert
_{C^{2,\alpha}(-1,1)}$ and $\left\Vert u\right\Vert _{C^{2}(\Omega_{+})}$. By
the chain rule,%
\begin{align*}
\partial_{\nu_{s}}  &  u_{s}(x,w+s\varphi)=\nabla u_{s}(x,w+s\varphi)\cdot
\nu_{s}(x,w+s\varphi)\\
&  =\nabla v_{s}(x,w)\cdot\nu(x,w)-s\varphi^{\prime}\partial_{y}%
v_{s}(x,w)e_{1}\cdot\nu_{s}(x,w+s\varphi)\\
&  \quad+\nabla v_{s}(x,w)\cdot(\nu_{s}(x,w+s\varphi)-\nu(x,w)).
\end{align*}
Using (\ref{s11}),%
\begin{equation}
|\partial_{\nu_{s}}u_{s}(x,w+s\varphi)-\partial_{\nu}u(x,w)|\leq|\partial
_{\nu}v_{s}(x,w)-\partial_{\nu}u(x,w)|+C\left\Vert \varphi\right\Vert
_{C^{1}([a,b])}, \label{s12}%
\end{equation}
where to estimate $|\nu_{s}(x,w+s\varphi)-\nu(x,w)|$ we used the fact that the function $t\mapsto\frac{1}{\sqrt{1+t^{2}}}$ is
$1$-Lipschitz. Similarly,
\begin{align}
|\kappa_{s}(x,w+s\varphi)-\kappa(x,w)|  &  =\Big|\frac{w^{\prime\prime
}+s\varphi^{\prime\prime}}{(1+(w^{\prime}+s\varphi^{\prime})^{2})^{3/2}}%
-\frac{w^{\prime\prime}}{(1+(w^{\prime})^{2})^{3/2}}\Big|\label{s13}\\
&  \leq C\left\Vert \varphi\right\Vert _{C^{2}([a,b])}.\nonumber
\end{align}
Combining (\ref{s11}), (\ref{s12}), and (\ref{s13}), and using a change of
variable, we obtain (\ref{stima 100}).

On the other hand,%
\begin{align*}
\partial_{\nu_{s}}Q^{2}(x,w+s\varphi)-\partial_{\nu}Q^{2}(x,w)  &  =\nabla
Q^{2}(x,w+s\varphi)\cdot(\nu_{s}(x,w+s\varphi)-\nu(x,w))\\
&  \quad+(\nabla Q^{2}(x,w+s\varphi)-\nabla Q^{2}(x,w))\cdot\nu(x,w),
\end{align*}
and so%
\[
|\partial_{\nu_{s}}Q^{2}(x,w+s\varphi)-\partial_{\nu}Q^{2}(x,w)|\leq
C\left\Vert \varphi\right\Vert _{C^{1}([a,b])},
\]
which gives (\ref{stima 101}).
\end{proof}

We now estimate the first integral on the right-hand side of (\ref{2var}).

\begin{proposition}
\label{proposition estimate 2}Let $\varphi$ be as in Theorem~\ref{theorem second variation positive}, 
and let $u_{s}$ be the solution to
problem (\ref{equation u t}), where $\Phi_{s}$ is given by (\ref{phi s}). Then
there exists $C>0$, depending only on $\left\Vert w\right\Vert _{C^{2,\alpha
}(-1,1)}$ and $\left\Vert u\right\Vert _{C^{2}(\Omega_{+})}$, such that for
every $s\in\lbrack0,1]$ and every $\psi\in C^{1}(\Gamma_{s})$,
\[
\Big|\int_{\Phi_{s}(\Omega_{+})}|\nabla u_{\psi}^{s}|^{2}\,d\boldsymbol{x}%
-\int_{\Omega_{+}}|\nabla u_{\psi\circ\Phi_{s}}^{0}|^{2}\,d\boldsymbol{x}%
\Big|\leq C\left\Vert \varphi\right\Vert _{C^{2,\alpha}(a,b)}\left\Vert
\psi\right\Vert _{H^{1/2}(\Gamma_{s})}^{2},
\]
where $u_{\psi}^{s}$ is the unique solution to the problem%
\[%
\begin{cases}
\Delta u_{\psi}^{s}=0 & \text{in }\Phi_{s}(\Omega_{+}),\\
u_{\psi}^{s}=-\psi\partial_{\nu_{s}}u_{s} & \text{on }\Gamma_{s},\\
u_{\psi}^{s}=0 & \text{on }\Phi_{s}(\{y=0\}\cap\partial\Omega_{+}).
\end{cases}
\]
with $u_{\psi}^{s}(-1,y)=u_{\psi}^{s}(1,y)$ for all $y$ such that $(\pm
1,y)\in\Phi_{s}(\overline{\Omega_{+}})$.
\end{proposition}

\begin{proof}
Reasoning as in the proof of (\ref{variant}), we have%
\[
\int_{\Phi_{s}(\Omega_{+})}|\nabla u_{\psi}^{s}|^{2}\,d\boldsymbol{x}%
=\int_{\Gamma_{s}}u_{\psi}^{s}\partial_{\nu_{s}}u_{\psi}^{s}\,d\mathcal{H}%
^{1}.
\]
Define%
\[
\hat{u}_{\psi}^{s}:=u_{\psi}^{s}\circ\Phi_{s}.
\]
Then $\hat{u}_{\psi}^{s}$ satisfies
\begin{equation}\label{ir03}
\left\{
\begin{array}
[c]{ll}%
\operatorname{div}(A_{s}\nabla\hat{u}_{\psi}^{s})=0 & \text{in }\Omega_{+},\\
\hat{u}_{\psi}^{s}=-(\psi\partial_{\nu_{s}}u_{s})\circ\Phi_{s} & \text{on
}\Gamma,\\
\hat{u}_{\psi}^{s}=0 & \text{on }\overline{\Omega_{+}}\cap\{y=0\},
\end{array}
\right.
\end{equation}
with $\hat{u}_{\psi}^{s}(-1,y)=\hat{u}_{\psi}^{s}(1,y)$ for all $y$ such that
$(\pm1,y)\in\overline{\Omega_{+}}$, where $A_{s}$ is given by (\ref{As}).
Multiplying the first equation in (\ref{ir03}) by $\hat{u}_{\psi}^{s}$ and by the divergence
theorem, we obtain%
\begin{align*}
\int_{\Omega_{+}}(A_{s}\nabla\hat{u}_{\psi}^{s})\cdot\nabla\hat{u}_{\psi}%
^{s}\,d\boldsymbol{x}  &  \leq\left\Vert (A_{s}\nabla\hat{u}_{\psi}^{s}%
)\cdot\nu\right\Vert _{H^{-1/2}(\Gamma)}\left\Vert (\psi\partial_{\nu_{s}%
}u_{s})\circ\Phi_{s}\right\Vert _{H^{1/2}(\Gamma)}\\
&  \leq C\left\Vert A_{s}\nabla\hat{u}_{\psi}^{s}\right\Vert _{L^{2}%
(\Omega_{+})}\left\Vert (\psi\partial_{\nu_{s}}u_{s})\circ\Phi_{s}\right\Vert
_{H^{1/2}(\Gamma)}\\
&  \leq C\left\Vert \nabla\hat{u}_{\psi}^{s}\right\Vert _{L^{2}(\Omega_{+}%
)}\left\Vert (\psi\partial_{\nu_{s}}u_{s})\circ\Phi_{s}\right\Vert
_{H^{1/2}(\Gamma)},
\end{align*}
where we used (\ref{s4}) and the continuity of the normal trace in the space
$H(\operatorname{div};\Omega_{+})$ (see, e.g., \cite[Section 3.2]%
{boyer-fabin}), and where the constant $C$ depends only on $\Omega
_{+}$. The previous estimate, together with (\ref{D phi s -I}) and (\ref{s4}),
implies that
\begin{align*}
\left\Vert \nabla\hat{u}_{\psi}^{s}\right\Vert _{L^{2}(\Omega_{+})}  &  \leq
C\left\Vert (\psi\partial_{\nu_{s}}u_{s})\circ\Phi_{s}\right\Vert
_{H^{1/2}(\Gamma)}\leq C\left\Vert \psi\partial_{\nu_{s}}u_{s}\right\Vert
_{H^{1/2}(\Gamma_{s})}\\
&  \leq C\left\Vert \psi\right\Vert _{L^{2}(\Gamma_{s})}\left\Vert
\partial_{\nu_{s}}u_{s}\right\Vert _{C^{0,1}(\Gamma_{s})}+C\left\vert
\psi\right\vert _{H^{1/2}(\Gamma_{s})}\left\Vert \partial_{\nu_{s}}%
u_{s}\right\Vert _{C^{0}(\Gamma_{s})}\\
&  \leq C\left\Vert \psi\right\Vert _{H^{1/2}(\Gamma_{s})},
\end{align*}
where in the last inequality we reasoned as in the proof of Proposition~\ref{proposition estimate 1} and used (\ref{s11}). By the Poincar\'{e} inequality, we get
\begin{equation}
\left\Vert \hat{u}_{\psi}^{s}\right\Vert _{H^{1}(\Omega_{+})}\leq C\left\Vert
\psi\right\Vert _{H^{1/2}(\Gamma_{s})}. \label{h1}%
\end{equation}

On the other hand,
\[%
\begin{cases}
\operatorname{div}(A_{s}\nabla(\hat{u}_{\psi}^{s}-u_{\psi\circ\Phi_{s}}%
^{0}))=-\operatorname{div}(B_{s}\nabla u_{\psi\circ\Phi_{s}}^{0}) & \text{in
}\Omega_{+},\\
\hat{u}_{\psi}^{s}-u_{\psi\circ\Phi_{s}}^{0}=-(\psi\partial_{\nu_{s}}%
u_{s})\circ\Phi_{s}+(\psi\circ\Phi_{s})\partial_{\nu}u & \text{on }\Gamma,\\
\hat{u}_{\psi}^{s}-u_{\psi\circ\Phi_{s}}^{0}=0 & \text{on }\{y=0\}\cap
\partial\Omega_{+},
\end{cases}
\]
with $(\hat{u}_{\psi}^{s}-u_{\psi\circ\Phi_{s}}^{0})(-1,y)=(\hat{u}_{\psi}%
^{s}-u_{\psi\circ\Phi_{s}}^{0})(1,y)$ for all $y$ such that $(\pm
1,y)\in\overline{\Omega_{+}}$, where $B_{s}=A_{s}-I_{2\times2}$. Reasoning as
before, by (\ref{s4}) and (\ref{h1}) we get%
\begin{align*}
&  \left\Vert \nabla(\hat{u}_{\psi}^{s}-u_{\psi\circ\Phi_{s}}^{0})\right\Vert
_{L^{2}(\Omega_{+})}^{2}\leq\left\Vert B_{s}\nabla u_{\psi\circ\Phi_{s}}%
^{0}\right\Vert _{L^{2}(\Omega_{+})}\left\Vert \nabla(\hat{u}_{\psi}%
^{s}-u_{\psi\circ\Phi_{s}}^{0})\right\Vert _{L^{2}(\Omega_{+})}\\
&  \quad+\left\Vert (A_{s}\nabla(\hat{u}_{\psi}^{s}-u_{\psi\circ\Phi_{s}}%
^{0})+B_{s}\nabla u_{\psi\circ\Phi_{s}}^{0})\cdot\nu\right\Vert _{H^{-1/2}%
(\Gamma)}\left\Vert (\psi\circ\Phi_{s})(\partial_{\nu}u-\partial_{\nu_{s}%
}u_{s}\circ\Phi_{s})\right\Vert _{H^{1/2}(\Gamma)}\\
&  \leq C\left\Vert \varphi\right\Vert _{C^{2}([a,b])}\left\Vert
\psi\right\Vert _{H^{1/2}(\Gamma_{s})}\left\Vert \nabla(\hat{u}_{\psi}%
^{s}-u_{\psi\circ\Phi_{s}}^{0})\right\Vert _{L^{2}(\Omega_{+})}\\
&  \quad+C\left\Vert A_{s}\nabla(\hat{u}_{\psi}^{s}-u_{\psi\circ\Phi_{s}}%
^{0})+B_{s}\nabla u_{\psi\circ\Phi_{s}}^{0}\right\Vert _{L^{2}(\Omega_{+}%
)}\left\Vert (\psi\circ\Phi_{s})(\partial_{\nu}u-\partial_{\nu_{s}}u_{s}%
\circ\Phi_{s})\right\Vert _{H^{1/2}(\Gamma)}\\
&  \leq C\left\Vert \varphi\right\Vert _{C^{2}([a,b])}\left\Vert
\psi\right\Vert _{H^{1/2}(\Gamma_{s})}\left\Vert \nabla(\hat{u}_{\psi}%
^{s}-u_{\psi\circ\Phi_{s}}^{0})\right\Vert _{L^{2}(\Omega_{+})}\\
&  \quad+C\left\Vert \nabla(\hat{u}_{\psi}^{s}-u_{\psi\circ\Phi_{s}}%
^{0})\right\Vert _{L^{2}(\Omega_{+})}\left\Vert (\psi\circ\Phi_{s}%
)(\partial_{\nu}u-\partial_{\nu_{s}}u_{s}\circ\Phi_{s})\right\Vert
_{H^{1/2}(\Gamma)}\\
&  \quad+C\left\Vert \varphi\right\Vert _{C^{2}([a,b])}\left\Vert
\psi\right\Vert _{H^{1/2}(\Gamma_{s})}\left\Vert (\psi\circ\Phi_{s}%
)(\partial_{\nu}u-\partial_{\nu_{s}}u_{s}\circ\Phi_{s})\right\Vert
_{H^{1/2}(\Gamma)}.
\end{align*}
Hence,%
\begin{align*}
\left\Vert \nabla(\hat{u}_{\psi}^{s}-u_{\psi\circ\Phi_{s}}^{0})\right\Vert
_{L^{2}(\Omega_{+})}  &  \leq C\left\Vert \varphi\right\Vert _{C^{2}%
([a,b])}\left\Vert \psi\right\Vert _{H^{1/2}(\Gamma_{s})}\\
&  \quad+C\left\Vert (\psi\circ\Phi_{s})(\partial_{\nu}u-\partial_{\nu_{s}%
}u_{s}\circ\Phi_{s})\right\Vert _{H^{1/2}(\Gamma)}.
\end{align*}
As before, using (\ref{s11}), we obtain
\begin{align*}
&  \left\Vert (\psi\circ\Phi_{s})(\partial_{\nu}u-\partial_{\nu_{s}}u_{s}%
\circ\Phi_{s})\right\Vert _{H^{1/2}(\Gamma)}\leq C\left\Vert \psi\right\Vert
_{L^{2}(\Gamma_{s})}\left\Vert \partial_{\nu}u-\partial_{\nu_{s}}u_{s}%
\circ\Phi_{s}\right\Vert _{C^{0,1}(\Gamma)}\\
&  \quad+C\left\vert \psi\right\vert _{H^{1/2}(\Gamma_{s})}\left\Vert
\partial_{\nu}u-\partial_{\nu_{s}}u_{s}\circ\Phi_{s}\right\Vert _{C^{0}%
(\Gamma)}\leq C\left\Vert \varphi\right\Vert _{C^{2,\alpha}(a,b)}\left\Vert
\psi\right\Vert _{H^{1/2}(\Gamma_{s})},
\end{align*}
and so%
\begin{equation}
\left\Vert \nabla\hat{u}_{\psi}^{s}-\nabla u_{\psi\circ\Phi_{s}}%
^{0}\right\Vert _{L^{2}(\Omega_{+})}\leq C\left\Vert \varphi\right\Vert
_{C^{2,\alpha}(a,b)}\left\Vert \psi\right\Vert _{H^{1/2}(\Gamma_{s})}.
\label{h2}%
\end{equation}
Then, also by (\ref{h1}),%
\[
\Big|\int_{\Omega_{+}}|\nabla\hat{u}_{\psi}^{s}|^{2}\,d\boldsymbol{x}%
-\int_{\Omega_{+}}|\nabla u_{\psi\circ\Phi_{s}}^{0}|^{2}\,d\boldsymbol{x}%
\Big|\leq C\left\Vert \varphi\right\Vert _{C^{2,\alpha}(a,b)}\left\Vert
\psi\right\Vert _{H^{1/2}(\Gamma_{s})}^{2},
\]
By a change of variables, we get%
\[
\int_{\Omega_{+}}|\nabla\hat{u}_{\psi}^{s}|^{2}\,d\boldsymbol{x}=\int%
_{\Phi_{s}(\Omega_{+})}|(D\Phi_{s}\circ\Phi_{s}^{-1})\nabla u_{\psi}^{s}%
|^{2}\det D\Phi_{s}\,d\boldsymbol{y}%
\]
In turn, by (\ref{D phi s -I}) and (\ref{h1}), we deduce that%
\begin{align*}
\Big|\int_{\Phi_{s}(\Omega_{+})}|\nabla u_{\psi}^{s}|^{2}\,d\boldsymbol{x}%
-\int_{\Omega_{+}}|\nabla\hat{u}_{\psi}^{s}|^{2}\,d\boldsymbol{x}\Big|  &
\leq C\left\Vert \varphi\right\Vert _{C^{2}([a,b])}\int_{\Phi_{s}(\Omega_{+}%
)}|\nabla u_{\psi}^{s}|^{2}\,d\boldsymbol{x}\\
&  \leq C\left\Vert \varphi\right\Vert _{C^{2}([a,b])}\left\Vert
\psi\right\Vert _{H^{1/2}(\Gamma_{s})}^{2},
\end{align*}
and this concludes the proof.
\end{proof}

Finally, we estimate the first term in the last integral on the right-hand
side of (\ref{2var}).

\begin{proposition}
\label{proposition tangential zero}Let $Q\in C^{1,1}(\Omega)$, let $\varphi$ be as in Theorem~\ref{theorem second variation positive}, 
and let $u_{s}$ be the solution to
problem (\ref{equation u t}), where $\Phi_{s}$ is given by (\ref{phi s}). If
$\left\Vert \varphi\right\Vert _{C^{2,\alpha}(a,b)}$ is sufficiently small,
then for every $s\in\lbrack0,1]$ the following inequality holds:
\begin{equation}
\Big|\int_{\Gamma_{s}}(Q^{2}-|\nabla u_{s}|^{2})Z_{s}\cdot\nu_{s}%
\,d\mathcal{H}^{1}\Big|\leq C\left\Vert \varphi\right\Vert _{C^{2,\alpha
}(a,b)}\int_{\Gamma_{s}}(X_{s}\cdot\nu_{s})^{2}\,d\mathcal{H}^{1},
\label{condition (ii)}%
\end{equation}
where $X_{s}$ and $Z_{s}$ are given in (\ref{X and Z}) and $C>0$ depends on
$\left\Vert w\right\Vert _{C^{2,\alpha}(-1,1)}$.
\end{proposition}

\begin{proof}
Observe that, by (\ref{phi s}), (\ref{partial h}), and the fact that $\lambda(y)=1$ if $L+\delta
_{0}\leq y\leq M+2-\delta_{0}$, we have that%
\begin{align}
&  \dot{\Phi}_{s}(x,w(x))\cdot(D\Phi_{s}(x,w(x)))^{-T}\nu(x,w(x))\nonumber\\
&  =\dot{\Psi}_{s}(x,w(x))\cdot(D\Psi_{s}(x,w(x)))^{-T}\nu(x,w(x))\nonumber\\
&  =-\frac{\partial_{x}h(s,x)}{\partial_{x}g(s,x)}\frac{\partial_{s}%
g(s,x)}{\sqrt{1+(w^{\prime}(x))^{2}}}+\frac{\partial_{s}h(s,x)}{\sqrt
{1+(w^{\prime}(x))^{2}}}=\frac{\varphi(g(s,x))}{\sqrt{1+(w^{\prime}(x))^{2}}}
\label{phi dot}%
\end{align}
provided $\left\Vert \varphi\right\Vert _{C^{2,\alpha}(a,b)}$ is sufficiently
small. Similarly,%
\begin{align}
&  \ddot{\Phi}_{s}(x,w(x))\cdot(D\Phi_{s}(x,w(x)))^{-T}\nu(x,w(x))\nonumber\\
&  =\ddot{\Psi}_{s}(x,w(x))\cdot(D\Psi_{s}(x,w(x)))^{-T}\nu(x,w(x))\nonumber\\
&  =\frac{1}{\partial_{x}g(s,x)}\left(  \partial_{s}^{2}g(s,x),\partial
_{s}^{2}h(s,x)\right)  \cdot\left(
\begin{array}
[c]{cc}%
1 & -\partial_{x}h(s,x)+w^{\prime}(x)\\
0 & \partial_{x}g(s,x)
\end{array}
\right)  \frac{\left(  -w^{\prime}(x),1\right)  }{\sqrt{1+(w^{\prime}(x))^{2}%
}}\nonumber\\
&  =-\frac{\partial_{x}h(s,x)}{\partial_{x}g(s,x)}\frac{\partial_{s}%
^{2}g(s,x)}{\sqrt{1+(w^{\prime}(x))^{2}}}+\frac{\partial_{s}^{2}h(s,x)}%
{\sqrt{1+(w^{\prime}(x))^{2}}}. \label{600}%
\end{align}
Differentiating (\ref{partial h})$_{2}$ with respect to $s$ gives%
\[
\partial_{s}^{2}h=(w^{\prime\prime}(g)+s\varphi^{\prime\prime}(g))(\partial
_{s}g)^{2}+2\varphi^{\prime}(g)\partial_{s}g+(w^{\prime}(g)+s\varphi^{\prime
}(g))\partial_{s}^{2}g.
\]
so that, by (\ref{partial h})$_{1}$ and (\ref{600}),%
\begin{align}
&  \ddot{\Phi}_{s}(x,w(x))\cdot(D\Phi_{s}(x,w(x)))^{-T}\nu(x,w(x))\nonumber\\
&  =\frac{(w^{\prime\prime}(g)+s\varphi^{\prime\prime}(g))(\partial_{s}g)^{2}%
}{\sqrt{1+(w^{\prime}(x))^{2}}}+\frac{2\varphi^{\prime}(g)\partial_{s}g}%
{\sqrt{1+(w^{\prime}(x))^{2}}}. \label{phi ddot}%
\end{align}
Since, by (\ref{tangent and normal}), we have
\[
\nu_{s}(\Phi_{s}(x,w(x)))=\frac{(D\Phi_{s}(x,w(x)))^{-T}\nu(x,w(x))}%
{\Big((D\Phi_{s}(x,w(x)))^{-T}\nu(x,w(x))\Big)}%
\]
and
\[
\Big((D\Phi_{s}(x,w(x)))^{-T}\nu(x,w(x))\Big)=\frac{\sqrt{1+[w^{\prime
}(g(x))+s\varphi^{\prime}(g(x))]^{2}}}{\sqrt{1+(w^{\prime}(x))^{2}}},
\]
by (\ref{X and Z}) we have%
\begin{align*}
&  \int_{\Gamma_{s}}(Q^{2}-|\nabla u_{s}|^{2})Z_{s}\cdot\nu_{s}\,d\mathcal{H}%
^{1}\\
&  =\int_{a}^{b}(Q^{2}(g,h)-|\nabla u_{s}(g,h)|^{2})(w^{\prime\prime
}(g)+s\varphi^{\prime\prime}(g))(\partial_{s}g)^{2}\partial_{x}g~dx\\
&  \quad+\int_{a}^{b}(Q^{2}(g,h)-|\nabla u_{s}(g,h)|^{2})2\varphi^{\prime
}(g)\,\partial_{s}g\,\partial_{x}g~dx=:I+II.
\end{align*}
By (\ref{partial s of g (new)}) and (\ref{partial s of g at a}) in the
appendix we obtain
\begin{equation}
|\partial_{s}g(s,x)|\leq|\varphi(g(s,x))| \label{est1}%
\end{equation}
for every $(s,x)\in\lbrack0,1]\times\lbrack a,b]$. Hence, by Proposition~\ref{proposition estimates us}, for $\left\Vert \varphi\right\Vert
_{C^{2,\alpha}(a,b)}$ sufficiently small,
\begin{align*}
|I|  &  \leq C\left\Vert \varphi\right\Vert _{C^{2,\alpha}(a,b)}\int_{a}%
^{b}\varphi^{2}(g(s,x))\,\partial_{x}g~dx\\
&  \leq C\left\Vert \varphi\right\Vert _{C^{2,\alpha}(a,b)}\int_{a}^{b}%
\frac{\varphi^{2}(g(s,x))}{\sqrt{1+[w^{\prime}(g)+s\varphi^{\prime}(g)]^{2}}%
}\,\partial_{x}g~dx\\
&  =C\left\Vert \varphi\right\Vert _{C^{2,\alpha}(a,b)}\int_{\Gamma_{s}}%
(X_{s}\cdot\nu_{s})^{2}\,d\mathcal{H}^{1},
\end{align*}
where $C$ depends only on $\left\Vert w\right\Vert _{C^{2,\alpha}(a,b)}$, and where
we have used (\ref{phi dot}).

To estimate $II$, we use (\ref{partial s of g (new)}) to write%
\[
II=-\int_{a}^{b}(Q^{2}(g,h)-|\nabla u_{s}(g,h)|^{2})\frac{2\varphi^{\prime
}(g)\varphi(g)(w^{\prime}(g)+s\varphi^{\prime}(g))\partial_{x}g}{1+[w^{\prime
}(g)+s\varphi^{\prime}(g)]^{2}}~dx
\]
Using the change of variables $r=g(x,s)$ and (\ref{h}), we have
\[
II=-\int_{a}^{b}\frac{(Q^{2}(r,w+s\varphi)-|\nabla u_{s}(r,w+s\varphi
)|^{2})2\varphi^{\prime}\varphi(w^{\prime}+s\varphi^{\prime})}{1+[w^{\prime
}+s\varphi^{\prime}]^{2}}~dr.
\]
Integrating by parts and using (\ref{f and itz derivatives zero at endpoints}),
we obtain%
\[
II=\int_{a}^{b}\varphi^{2}\partial_{r}\Big(\frac{(Q^{2}(r,w+s\varphi)-|\nabla
u_{s}((r,w+s\varphi)|^{2})(w^{\prime}+s\varphi^{\prime})}{1+[w^{\prime
}+s\varphi^{\prime}]^{2}}\Big)~dr.
\]
It follows from Proposition~\ref{proposition estimates us} that
\[
|II|\leq C\left\Vert \varphi\right\Vert _{C^{2,\alpha}(a,b)}\int_{a}%
^{b}\varphi^{2}~dr\leq C\left\Vert \varphi\right\Vert _{C^{2,\alpha}(a,b)}%
\int_{\Gamma_{s}}(X_{s}\cdot\nu_{s})^{2}\,d\mathcal{H}^{1},
\]
where in the last inequality we have reasoned as in the estimate of $I$.
\end{proof}

Next we prove Theorem~\ref{theorem second variation positive}.\medskip

\begin{proof}
[Proof of Theorem~\ref{theorem second variation positive}]Let $\varphi$ be as
in the statement of Theorem~\ref{theorem second variation positive}, and let
$\{\Phi_{s}\}_{s\in\lbrack0,1]}$ be the admissible flow given in Theorem~\ref{theorem diffeomorphism}. By Theorem~\ref{thm:vars} and
Propositions~\ref{proposition estimates us}
and~\ref{proposition tangential zero},
\begin{align*}
\frac{d^{2}}{ds^{2}}\mathcal{F}(u_{s})\geq &  \int_{\Phi_{s}(\Omega_{+}%
)}2|\nabla\dot{u}_{s}|^{2}\,d\boldsymbol{x}+\int_{\Gamma_{s}}\big(\partial
_{\nu_{s}}Q^{2}+2\kappa_{s}(\partial_{\nu_{s}}u_{s})^{2}\big)(X_{s}\cdot
\nu_{s})^{2}\,d\mathcal{H}^{1}\\
&  \quad-C\left\Vert \varphi\right\Vert _{C^{2,\alpha}(a,b)}\int_{\Gamma_{s}%
}(X_{s}\cdot\nu_{s})^{2}\,d\mathcal{H}^{1},
\end{align*}
where we used the fact that $|\kappa_{s}|\leq C$. On the other hand, by
(\ref{udot eq}) and by Propositions~\ref{proposition estimate 1}
and~\ref{proposition estimate 2} with $\psi=X_{s}\cdot\nu_{s}$, we have
\begin{align*}
\frac{d^{2}}{ds^{2}}\mathcal{F}(u_{s})\geq &  \int_{\Omega_{+}}2|\nabla
u_{\psi_{s}}|^{2}\,d\boldsymbol{x}+\int_{\Gamma}\big(\partial_{\nu}%
Q^{2}+2\kappa(\partial_{\nu}u)^{2}\big)(X_{s}\cdot\nu_{s})^{2}\circ\Phi
_{s}\,d\mathcal{H}^{1}\\
&  \quad-C_{1}\left\Vert \varphi\right\Vert _{C^{2,\alpha}(a,b)}\left\Vert
X_{s}\cdot\nu_{s}\right\Vert _{H^{1/2}(\Gamma_{s})}^{2},
\end{align*}
where $\psi_{s}:=(X_{s}\cdot\nu_{s})\circ\Phi_{s}$ and $u_{\psi_{s}}$ is the
unique solution to the problem%
\[%
\begin{cases}
\Delta u_{\psi_{s}}=0 & \text{in }\Omega_{+},\\
u_{\psi_{s}}=-\psi_{s}\partial_{\nu}u & \text{on }\Gamma,\\
u_{\psi_{s}}=0 & \text{on }\{y=0\}\cap\partial\Omega_{+},%
\end{cases}
\]
with $u_{\psi_{s}}(-1,y)=u_{\psi_{s}}(1,y)$ for all $y$ such that $(\pm
1,y)\in\overline{\Omega_{+}}$. Now we apply
(\ref{u stat}) and (\ref{coercivity}) to obtain
\[
\frac{d^{2}}{ds^{2}}\mathcal{F}(u_{s})\geq(C_{0}-C_{1}\left\Vert
\varphi\right\Vert _{C^{2,\alpha}(a,b)})\left\Vert X_{s}\cdot\nu
_{s}\right\Vert _{H^{1/2}(\Gamma_{s})}^{2}.
\]
By taking $\left\Vert \varphi\right\Vert _{C^{2,\alpha}(a,b)}\leq
C_{0}/(2C_{1})$ we get $\frac{d^{2}}{ds^{2}}\mathcal{F}(u_{s})\geq0$ for all
$s\in\lbrack0,1]$. In turn, by~(\ref{u stat}),%
\begin{align*}
\mathcal{F}(u)  &  =\mathcal{F}(u_{1})-\int_{0}^{1}(1-s)\frac{d^{2}}{ds^{2}%
}\mathcal{F}(u_{s})~ds\\
&  \leq\mathcal{F}(u_{1})=\int_{\Phi_{1}(\Omega_{+})}\big(|\nabla u_{1}%
|^{2}+Q^{2}(\boldsymbol{x})\big)\,d\boldsymbol{x}.
\end{align*}
In view of (\ref{equation u t}), $u_{1}$ is the unique minimizer of
$\mathcal{F}$ over all functions $v\in H^{1}(\Phi_{1}(\Omega_{+}))$ such that
$v=0$ on $\Phi_{1}(\Gamma)$, $v=u$ on $\partial\Phi_{1}(\{y=0\}\cap
\partial\Omega_{+})$ and $v(-1,y)=v(1,y)$ for all $y$ such that $(\pm
1,y)\in\Phi_{1}(\overline{\Omega_{+}})$.
In particular, for every $v\in\mathcal{A}$ with $\{v>0\}=\Phi_{1}(\{u>0\})$,
we have%
\[
\mathcal{F}(u)\leq\mathcal{F}(v),
\]
which concludes the proof.
\end{proof}

We conclude this section with the proof of the main theorem.\medskip

\begin{proof}
[Proof of Theorem~\ref{theorem main}]Let $U\Subset\Omega$, $\delta>0$, and let
$\Phi\in C^{2,\alpha}(\mathbb{R}^{2};\mathbb{R}^{2})$ be a diffeomorphism
satisfying (\ref{m1}) and (\ref{m2}).\smallskip

\noindent\textbf{Step 1:} We begin by proving that there exist a constant
$C>0$ and an interval $[a,b]\subset(-1,1)$ (independent of $\Phi$) such that
the set $\Phi(\Gamma)$ is the graph of a function $w+\varphi$, where
$\varphi\in C^{2,\alpha}(-1,1)$ has compact support in $[a,b]$ and satisfies
\begin{equation}
\left\Vert \varphi\right\Vert _{C^{2,\alpha}(-1,1)}\leq C\delta. \label{m4}%
\end{equation}

Consider the function
\[
\psi(x):=\Phi^{1}(x,w(x)),\quad x\in\lbrack-1,1],
\]
where $\Phi=(\Phi^{1},\Phi^{2})$. By the chain rule, $\psi\in C^{2,\alpha
}(-1,1)$ with
\begin{align}
\psi^{\prime}(x)  &  =\partial_{x}\Phi^{1}(x,w(x))+w^{\prime}(x)\partial
_{y}\Phi^{1}(x,w(x))\label{m3}\\
&  \geq1-\delta-\delta\left\Vert w^{\prime}\right\Vert _{C^{0}(-1,1)}\geq
\frac{1}{2}\nonumber
\end{align}
for all $0<\delta<\frac{1}{2+2\left\Vert w^{\prime}\right\Vert _{C^{0}(-1,1)}%
}$, where we used the facts that $\partial_{x}\Phi_{1}(x,y)\geq1-\delta$ and
$\left\vert \partial_{y}\Phi_{1}(x,y)\right\vert \leq\delta$ by (\ref{m2}).
Moreover, by (\ref{m1}), $\psi(-1)=-1$ and $\psi(1)=1$. Hence, $\psi
:[-1,1]\rightarrow\lbrack-1,1]$ is invertible, and by the chain rule $\psi
^{-1}\in C^{2,\alpha}(-1,1)$. It follows that
\[
\Phi(\Gamma)=\{(x,\Phi^{2}(\psi^{-1}(x),w(\psi^{-1}(x)))):~x\in\lbrack
-1,1]\}.
\]
Define $\varphi(x):=\Phi^{2}(\psi^{-1}(x),w(\psi^{-1}(x)))-w(x)$. By
(\ref{m1}), $\psi(x)=x$ for $x$ in a neighborhood of $-1$ and of $1$,
$\Phi^{2}(x,y)=y$ for $(x,y)\notin U$. Hence, $\varphi$ has compact support in
$(-1,1)$. A lengthy, but straightforward calculation using (\ref{m2}), shows
that (\ref{m4}) holds.\smallskip

\noindent\textbf{Step 2:} Let now $\{\varphi_{n}\}_{n}$ be a sequence of
polynomials satisfying (\ref{f and itz derivatives zero at endpoints}) and
such that $\varphi_{n}\rightarrow\varphi$ in $C^{2,\alpha}(a,b)$. By
Theorems~\ref{theorem diffeomorphism}
and~\ref{theorem second variation positive}, for $\delta$ small enough we can
construct an admissible flow $\{\Phi_{s,n}\}_{s\in[0,1]}$ (see
Definition~\ref{definition admissible flow}) for every $n$ such that
\[
\Phi_{1,n}(\Gamma)=\left\{  (x,w(x)+\varphi_{n}(x)):\,x\in(-1,1)\right\}
\]
and
\begin{equation}
\mathcal{F}(u)\leq\mathcal{F}(v) \label{m6}%
\end{equation}
for every $v\in\mathcal{A}$ with $\{v>0\}=\Phi_{1,n}(\{u>0\})$.

Consider now a function $v\in\mathcal{A}$ with $\{v>0\}=\Phi(\{u>0\})$, and
define
\[
v_{n}(x,y):=v(x,y-\varphi_{n}(x)+\varphi(x)).
\]
Then $y<w(x)+\varphi_{n}(x)$ if and only if $y-\varphi_{n}(x)+\varphi
(x)<w(x)+\varphi(x)$. Let $\tau>0$. Since $\varphi_{n}\rightarrow\varphi$ in
$C^{2,\alpha}(a,b)$, we have that $v_{n}\rightarrow v$ in $H^{1}%
((-1,1)\times(\tau,\infty))$.

We now construct $\lambda_{\tau}\in C^{\infty}(\mathbb{R})$ such that
$0\leq\lambda_{\tau}\leq1$, $\lambda_{\tau}(y)=1$ if $2\tau\leq y$,
$\lambda_{\tau}(y)=0$ if $y\leq\tau$ and $\left\vert \lambda_{\tau}^{\prime
}(y)\right\vert \leq2/\tau$ for all $y\in\mathbb{R}$. Define
\[
v_{n,\tau}(x,y):=\lambda_{\tau}(y)v_{n}(x,y)+(1-\lambda_{\tau}(y))u(x,y).
\]
Since $\Phi_{1,n}$ satisfies (\ref{m1}), we have that $v_{n,\tau}%
\in\mathcal{A}$ and $\{v_{n,\tau}>0\}=\Phi_{1,n}(\{u>0\})$. Hence, by
(\ref{m6}), we have
\begin{align}
\mathcal{F}(u)  &  \leq\mathcal{F}(v_{n,\tau})=\int_{\Phi_{1,n}(\Omega_{+}%
)}\big(|\nabla v_{n,\tau}|^{2}+Q^{2}(\boldsymbol{x})\big)\,d\boldsymbol{x}%
\nonumber\\
&  =\int_{\Omega}\big(|\nabla v_{n,\tau}|^{2}+\chi_{\{v_{n,\tau}>0\}}%
Q^{2}(\boldsymbol{x})\big)\,d\boldsymbol{x}. \label{m6a}%
\end{align}
Since $\varphi_{n}\rightarrow\varphi$ in $C^{2,\alpha}(a,b)$, if
$(x,y)\in\Omega$ is such that $y\neq w(x)+\varphi(x)$ then for all $n$
sufficiently large $y\neq w(x)+\varphi_{n}(x)$, and so $\chi_{\{v_{n,\tau}%
>0\}}(x,y)=\chi_{\{v>0\}}(x,y)$. It follows by the Lebesgue dominated
convergence theorem that
\begin{equation}
\lim_{n\rightarrow\infty}\int_{\Omega}\chi_{\{v_{n,\tau}>0\}}Q^{2}%
(\boldsymbol{x})\,d\boldsymbol{x}=\int_{\Omega}\chi_{\{v>0\}}Q^{2}%
(\boldsymbol{x})\,d\boldsymbol{x}. \label{m6b}%
\end{equation}
On the other hand,
\[
\nabla v_{n,\tau}=\lambda_{\tau}\nabla v_{n}+(1-\lambda_{\tau})\nabla
u+(v_{n}-u)\lambda_{\tau}^{\prime}e_{2}.
\]
Hence, using convexity and the inequality $(a+b)^{2}\leq(1+\varepsilon)a^{2}+C_{\varepsilon
}b^{2}$, we obtain%
\begin{align*}
\int_{\Omega}|\nabla v_{n,\tau}|^{2}\,d\boldsymbol{x}  &  \leq(1+\varepsilon
)\int_{\Omega}\lambda_{\tau}|\nabla v_{n}|^{2}\,d\boldsymbol{x}+(1+\varepsilon
)\int_{\Omega}(1-\lambda_{\tau})|\nabla u|^{2}\,d\boldsymbol{x}\\
&  \quad+\frac{4C_{\varepsilon}}{\tau^{2}}\int_{(-1,1)\times(\tau,2\tau
)}|v_{n}-u|^{2}\,d\boldsymbol{x}.
\end{align*}
Since $v_{n}\rightarrow v$ in $H^{1}((-1,1)\times(\tau,\infty))$, letting
$n\rightarrow\infty$ we have that%
\begin{align}
\limsup_{n\rightarrow\infty}\int_{\Omega}|\nabla v_{n,\tau}|^{2}%
\,d\boldsymbol{x}  &  \leq(1+\varepsilon)\int_{\Omega}\lambda_{\tau}|\nabla
v|^{2}\,d\boldsymbol{x}+(1+\varepsilon)\int_{\Omega}(1-\lambda_{\tau})|\nabla
u|^{2}\,d\boldsymbol{x}\nonumber\\
&  \quad+\frac{4C_{\varepsilon}}{\tau^{2}}\int_{(-1,1)\times(\tau,2\tau
)}|v-u|^{2}\,d\boldsymbol{x}. \label{m7}%
\end{align}
By (\ref{class admissible}), if $v$ is of class $C^{1}$, it holds%
\[
v(x,y)-u(x,y)=\int_{0}^{y}(\partial_{y}v(x,r)-\partial_{y}u(x,r))~dr,
\]
and so by H\"{o}lder's inequality
\begin{align*}
\int_{(-1,1)\times(\tau,2\tau)}|v-u|^{2}\,d\boldsymbol{x}  &  \leq
\int_{(-1,1)\times(\tau,2\tau)}\Big(\int_{0}^{y}|\partial_{y}v(x,r)|+|\partial
_{y}u(x,r)|~dr\Big)^{2}\,d\boldsymbol{x}\\
&  \leq\int_{(-1,1)\times(\tau,2\tau)}\int_{0}^{y}y((\partial_{y}%
v(x,r))^{2}+(\partial_{y}u(x,r))^{2})~dr\,d\boldsymbol{x}\\
&  \leq4\tau^{2}\int_{(-1,1)\times(0,2\tau)}(\partial_{y}v(x,r))^{2}%
+(\partial_{y}u(x,r))^{2})~dx\,dr.
\end{align*}
By density, the same inequality is satisfied without any extra regularity on $v$.

We now combine (\ref{m6a})--(\ref{m7}) with the previous inequality. By first
letting $\tau\rightarrow0^{+}$ and then $\varepsilon\rightarrow0^{+}$, we
conclude that $\mathcal{F}(u)\leq\mathcal{F}(v)$.
\end{proof}

\section{Proof of Theorem~\ref{thm stability}}

\label{section proof stability}

The proof of Theorem~\ref{thm stability} is based on some auxiliary lemmas. We
start by showing that the first term in the expression (\ref{coercivity}) of
the second variation is coercive with respect to the $H^{1/2}$ norm of the
boundary datum on $\Gamma$.

\begin{lemma}
\label{lm1}Let $Q$, $u$, and $\Gamma$ be as in Theorem~\ref{thm stability},
let $U\subset\Omega$ be an open set such that $U\cap\Gamma\neq\emptyset$, and
let $A:=\{u>0\}\cap U$. Assume that $A$ has a Lipschitz boundary. Then there
exist two constants $C_{1},C_{2}>0$, depending on $A$, such that
\begin{equation}
C_{1}\Vert\hat{\psi}\Vert_{H^{1/2}(\partial A)}^{2}\leq\inf\Big\{\int%
_{A}|\nabla v|^{2}\,d\boldsymbol{x}:\ v\in H^{1}(A),\ v=\hat{\psi}\text{ on
}\partial A\Big\}\leq C_{2}\Vert\hat{\psi}\Vert_{H^{1/2}(\partial A)}^{2}
\label{trace ineq}%
\end{equation}
for every $\psi\in C_{c}^{1}(\Gamma\cap U)$, where
\[
\hat{\psi}:=%
\begin{cases}
Q\,\psi & \text{in }\Gamma\cap U,\\
0 & \text{in }\partial U\cap\{u>0\}.
\end{cases}
\]

\end{lemma}

\begin{proof}
Let $\psi\in C_{c}^{1}(\Gamma\cap U)$. Since $Q^{2}\in C^{0,1}$ and $Q\geq
Q_{\mathrm{min}}>0$, we have that $\hat{\psi}\in H^{1/2}(\partial A)$, and so
there exists $v^{\ast}\in H^{1}(A)$ such that $v^{\ast}=\hat{\psi}$ on
$\partial A$ in the sense of traces and
\[
\Vert v^{\ast}\Vert_{H^{1}(A)}^{2}\leq C_{2}\Vert\hat{\psi}\Vert
_{H^{1/2}(\partial A)}^{2},
\]
where $C_{2}$ is a positive constant depending on $A$. Thus, the second
inequality in (\ref{trace ineq}) holds. On the other hand, the trace operator
$T:H^{1}(A)\rightarrow H^{1/2}(\partial A)$ is continuous, and so there exists
a positive constant $\hat{C}_{1}$, depending on $A$, such that
\[
\Vert T(v)\Vert_{H^{1/2}(\partial A)}^{2}\leq\hat{C}_{1}\Vert v\Vert
_{H^{1}(A)}^{2}%
\]
for every $v\in H^{1}(A)$. In particular, given $\psi\in C_{c}^{1}(\Gamma\cap
U)$, we have that
\[
\Vert\hat{\psi}\Vert_{H^{1/2}(\partial A)}^{2}\leq\hat{C}_{1}\Vert
v\Vert_{H^{1}(A)}^{2}%
\]
for every $v\in H^{1}(A)$ with $T(v)=\hat{\psi}$. Since $\hat{\psi}=0$ in
$\partial U\cap\{u>0\}$, by Poincar\'{e}'s inequality
\[
\Vert v\Vert_{H^{1}(A)}^{2}\leq\tilde{C}_{1}\Vert\nabla v\Vert_{L^{2}(A)}^{2}%
\]
for every $v\in H^{1}(A)$ with $T(v)=\hat{\psi}$. Combining these two last
inequalities, we get the first inequality in (\ref{trace ineq}).
\end{proof}

\begin{lemma}
\label{lm2}Let $Q$, $u$, and $\Gamma$ be as in Theorem~\ref{thm stability}.
For every $\varepsilon>0$ let $U_{\varepsilon}$ be the intersection of
$\Omega$ with the $\varepsilon$-tubular neighborhood of $\Gamma$. Define
\[
\mu_{\varepsilon}:=\inf\Big\{\int_{U_{\varepsilon}\cap\{u>0\}}|\nabla u_{\psi
}|^{2}\,d\boldsymbol{x}:\ \psi\in C_{c}^{1}(\Gamma),\ \Vert\psi\Vert
_{L^{2}(\Gamma)}=1\Big\},
\]
where for every $\psi\in C_{c}^{1}(\Gamma)$ the function $u_{\psi}$ is the
solution to
\[%
\begin{cases}
\Delta u_{\psi}=0 & \text{in }U_{\varepsilon}\cap\{u>0\},\\
u_{\psi}=Q\psi & \text{on }\Gamma,\\
u_{\psi}=0 & \text{on }\partial U_{\varepsilon}\cap\{u>0\},
\end{cases}
\]
with $u_{\psi}(-1,y)=u_{\psi}(1,y)$ for all $y$ such that $(\pm1,y)\in
\overline{U_{\varepsilon}\cap\{u>0\}}$. Then
\[
\lim_{\varepsilon\rightarrow0^{+}}\mu_{\varepsilon}=\infty.
\]

\end{lemma}

\begin{proof}
Assume, by contradiction, that there exist $C>0$, $\varepsilon_{n}%
\rightarrow0^{+}$, and $\psi_{n}\in C_{c}^{1}(\Gamma)$ with $\Vert\psi
_{n}\Vert_{L^{2}(\Gamma)}=1$, such that
\[
\int_{U_{\varepsilon_{n}}\cap\{u>0\}}|\nabla u_{n}|^{2}\,d\boldsymbol{x}\leq
C\quad\text{for every }n,
\]
where $u_{n}:=u_{\psi_{n}}$. We extend $u_{n}$ by $0$ to the set $U_{1}%
\cap\{u>0\}=:V$. Then
\[
\int_{V}|\nabla u_{n}|^{2}\,d\boldsymbol{x}\leq C\quad\text{for every }n.
\]
By Poincar\'{e}'s inequality there exists $u_{\infty}\in H^{1}(V)$ such that
$u_{n}\rightharpoonup u_{\infty}$ weakly in $H^{1}(V)$, up to a subsequence,
not relabeled. This implies that $u_{n}\rightarrow u_{\infty}$ strongly in
$L^{2}(\partial V)$. Since $\mathcal{L}^{2}(U_{\varepsilon_{n}})\rightarrow0$,
we have that $u_{n}\rightarrow0$ a.e.\ in $V$, hence $u_{\infty}=0$ and
$u_{n}\rightarrow0$ strongly in $L^{2}(\partial V)$.

On the other hand,
\[
1=\Vert\psi_{n}\Vert_{L^{2}(\Gamma)}\leq C\big\|Q^{-1/2}\big\|_{C^{0}}\Vert
u_{n}\Vert_{L^{2}(\Gamma)}.
\]
Since the right-hand side tends to $0$, we arrive at a contradiction.
\end{proof}

We now prove Theorem~\ref{thm stability}.\medskip

\begin{proof}
[Proof of Theorem~\ref{thm stability}]By Lemmas~\ref{lm1} and \ref{lm2}, for
$\varepsilon>0$ small enough we have%
\[
\int_{U_{\varepsilon}\cap\{u>0\}}2|\nabla u_{\psi}|^{2}\,d\boldsymbol{x}%
+\int_{\Gamma}(\partial_{\nu}Q^{2}+2\kappa Q^{2})\psi^{2}\,d\mathcal{H}%
^{1}\geq C_{\varepsilon}|\hat{\psi}|_{H^{1/2}(\Gamma)}^{2}+(\mu_{\varepsilon
}-C_{3})\Vert\psi\Vert_{L^{2}(\Gamma)}^{2},
\]
where $C_{3}:=(1+\Vert\kappa\Vert_{C^{0}})\Vert Q^{2}\Vert_{C^{0,1}}$. On the
other hand,
\begin{align*}
|\psi|_{H^{1/2}(\Gamma)}  &  =|\hat{\psi}/Q|_{H^{1/2}(\Gamma)}\leq\Vert
\hat{\psi}\Vert_{L^{2}(\Gamma)}|1/Q|_{C^{0,1}(\Gamma)}+|\hat{\psi}%
|_{H^{1/2}(\Gamma)}\Vert1/Q\Vert_{C^{0}(\Gamma)}\\
&  \leq\Vert\psi\Vert_{L^{2}(\Gamma)}\Vert Q\Vert_{C^{0}(\Gamma)}%
|1/Q|_{C^{0,1}(\Gamma)}+|\hat{\psi}|_{H^{1/2}(\Gamma)}/Q_{\mathrm{min}},
\end{align*}
and so
\[
|\hat{\psi}|_{H^{1/2}(\Gamma)}^{2}\geq\frac{1}{2}Q_{\mathrm{min}}^{2}%
|\psi|_{H^{1/2}(\Gamma)}^{2}-C_{4}\Vert\psi\Vert_{L^{2}(\Gamma)}^{2}.
\]
Hence,
\begin{align*}
\int_{U_{\varepsilon}\cap\{u>0\}}  &  2|\nabla u_{\psi}|^{2}\,d\boldsymbol{x}%
+\int_{\Gamma}(\partial_{\nu}Q^{2}+2\kappa Q^{2})\psi^{2}\,d\mathcal{H}^{1}\\
&  \geq\frac{C_{\varepsilon}}{2}Q_{\mathrm{min}}^{2}|\psi|_{H^{1/2}(\Gamma
)}^{2}+(\mu_{\varepsilon}-C_{3}-\min\{C_{\varepsilon}/2,1\}C_{4})\Vert
\psi\Vert_{L^{2}(\Gamma)}^{2}.
\end{align*}
Since $\mu_{\varepsilon}\rightarrow\infty$ by Lemma~\ref{lm2}, the inequality
(\ref{coerc Ue}) holds.

The second part of the statement follows from (\ref{coerc Ue}) by repeating
the proof of Theorem~\ref{theorem main}. We omit the details.
\end{proof}

\section*{Appendix}
Here we sketch the proof of the derivation of the  Euler-Lagrange equations \eqref{E-L} of (\ref{functional F}). Let $v\in\mathcal{A}_{0}\cap C^2(\overline{\Omega})$ be such that $\partial\Omega_+\cap\Omega$ (see \eqref{Omega plus}) is a manifold of class $C^2$,  $\mathcal{F}(v)\in\mathbb{R}$ and \eqref{critical} holds. Since $\Omega_{+}$ is open, consider variations $\varphi\in C^\infty_c(\Omega_{+})$. For $\varepsilon>0$ sufficiently small it can be shown that $\{v+\varepsilon\varphi>0\}=\{v>0\}$, therefore from \eqref{critical} we obtain
\begin{align*}
0=\frac{d}{d\varepsilon}\int_{\Omega}\vert\nabla (v+\varepsilon\varphi)\vert^2\,d\boldsymbol{x}\Big{\vert}_{\varepsilon=0}
=2\int_{\Omega_+}\nabla v\cdot\nabla\varphi\,d\boldsymbol{x}.
\end{align*}
This gives \eqref{E-L}$_1$, and the condition $v=0$ on $\Omega\cap
\partial\left\{  v>0\right\}$ follows from the continuity of $v$. To prove that $\left\vert \nabla v\right\vert =Q$ on $\Omega\cap
\partial\left\{  v>0\right\}$ we use Theorem 2.5 in \cite{alt-caffarelli81} to obtain 
\begin{equation}\label{boundary AC}
\lim_{\varepsilon\rightarrow0^{+}}\int_{\partial\{v>\varepsilon\}}(|\nabla v|^2-Q^2)\eta\cdot\nu \,d\mathcal{H}^1=0
\end{equation}
for every $\eta\in C^\infty_c(\Omega;\mathbb{R}^2)$. Note that in the original proof of \eqref{boundary AC}, $v$ was assumed to be a local minimizer, but this property was used only to guarantee the validity of   \eqref{critical}.
In view of the smoothness of $v$ and $\partial\Omega_+\cap\Omega$, for $\varepsilon$ sufficiently small $\partial\{v>\varepsilon\}$ is a smooth manifold of class $C^2$, and using a partition of unity, it can be shown that \eqref{boundary AC} reduces to
\begin{equation}\label{AC2}
\int_{\partial\{v>0\}}(|\nabla v|^2-Q^2)\eta\cdot\nu \,d\mathcal{H}^1=0.
\end{equation}
Extend locally the outward unit normal $\nu$ to $\partial\{v>0\}$ as a $C^1$ function $\bar{\nu}$ in an open neighborhood of $\partial\{v>0\}$, and take $\eta:=\varphi\bar{\nu}$, where $\varphi \in C^\infty_c(\Omega)$ is supported in that neighborhood. Then \eqref{AC2} yields
\begin{equation*}
\int_{\partial\{v>0\}}(|\nabla v|^2-Q^2)\varphi \,d\mathcal{H}^1=0.
\end{equation*}
By the arbitrariness of $\varphi$ we deduce that $\left\vert \nabla v\right\vert =Q$ on $\Omega\cap
\partial\left\{  v>0\right\}$.

\bigskip

The remaining of the Appendix is dedicated to the proof of Theorem~\ref{theorem regularity g}.\medskip

\begin{proof}
[Proof of Theorem~\ref{theorem regularity g}]Let
\[
D:=\left\{  x\in\mathbb{R}:\,\varphi(x)\neq0\right\}  .
\]

\noindent\textbf{Step 1:} \textbf{Regularity at points }$s\geq0$, $x\in
D$\textbf{.} By Theorem~\ref{theorem existence t0} the function $t_{0}$ is of
class $C^{\infty}$ in $[0,1]\times D$. Hence, by (\ref{g}) and the smooth
dependence of $\xi$ with respect to initial data, we have that $g$ is of class
$C^{\infty}$ in $[0,1]\times D$. Taking $t=t_{0}(s,x)$ in
(\ref{conservation law}) gives%
\[
\int_{0}^{t_{0}(s,x)}\left[  \frac{\left(  w^{\prime}(\xi(r,x))\varphi
(\xi(r,x))+\left(  \eta(r,x)-w(\xi(r,x))\right)  \varphi^{\prime}%
(\xi(r,x))\right)  ^{2}}{\varphi^{2}(\xi(r,x))}+1\right]  \,dr=s
\]
for all $(s,x)\in[0,1]\times D$. Differentiating with respect to $s$, and using
(\ref{eta at time t0}), yields%
\begin{equation}
\partial_{s}t_{0}(s,x)=\frac{1}{1+\left[  w^{\prime}(\xi(t_{0}%
(s,x),x))+s\varphi^{\prime}(\xi(t_{0}(s,x),x))\right]  ^{2}}
\label{partial s of t0}%
\end{equation}
for all $(s,x)\in[0,1]\times D$. Since $t_{0}(0,x)=0$, it follows upon
integration and by (\ref{g}) that%
\begin{equation}
t_{0}(s,x)=\int_{0}^{s}\frac{1}{1+\left[  w^{\prime}(g(r,x))+r\varphi^{\prime
}(g(r,x))\right]  ^{2}}\,dr. \label{formula t0}%
\end{equation}

By (\ref{IVP}), (\ref{g}), (\ref{eta at time t0}), and (\ref{partial s of t0}%
),%
\begin{align}
\partial_{s}g(s,x)  &  =\partial_{t}\xi(t_{0}(s,x),x)\partial_{s}%
t_{0}(s,x)\nonumber\\
&  =\frac{-w^{\prime}(g(s,x))\varphi(g(s,x))-\left(  \eta(t_{0}%
(s,x),x)-w(g(s,x))\right)  \varphi^{\prime}(g(s,x))}{1+\left[  w^{\prime
}(g(s,x))+s\varphi^{\prime}(g(s,x))\right]  ^{2}} \label{partial s of g (new)}%
\\
&  =-\frac{\varphi(g(s,x))\left[  w^{\prime}(g(s,x))+s\varphi^{\prime
}(g(s,x))\right]  }{1+\left[  w^{\prime}(g(s,x))+s\varphi^{\prime
}(g(s,x))\right]  ^{2}}.\nonumber
\end{align}
Differentiating (\ref{partial s of g (new)}) with respect to $s$ and $x$,
respectively, gives
\begin{align}
\partial_{s}^{2}g  &  =-\frac{\varphi^{\prime}(g)\left[  w^{\prime
}(g)+s\varphi^{\prime}(g)\right]  \partial_{s}g+\varphi(g)\left[
w^{\prime\prime}(g)+s\varphi^{\prime\prime}(g)\right]  \partial_{s}%
g+\varphi(g)\varphi^{\prime}(g)}{1+\left[  w^{\prime}(g)+s\varphi^{\prime
}(g)\right]  ^{2}}\nonumber\\
&  \quad{}+\frac{2\varphi(g)\left(  w^{\prime}(g)+s\varphi^{\prime}(g)\right)
^{2}\left\{  \left(  w^{\prime\prime}(g)+s\varphi^{\prime\prime}(g)\right)
\partial_{s}g+\varphi^{\prime}(g)\right\}  }{[1+(w^{\prime}(g)+s\varphi
^{\prime}(g))^{2}]^{2}}, \label{partial 2s of g}%
\end{align}
and
\begin{align}
\partial_{x,s}^{2}g  &  =-\frac{\varphi^{\prime}(g)\left[  w^{\prime
}(g)+s\varphi^{\prime}(g)\right]  \partial_{x}g+\varphi(g)\left[
w^{\prime\prime}(g)+s\varphi^{\prime\prime}(g)\right]  \partial_{x}%
g}{1+\left[  w^{\prime}(g)+s\varphi^{\prime}(g)\right]  ^{2}}%
\label{partial 2sx of g}\\
&  \quad+\frac{2\varphi(g)\left(  w^{\prime}(g)+s\varphi^{\prime}(g)\right)
^{2}\left(  w^{\prime\prime}(g)+s\varphi^{\prime\prime}(g)\right)
\partial_{x}g}{[1+(w^{\prime}(g)+s\varphi^{\prime}(g))^{2}]^{2}},\nonumber
\end{align}
while by Schwartz's theorem $\partial_{s,x}^{2}g=\partial_{x,s}^{2}g$.

On the other hand, by Theorem~\ref{theorem existence t0}, (\ref{IVP}),
(\ref{partial x of t0}), (\ref{g}), and (\ref{eta at time t0}), for
$(s,x)\in[0,1]\times D$,%
\begin{align}
\partial_{x}g(s,x)  &  =\partial_{t}\xi(t_{0}(s,x),x)\partial_{x}%
t_{0}(s,x)+\partial_{x}\xi(t_{0}(s,x),x)\nonumber\\
&  =\frac{\partial_{x}\xi(t_{0}(s,x),x)}{1+\left[  w^{\prime}(g(s,x))+s\varphi
^{\prime}(g(s,x))\right]  ^{2}}\label{partial x of g (new)}\\
&  \quad+\frac{w^{\prime}(g(s,x))+s\varphi^{\prime}(g(s,x))}{1+\left[
w^{\prime}(g(s,x))+s\varphi^{\prime}(g(s,x))\right]  ^{2}}\partial_{x}%
\eta(t_{0}(s,x),x).\nonumber
\end{align}
Differentiating with respect to $x$, we get%
\begin{align}
\partial_{x}^{2}g  &  =-\frac{2\left(  w^{\prime}(g)+s\varphi^{\prime
}(g)\right)  (w^{\prime\prime}(g)+s\varphi^{\prime\prime}(g))}{\left[
1+\left(  w^{\prime}(g)+s\varphi^{\prime}(g)\right)  ^{2}\right]  ^{2}%
}\partial_{x}g\partial_{x}\xi(t_{0},x)\nonumber\\
&  \quad+\frac{\partial_{x,t}^{2}\xi(t_{0},x)\partial_{x}t_{0}}{1+\left(
w^{\prime}(g)+s\varphi^{\prime}(g)\right)  ^{2}}+\frac{\partial_{x}^{2}%
\xi(t_{0},x)}{1+\left(  w^{\prime}(g)+s\varphi^{\prime}(g)\right)  ^{2}%
}\label{partial 2x of g (new)}\\
&  \quad+\frac{\left[  1-\left(  w^{\prime}(g)+s\varphi^{\prime}(g)\right)
^{2}\right]  (w^{\prime\prime}(g)+s\varphi^{\prime\prime}(g))}{\left[
1+\left(  w^{\prime}(g)+s\varphi^{\prime}(g)\right)  ^{2}\right]  ^{2}%
}\partial_{x}g\partial_{x}\eta(t_{0},x)\nonumber\\
&  \quad+\frac{(w^{\prime}(g)+s\varphi^{\prime}(g))\partial_{x,t}^{2}%
\eta(t_{0},x)\partial_{x}t_{0}}{1+\left(  w^{\prime}(g)+s\varphi^{\prime
}(g)\right)  ^{2}}+\frac{(w^{\prime}(g)+s\varphi^{\prime}(g))\partial_{x}%
^{2}\eta(t_{0},x)}{1+\left(  w^{\prime}(g)+s\varphi^{\prime}(g)\right)  ^{2}%
}\nonumber\\
&  =:I+II+III+IV+V+VI.\nonumber
\end{align}

It remains to study the regularity of $g$ at points $(s_{0},x_{0})$ with
$\varphi(x_{0})=0$.\smallskip

\noindent\textbf{Step 2:} \textbf{Regularity at points }$s\geq0$, $x\notin[
a,b]$\textbf{.} Let $s_{0}\geq0$ and $x_{0}\notin[ a,b]$. Since $\varphi
\equiv0$ outside $(a,b)$, by (\ref{g at s=0}) we have that $g(s,x)=x$ for all
$s\geq0$ and $x\in\mathbb{R}\setminus(a,b)$. It follows that for all $s\geq0$
and $x\in\mathbb{R}\setminus[ a,b]$,
\begin{align}
\partial_{x}g(s,x)  &  =1,\quad\partial_{s}%
g(s,x)=0,\label{partial g outside [a,b]}\\
\partial_{s}^{2}g(s,x)  &  =\partial_{x}^{2}g(s,x)=\partial_{x,s}%
^{2}g(s,x)=\partial_{s,x}^{2}g(s,x)=0.\nonumber
\end{align}

\noindent\textbf{Step 3: Continuity of }$g$\textbf{.} Let $s_{0}\geq0$ and let
$x_{0}\in[ a,b]$ be such that $\varphi(x_{0})=0$. By (\ref{t0 when u=0}) and
(\ref{g at s=0}), we have that $t_{0}(\cdot,x_{0})\equiv0$ and $g(\cdot
,x_{0})\equiv x_{0}$, respectively. Then by (\ref{IVP}), (\ref{g}), and
(\ref{g at s=0}), we have
\begin{align}
g(s,x)-x_{0}  &  =\xi(t_{0}(s,x),x)-\xi(0,x)+x-x_{0}\nonumber\\
&  =\int_{0}^{t_{0}(s,x)}\partial_{t}\xi(r,x)\,dr+x-x_{0}\label{g-a}\\
&  =-\int_{0}^{t_{0}(s,x)}[w^{\prime}(\xi(r,x))\varphi(\xi(r,x))\nonumber\\
&  \quad+\left(  \eta(r,x)-w(\xi(r,x))\right)  \varphi^{\prime}(\xi
(r,x))\,dr+x-x_{0}.\nonumber
\end{align}
Since $\xi(\cdot,x_{0})\equiv x_{0}$ and $\eta(\cdot,x_{0})\equiv w(x_{0})$ by
(\ref{solution when u=0}), it follows that%
\[
w^{\prime}(\xi(t,x_{0}))\varphi(\xi(t,x_{0}))+\left(  \eta(t,x_{0}%
)-w(\xi(t,x_{0}))\right)  \varphi^{\prime}(\xi(t,x_{0}))=0
\]
for all $t\in[0,1]$. By continuity with respect to initial data, we deduce that
the functions $(t,x)\mapsto\xi(t,x)$ and $(t,x)\mapsto\eta(t,x)$ are uniformly
continuous on compact sets, and so using also the facts that $w$ is smooth and
$\varphi\in C^{2}(\mathbb{R})$, we have that given $\varepsilon>0$ there
exists $\delta>0$ such that
\[
\left\vert w^{\prime}(\xi(t,x))\varphi(\xi(t,x))+\left(  \eta(t,x)-w(\xi
(t,x))\right)  \varphi^{\prime}(\xi(t,x))\right\vert \leq\varepsilon
\]
for all $t\in\left[  0,1\right]  $ and all $x$ with $\left\vert x-x_{0}%
\right\vert \leq\delta$. Since $0\leq t_{0}\leq1$ by (\ref{t0 bounded}), it
follows that
\[
\int_{0}^{t_{0}(s,x)}\left\vert w^{\prime}(\xi(r,x))\varphi(\xi(r,x))+\left(
\eta(r,x)-w(\xi(r,x))\right)  \varphi^{\prime}(\xi(r,x))\right\vert
\,dr\leq\varepsilon
\]
for all $(s,x) $ with $\left\vert x-x_{0}\right\vert \leq\delta$. By (\ref{g-a}) we obtain
\begin{equation}
\lim_{x\rightarrow x_{0}}g(s,x)=x_{0} \label{g continuous at x0}%
\end{equation}
uniformly for all $s\in[0,1]$. This shows that $g$ is continuous at
$(s_{0},x_{0})$.

In particular, if $\varphi\neq0$ in some interval $(\alpha,\beta)$ and
$\varphi(\alpha)=\varphi(\beta)=0$, by the continuity of $g$, it follows from
(\ref{formula t0}) that%
\begin{equation}
\lim_{(s,x) \rightarrow(s_{0},\alpha) ^{+}}t_{0}(s,x) =T_{0}( s_{0},\alpha) ,
\label{limit t0}%
\end{equation}
where%
\begin{align}
T_{0}(s,x) :  &  =\int_{0}^{s}\frac{1}{1+\left[  w^{\prime}(x)+r\varphi
^{\prime}(x)\right]  ^{2}}\,dr\label{T0(s,x)}\\
&  =\left\{
\begin{array}
[c]{ll}%
[\arctan(w^{\prime}(x)+s\varphi^{\prime}(x))-\arctan(w^{\prime}(x))]/\varphi
^{\prime}(x) & \text{if }\varphi^{\prime}(x)\neq0,\\
s/[1+(w^{\prime}(x))^{2}] & \text{if }\varphi^{\prime}(x)=0.
\end{array}
\right. \nonumber
\end{align}
Since $t_{0}(s,\alpha) =0$, this shows that the function $t_{0}$ is
discontinuous at $( s_{0},\alpha)$ for all $s_{0}>0$. A similar result holds
at the endpoint $\beta$.\smallskip

\noindent\textbf{Step 4: Existence and continuity of }$\partial_{s}g$
\textbf{and} $\partial_{x}g$. Let $s_{0}\geq0$ and let $x_{0}\in\left[
a,b\right]  $ be such that $\varphi(x_{0}) =0$. By (\ref{g at s=0}), we have
that $g(\cdot,x_{0})\equiv x_{0}$, and so
\begin{equation}
\partial_{s}g(s,x_{0})=0 \label{partial s of g at a}%
\end{equation}
for all $s\geq0$. On the other hand, if $\varphi\neq0$ in some interval
$(x_{0},x_{0}+\delta)$ (the case $(x_{0}-\delta,x_{0})$ is similar), by the
continuity of $g$ and (\ref{partial s of g (new)}),
\[
\partial_{s}g(s,x)\rightarrow\frac{-\varphi(x_{0})\left[  w^{\prime}%
(x_{0})+s\varphi^{\prime}(x_{0})\right]  }{1+\left[  w^{\prime}(x_{0}%
)+s\varphi^{\prime}(x_{0})\right]  ^{2}}=0
\]
as $(s,x)\rightarrow(s_{0},x_{0})^{+}$. Hence, $\partial_{s}g$ is continuous
at $(s_{0},x_{0})$ for all $s_{0}\geq0$.

Next, we prove the existence and continuity of $\partial_{x}g$ at $(s_{0}%
,x_{0})$ for all $s_{0}\geq0$. We assume, as before, that $\varphi\neq0$ in some
interval $(x_{0},x_{0}+\delta)$ (the case $(x_{0}-\delta,x_{0})$ is similar).
Differentiating (\ref{IVP}) with respect to $x$, we obtain%
\begin{equation}
\left\{
\begin{array}
[c]{l}%
\partial_{t}( \partial_{x}\xi) =-[w^{\prime\prime}(\xi)\varphi(\xi)+\left(
\eta-w(\xi)\right)  \varphi^{\prime\prime}(\xi)]\partial_{x}\xi-\varphi
^{\prime}(\xi)\partial_{x}\eta,\\
\partial_{t}\left(  \partial_{x}\eta\right)  =\varphi^{\prime}(\xi
)\partial_{x}\xi,\\
\partial_{x}\xi(0,x)=1,\\
\partial_{x}\eta(0,x)=w^{\prime}(x) .
\end{array}
\right.  \label{IVP for partial x}%
\end{equation}
Since $\xi(\cdot,x_{0})\equiv x_{0}$ and $\eta(\cdot,x_{0})\equiv w(x_{0})$ by
(\ref{solution when u=0}), we have that $\partial_{x}\xi(\cdot,x_{0})$ and
$\partial_{x}\eta(\cdot,x_{0})$ solve the system%
\[
\left\{
\begin{array}
[c]{l}%
\partial_{t}\left(  \partial_{x}\xi(\cdot,x_{0})\right)  =-\varphi^{\prime
}(x_{0})\partial_{x}\eta(\cdot,x_{0}),\\
\partial_{t}\left(  \partial_{x}\eta(\cdot,x_{0})\right)  =\varphi^{\prime
}(x_{0})\partial_{x}\xi(\cdot,x_{0}),\\
\partial_{x}\xi(0,x_{0})=1,\\
\partial_{x}\eta(0,x_{0})=w^{\prime}( x_{0} ) ,
\end{array}
\right.
\]
and so
\begin{align}
\partial_{x}\xi(t,x_{0})  &  =\cos(\varphi^{\prime}(x_{0})t)-w^{\prime}%
(x_{0})\sin(\varphi^{\prime}(x_{0})t),\label{partialx x of xi and eta at x0}\\
\partial_{x}\eta(t,x_{0})  &  =w^{\prime}(x_{0})\cos(\varphi^{\prime}%
(x_{0})t)+\sin(\varphi^{\prime}(x_{0})t).\nonumber
\end{align}
By the the continuity of $\partial_{x}\xi$ and $\partial_{x}\eta$,
(\ref{limit t0}) and (\ref{partialx x of xi and eta at x0}),%
\begin{align}
\lim_{(s,x)\rightarrow(s_{0},x_{0})^{+}}  &  \partial_{x}\xi(t_{0}%
(s,x),x)=\partial_{x}\xi(T_{0}(s_{0},x_{0}),x_{0})\nonumber\\
&  =\cos(\varphi^{\prime}(x_{0})T_{0}(s,x_{0}))-w^{\prime}(x_{0})\sin
(\varphi^{\prime}(x_{0})T_{0}(s,x_{0}))\label{limit partial x xi}\\
&  =\frac{\sqrt{1+\left[  w^{\prime}(x_{0})\right]  ^{2}}}{\sqrt{1+\left[
w^{\prime}(x_{0})+s_{0}\varphi^{\prime}(x_{0})\right]  ^{2}}}\nonumber
\end{align}
and%
\begin{align}
\lim_{(s,x)\rightarrow(s_{0},x_{0})^{+}}  &  \partial_{x}\eta(t_{0}%
(s,x),x)=\partial_{x}\eta(T_{0}(s_{0},x_{0}),x_{0})\nonumber\\
&  =w^{\prime}(x_{0})\cos(\varphi^{\prime}(x_{0})T_{0}(s,x_{0}))+\sin
(\varphi^{\prime}(x_{0})T_{0}(s,x_{0}))\label{limit partial x eta}\\
&  =\frac{(w^{\prime}(x_{0})+s_{0}\varphi^{\prime}(x_{0}))\sqrt{1+\left[
w^{\prime}(x_{0})\right]  ^{2}}}{\sqrt{1+\left[  w^{\prime}(x_{0}%
)+s_{0}\varphi^{\prime}(x_{0})\right]  ^{2}}},\nonumber
\end{align}
where we have used the formulas%
\begin{align*}
\cos\left(  \arctan( x+y ) -\arctan y\right)   &  =\frac{1+xy+y^{2}}%
{\sqrt{1+y^{2}}\sqrt{1+(x+y)^{2}}},\\
\sin\left(  \arctan( x+y) -\arctan y\right)   &  =\frac{x}{\sqrt{1+y^{2}}%
\sqrt{1+(x+y)^{2}}}.
\end{align*}
Note that
\begin{equation}
\left[  w^{\prime}(x_{0})+s\varphi^{\prime}(x_{0})\right]  \partial_{x}%
\xi(T_{0}(s,x_{0}),x_{0})-\partial_{x}\eta(T_{0}(s,x_{0}),x_{0})=0
\label{ell 1 and ell2 =0}%
\end{equation}
for every $s\in\left[  0,1\right]  $.

By (\ref{partial x of g (new)}), (\ref{limit partial x xi}),
(\ref{limit partial x eta}), we obtain%
\begin{equation}
\lim_{(s,x)\rightarrow(s_{0},x_{0})^{+}}\partial_{x}g(s,x)=\frac
{\sqrt{1+\left[  w^{\prime}(x_{0})\right]  ^{2}}}{\sqrt{1+\left[  w^{\prime
}(x_{0})+s_{0}\varphi^{\prime}(x_{0})\right]  ^{2}}}.
\label{limit partial x g}%
\end{equation}
By the continuity of $g$ proved in Step~3,%
\[
\lim_{x\rightarrow x_{0}^{+}}\frac{g(s_{0},x)-g(s_{0},x_{0})}{x-x_{0}}%
=\frac{0}{0},
\]
and so we can apply L'H\^{o}pital's rule to the function
$x\mapsto g(s_{0},x)$ to conclude that there exists
\[
\lim_{x\rightarrow x_{0}^{+}}\frac{g(s_{0},x)-g(s_{0},x_{0})}{x-x_{0}}%
=\lim_{x\rightarrow x_{0}^{+}}\frac{\partial_{x}g(s_{0},x)}{1}=\frac
{\sqrt{1+\left[  w^{\prime}(x_{0})\right]  ^{2}}}{\sqrt{1+\left[  w^{\prime
}(x_{0})+s_{0}\varphi^{\prime}(x_{0})\right]  ^{2}}}.
\]
If $\varphi\neq0$ also in some interval $(x_{0}-\delta_{1},x_{0})$, then we
conclude in the same way that%
\[
\lim_{x\rightarrow x_{0}^{-}}\frac{g(s_{0},x)-g(s_{0},x_{0})}{x-x_{0}}%
=\lim_{x\rightarrow x_{0}^{-}}\frac{\partial_{x}g(s_{0},x)}{1}=\frac
{\sqrt{1+\left[  w^{\prime}(x_{0})\right]  ^{2}}}{\sqrt{1+\left[  w^{\prime
}(x_{0})+s_{0}\varphi^{\prime}(x_{0})\right]  ^{2}}},
\]
and so we deduce that there exists
\begin{equation}
\partial_{x}g(s_{0},x_{0})=\frac{\sqrt{1+\left[  w^{\prime}(x_{0})\right]
^{2}}}{\sqrt{1+\left[  w^{\prime}(x_{0})+s_{0}\varphi^{\prime}(x_{0})\right]
^{2}}} \label{partial x of g at x=a}%
\end{equation}
and that $\partial_{x}g$ is continuous at $(s_{0},x_{0})$. On the other hand,
if $\varphi=0$ in some interval $(x_{0}-\delta_{1},x_{0})$, then $x_{0}=a$,
and $\varphi^{\prime}( x_{0} ) =0$. It follows that the limit in
(\ref{limit partial x g}) is $1$, and so by (\ref{partial g outside [a,b]}) we
obtain again that there exists $\partial_{x}g(s_{0},x_{0})=1$ and that
$\partial_{x}g$ is continuous at $(s_{0},x_{0})$.\smallskip

\noindent\textbf{Step 5: Existence and continuity of }$\partial_{s}^{2}g$,
$\partial_{x,s}^{2}g$, \textbf{and} $\partial_{s,x}^{2}g$. Let $s_{0}\geq0$
and let $x_{0}\in\left[  a,b\right]  $ be such that $\varphi( x_{0}) =0$. We
assume, as before, that $\varphi\neq0$ in some interval $(x_{0},x_{0}+\delta)$
(the case $(x_{0}-\delta,x_{0})$ is similar). By (\ref{partial s of g at a}),
we have that $\partial_{s}^{2}g(s,x_{0})=0$ for all $s\geq0$. On the other
hand, by (\ref{partial 2s of g}) and the continuity of $g$ and $\partial_{s}%
g$,%
\begin{equation}
\lim_{(s,x)\rightarrow(s_{0},a)^{+}}\partial_{s}^{2}g(s,x) =0.\nonumber
\end{equation}
By (\ref{partial 2sx of g}) and the continuity of $g$, $\partial_{s}g$, and
$\partial_{x}g$,%
\begin{equation}
\lim_{(s,x)\rightarrow(s_{0},x_{0})^{+}}\partial_{x,s}^{2}g(s,x)
=-\frac{\varphi^{\prime}(x_{0})\left[  w^{\prime}(x_{0})+s_{0}\varphi^{\prime
}(x_{0})\right]  \partial_{x}g(s_{0},x_{0})}{1+\left[  w^{\prime}%
(x_{0})+s\varphi^{\prime}(x_{0})\right]  ^{2}}.\nonumber
\end{equation}
On the other hand, by Step~4,%
\[
\lim_{x\rightarrow x_{0}^{+}}\frac{\partial_{s}g(s_{0},x)-\partial_{s}%
g(s_{0},x_{0})}{x-x_{0}}=\frac{0}{0},
\]
and so we can apply L'H\^{o}pital's rule to the function
$x\mapsto\partial_{s}g(s_{0},x)$ to conclude that there exists%
\begin{align}
\lim_{x\rightarrow x_{0}^{+}}\frac{\partial_{s}g(s_{0},x)-\partial_{s}%
g(s_{0},x_{0})}{x-x_{0}}  &  =\lim_{x\rightarrow x_{0}^{+}}\frac
{\partial_{x,s}^{2}g(s_{0},x)}{1}\label{295}\\
&  =-\frac{\varphi^{\prime}(x_{0})\left[  w^{\prime}(x_{0})+s_{0}%
\varphi^{\prime}(x_{0})\right]  \partial_{x}g(s_{0},x_{0})}{1+\left[
w^{\prime}(x_{0})+s\varphi^{\prime}(x_{0})\right]  ^{2}}.\nonumber
\end{align}
If $\varphi\neq0$ also in some interval $(x_{0}-\delta_{1},x_{0})$, then we
deduce as in the previous step that that there exists $\partial
_{x,s}^{2}g(s_{0},x_{0})$ and that $\partial_{x,s}^{2}g$ is continuous at
$(s_{0},x_{0})$. On the other hand, if $\varphi=0$ in some interval
$(x_{0}-\delta_{1},x_{0})$, then $x_{0}=a$, and $\varphi^{\prime}( x_{0}) =0$.
It follows from (\ref{295}) and (\ref{partial g outside [a,b]}) we obtain
again that there exists $\partial_{x,s}^{2}g(s_{0},x_{0})=0$ and that
$\partial_{x,s}^{2}g$ is continuous at $(s_{0},x_{0})$.

In both cases we can apply Schwartz's theorem to conclude that there exists
$\partial_{s,x}^{2}g(s_{0},x_{0})$ and that%
\[
\partial_{s,x}^{2}g(s_{0},x_{0})=\partial_{x,s}^{2}g(s_{0},x_{0}).
\]

\noindent\textbf{Step 6: Existence and continuity of }$\partial_{x}^{2}g$. By
Step~3, (\ref{limit partial x xi}), (\ref{limit partial x eta}),
(\ref{limit partial x g}), and (\ref{partial 2x of g (new)}), we have%
\begin{align}
\lim_{(s,x)\rightarrow(s_{0},x_{0})^{+}}I  &  =-\frac{2\left(  w^{\prime
}(x_{0})+s\varphi^{\prime}(x_{0})\right)  (w^{\prime\prime}(x_{0}%
)+s\varphi^{\prime\prime}(x_{0}))}{\left[  1+\left(  w^{\prime}(x_{0}%
)+s\varphi^{\prime}(x_{0})\right)  ^{2}\right]  ^{2}}\label{limit I}\\
&  \quad\times\partial_{x}g(s_{0},x_{0})\partial_{x}\xi(T_{0}(s_{0}%
,x_{0}),x_{0}),\nonumber\\
\lim_{(s,x)\rightarrow(s_{0},x_{0})^{+}}IV  &  =\frac{\left[  1-\left(
w^{\prime}(x_{0})+s\varphi^{\prime}(x_{0})\right)  ^{2}\right]  (w^{\prime
\prime}(x_{0})+s\varphi^{\prime\prime}(x_{0}))}{\left[  1+\left(  w^{\prime
}(x_{0})+s\varphi^{\prime}(x_{0})\right)  ^{2}\right]  ^{2}}\label{limit IV}\\
&  \quad\times\partial_{x}g(s_{0},x_{0})\partial_{x}\eta(T_{0}(s_{0}%
,x_{0}),x_{0}).\nonumber
\end{align}
On the other hand, by Step~3, the continuity of $\partial_{x}^{2}\xi$ and
$\partial_{x}^{2}\eta$, (\ref{partial 2x of g (new)}), and (\ref{limit t0}),%
\begin{equation}
\lim_{(s,x)\rightarrow(s_{0},x_{0})^{+}}III=\frac{\partial_{x}^{2}\xi
(T_{0}(s_{0},x_{0}),x_{0})}{1+\left(  w^{\prime}(x_{0})+s_{0}\varphi^{\prime
}(x_{0})\right)  ^{2}}, \label{limit III}%
\end{equation}
and
\begin{equation}
\lim_{(s,x)\rightarrow(s_{0},x_{0})^{+}}VI=\frac{(w^{\prime}(x_{0}%
)+s\varphi^{\prime}(x_{0}))\partial_{x}^{2}\eta(T_{0}(s_{0},x_{0}),x_{0}%
)}{1+\left(  w^{\prime}(x_{0})+s\varphi^{\prime}(x_{0})\right)  ^{2}}.
\label{limit VI}%
\end{equation}

It remains to estimate $II$ and $V$ in (\ref{partial 2x of g (new)}). By
Taylor's formula, we obtain%
\[
w^{\prime}(z)+s\varphi^{\prime}(z)=w^{\prime}(x_{0})+s\varphi^{\prime}%
(x_{0})+(w^{\prime\prime}(x_{0})+s\varphi^{\prime\prime}(x_{0})) ( z-x_{0}) +o
( z-x_{0}) .
\]
Hence, also by (\ref{g continuous at x0}),%
\begin{align}
&  \frac{\left[  w^{\prime}(g(s,x))+s\varphi^{\prime}(g(s,x))\right]
\partial_{x}\xi(t_{0}(s,x),x)-\partial_{x}\eta(t_{0}(s,x),x)}{g(s,x)-x_{0}%
}\label{901}\\
&  =\frac{(w^{\prime}(x_{0})+s\varphi^{\prime}(x_{0}))\partial_{x}\xi
(t_{0}(s,x),x)-\partial_{x}\eta(t_{0}(s,x),x)}{g(s,x)-x_{0}}\nonumber\\
&  \quad+\left(  w^{\prime\prime}(x_{0})+s\varphi^{\prime\prime}(x_{0})+o( 1)
\right)  \partial_{x}\xi(t_{0}(s,x),x).\nonumber
\end{align}
By repeated applications of the mean value theorem, we have that
\begin{align}
(w^{\prime}  &  (x_{0})+s\varphi^{\prime}(x_{0}))\partial_{x}\xi
(t_{0}(s,x),x)-\partial_{x}\eta(t_{0}(s,x),x)\nonumber\\
&  =(w^{\prime}(x_{0})+s\varphi^{\prime}(x_{0}))\partial_{x}\xi(t_{0}%
(s,x),x_{0})-\partial_{x}\eta(t_{0}(s,x),x_{0})\nonumber\\
&  \quad+(x-x_{0})[(w^{\prime}(x_{0})+s\varphi^{\prime}(x_{0}))\partial
_{x}^{2}\xi(t_{0}(s,x),x_{1})-\partial_{x}^{2}\eta(t_{0}(s,x),x_{1}%
)]\label{902}\\
&  =(x-x_{0})\Big\{ \Big( \frac{t_{0}(s,x)-T_{0}(s,x_{0})}{x-x_{0}%
}\Big) [(w^{\prime}(x_{0})+s\varphi^{\prime}(x_{0}))\partial_{x,t}^{2}%
\xi(t_{1},x_{0})\nonumber\\
&  \quad-\partial_{x,t}^{2}\eta(t_{1},x_{0})]+ [(w^{\prime}(x_{0}%
)+s\varphi^{\prime}(x_{0}))\partial_{x}^{2}\xi(t_{0}(s,x),x_{1})-\partial
_{x}^{2}\eta(t_{0}(s,x),x_{1})]\Big\}\nonumber
\end{align}
for some $x_{1}$ between $x$ and $x_{0}$ and for some $t_{1}(s,x)$ between
$t_{0}(s,x)$ and $T_{0}(s,x_{0})$, and where we have used
(\ref{ell 1 and ell2 =0}).

By (\ref{g at s=0}), (\ref{formula t0}), (\ref{T0(s,x)}), and again the mean
value theorem, we get%
\begin{align}
&  t_{0}(s,x)-T_{0}(s,x_{0})\nonumber\\
&  =\int_{0}^{s}\frac{(w^{\prime}(x_{0})+r\varphi^{\prime}(x_{0}%
))^{2}-(w^{\prime}(g(r,x) )+r\varphi^{\prime}(g( r,x) ))^{2}}{[1+(w^{\prime
}(g(r,x) )+r\varphi^{\prime}(g(r,x) ))^{2}][1+(w^{\prime}(x_{0})+r\varphi
^{\prime}(x_{0}))^{2}]}\,dr\label{903}\\
&  =-(x-x_{0})\int_{0}^{s}\frac{2(w^{\prime}(g( r,c) )+r\varphi^{\prime}(g(
r,c) ))(w^{\prime\prime}(g( r,c) )+r\varphi^{\prime\prime}(g( r,c)
))\partial_{x}g(r,c)}{[1+(w^{\prime}(g(r,x) )+r\varphi^{\prime}(g( r,x)
))^{2}][1+(w^{\prime}(x_{0})+r\varphi^{\prime}(x_{0}))^{2}]}\,dr\nonumber
\end{align}
for some $c=c( r,x,x_{0})$ between $x$ and $x_{0}$. Hence, by (\ref{g at s=0})
and the continuity of $g$ and $\partial_{x}g$,%
\begin{align}
&  \lim_{(s,x)\rightarrow(s_{0},x_{0})^{+}}\frac{t_{0}(s,x)-T_{0}(s,x_{0}%
)}{x-x_{0}}\label{904}\\
&  \quad=-\int_{0}^{s_{0}}\frac{2(w^{\prime}(x_{0})+r\varphi^{\prime}%
(x_{0}))(w^{\prime\prime}(x_{0})+r\varphi^{\prime\prime}(x_{0}))\partial
_{x}g(r,x_{0})}{[1+(w^{\prime}(x_{0})+r\varphi^{\prime}(x_{0}))^{2}]^{2}%
}\,dr=:\ell_{1}.\nonumber
\end{align}
By (\ref{g at s=0}), (\ref{limit partial x g}), and the mean value theorem, we
deduce that
\begin{equation}
\frac{g(s,x) -x_{0}}{x-x_{0}}=\frac{g(s,x) -g( s,x_{0}) }{x-x_{0}}%
=\partial_{x}g( s,\theta) \rightarrow\partial_{x}g(s_{0},x_{0})
\label{limit xi-a}%
\end{equation}
as $(s,x)\rightarrow(s_{0},x_{0})$. Hence, letting $(s,x)\rightarrow
(s_{0},x_{0})$ in (\ref{901}) and using (\ref{902}), (\ref{904}), and
(\ref{limit xi-a}) gives%
\begin{align}
&  \lim_{(s,x)\rightarrow(s_{0},x_{0})^{+}}\frac{\left[  w^{\prime
}(g(s,x))+s\varphi^{\prime}(g(s,x))\right]  \partial_{x}\xi(t_{0}%
(s,x),x)-\partial_{x}\eta(t_{0}(s,x),x)}{g(s,x)-x_{0}}\nonumber\\
&  =\frac{\left\{  \ell_{1}[(w^{\prime}(x_{0})+s\varphi^{\prime}%
(x_{0}))\partial_{x,t}^{2}\xi(T_{0}(s_{0},x_{0}),x_{0})-\partial_{x,t}^{2}%
\eta(T_{0}(s_{0},x_{0}),x_{0})]\right.  }{\partial_{x}g(s_{0},x_{0}%
)}\label{905}\\
&  \quad+\left.  [(w^{\prime}(x_{0})+s\varphi^{\prime}(x_{0}))\partial_{x}%
^{2}\xi(T_{0}(s_{0},x_{0}),x_{0})-\partial_{x}^{2}\eta(T_{0}(s_{0}%
,x_{0}),x_{0})]\right\} \nonumber\\
&  \quad+\left(  w^{\prime\prime}(x_{0})+s\varphi^{\prime\prime}%
(x_{0})\right)  \partial_{x}\xi(T_{0}(s_{0},x_{0}),x_{0})=:\ell_{2}.\nonumber
\end{align}

By (\ref{eta at time t0}), (\ref{partial x of t0}), (\ref{IVP for partial x}),
(\ref{partial 2x of g (new)}),%
\begin{align}
II  &  = -\frac{[w^{\prime\prime}(g)\varphi(g)+s\varphi(g)\varphi
^{\prime\prime}(g)]\partial_{x}\xi(t_{0},x)-\varphi^{\prime}(g)\partial
_{x}\eta(t_{0},x)}{1+\left(  w^{\prime}(g)+s\varphi^{\prime}(g)\right)  ^{2}%
}\nonumber\\
&  \times\frac{\left[  w^{\prime}(g)+s\varphi^{\prime}(g)\right]  \partial
_{x}\xi(t_{0},x)-\partial_{x}\eta(t_{0},x)}{\varphi(g)\left[  1+\left(
w^{\prime}(g)+s\varphi^{\prime}(g)\right)  ^{2}\right]  }\nonumber\\
&  =-[w^{\prime\prime}(g)+s\varphi^{\prime\prime}(g)]\partial_{x}\xi
(t_{0},x)\frac{\left[  w^{\prime}(g)+s\varphi^{\prime}(g)\right]  \partial
_{x}\xi(t_{0},x)-\partial_{x}\eta(t_{0},x)}{\left[  1+\left(  w^{\prime
}(g)+s\varphi^{\prime}(g)\right)  ^{2}\right]  ^{2}}\label{II=IIa+IIb}\\
&  \quad-\frac{\varphi^{\prime}(g)}{\varphi(g)}\partial_{x}\eta(t_{0}%
,x)\frac{\left[  w^{\prime}(g)+s\varphi^{\prime}(g)\right]  \partial_{x}%
\xi(t_{0},x)-\partial_{x}\eta(t_{0},x)}{\left[  1+\left(  w^{\prime
}(g)+s\varphi^{\prime}(g)\right)  ^{2}\right]  ^{2}}\nonumber\\
&  =II_{a}+II_{b}.\nonumber
\end{align}
By Step~3, (\ref{limit partial x xi}), (\ref{limit partial x eta}), and
(\ref{ell 1 and ell2 =0}),
\begin{align}
&  \lim_{(s,x)\rightarrow(s_{0},x_{0})^{+}}II_{a}=-[w^{\prime\prime}%
(x_{0})+s_{0}\varphi^{\prime\prime}(x_{0})]\partial_{x}\xi(T_{0}(s_{0}%
,x_{0}),x_{0})\label{limit IIa}\\
&  \quad\times\frac{\left[  w^{\prime}(x_{0})+s_{0}\varphi^{\prime}%
(x_{0})\right]  \partial_{x}\xi(T_{0}(s_{0},x_{0}),x_{0})-\partial_{x}%
\eta(T_{0}(s_{0},x_{0}),x_{0})}{\left[  1+\left(  w^{\prime}(x_{0}%
)+s_{0}\varphi^{\prime}(x_{0})\right)  ^{2}\right]  ^{2}}=0.\nonumber
\end{align}
Since $\varphi$ is a polynomial with $\varphi( x_{0}) =0$, we may write%
\begin{equation}
\varphi( z ) =p ( z) ( z-x_{0}) ^{k}, \label{u polynomial near a}%
\end{equation}
where $p$ is a polynomial with $p( x_{0}) \neq0$ and $k\geq1$. In turn,
\begin{equation}
\varphi^{\prime}( z) = ( z-x_{0}) ^{k-1}[p^{\prime}( z) ( z-x_{0}) +kp( z) ].
\label{derrivative u}%
\end{equation}
By (\ref{u polynomial near a}), (\ref{derrivative u}),
\begin{align*}
II_{b}  &  =-\frac{[p^{\prime}( g) ( g-x_{0}) +kp( g) ]\partial_{x}\eta
(t_{0},x)}{\left[  1+\left(  w^{\prime}(g)+s\varphi^{\prime}(g)\right)
^{2}\right]  ^{2}p(g)}\\
&  \quad\times\frac{\left[  w^{\prime}(g)+s\varphi^{\prime}(g)\right]
\partial_{x}\xi(t_{0},x)-\partial_{x}\eta(t_{0},x)}{g-x_{0}},
\end{align*}
and so by Step~3, (\ref{f and itz derivatives zero at endpoints}),
(\ref{limit partial x eta}), and (\ref{905}),%
\begin{equation}
\lim_{(s,x)\rightarrow(s_{0},x_{0})^{+}}II_{b}=-\frac{k \partial_{x}\eta
(T_{0}(s_{0},x_{0}),x_{0})}{\left[  1+\left(  w^{\prime}(x_{0})+s_{0}%
\varphi^{\prime}(x_{0})\right)  ^{2}\right]  ^{2}}\ell_{2}. \label{limit IIb}%
\end{equation}

By (\ref{eta at time t0}), (\ref{u polynomial near a}), (\ref{partial x of t0}%
), (\ref{IVP for partial x}), (\ref{partial 2x of g (new)}),
(\ref{derrivative u}),
\begin{align*}
V  &  =\frac{(w^{\prime}(g)+s\varphi^{\prime}(g))\varphi^{\prime}(g
)\partial_{x}\xi(t_{0},x)}{1+\left(  w^{\prime}(g)+s\varphi^{\prime
}(g)\right)  ^{2}}\frac{\left[  w^{\prime}(g)+s\varphi^{\prime}(g)\right]
\partial_{x}\xi(t_{0},x)-\partial_{x}\eta(t_{0},x)}{\varphi(g)\left[
1+\left(  w^{\prime}(g)+s\varphi^{\prime}(g)\right)  ^{2}\right]  }\\
&  =\frac{(w^{\prime}(g)+s\varphi^{\prime}(g))[p^{\prime}( g) ( g-x_{0}) +kp(
g) ]\partial_{x}\xi(t_{0},x)}{\left[  1+\left(  w^{\prime}(g)+s\varphi
^{\prime}(g)\right)  ^{2}\right]  ^{2}p(g)}\\
&  \quad\times\frac{\left[  w^{\prime}(g)+s\varphi^{\prime}(g)\right]
\partial_{x}\xi(t_{0},x)-\partial_{x}\eta(t_{0},x)}{g-x_{0}},
\end{align*}
and so by Step~3, (\ref{limit partial x eta}), and (\ref{905}),%
\begin{equation}
\lim_{(s,x)\rightarrow(s_{0},x_{0})^{+}}V=\frac{(w^{\prime}(x_{0}%
)+s\varphi^{\prime}(x_{0}))k \partial_{x}\xi(T_{0}(s_{0},x_{0}),x_{0}%
)}{\left[  1+\left(  w^{\prime}(x_{0})+s_{0}\varphi^{\prime}(x_{0})\right)
^{2}\right]  ^{2}}\ell_{2}. \label{limit V}%
\end{equation}
Finally, by (\ref{partial 2x of g (new)}), (\ref{limit I}), (\ref{limit IV}),
(\ref{II=IIa+IIb}), (\ref{limit IIa}), (\ref{limit IIb}), (\ref{limit V}),
(\ref{limit III}), and (\ref{limit VI}), there exists
\[
\lim_{(s,x)\rightarrow(s_{0},x_{0})^{+}}\partial_{x}^{2}g(s,x) =\ell_{3}%
\in\mathbb{R}.
\]
By the continuity of $\partial_{x}g$ proved in Step~4,%
\[
\lim_{x\rightarrow x_{0}^{+}}\frac{\partial_{x}g(s_{0},x)-\partial_{x}%
g(s_{0},x_{0})}{x-x_{0}}=\frac{0}{0},
\]
and so we can apply L'H\^{o}pital's rule to the function
$x\mapsto\partial_{x}g(s_{0},x)$ to conclude that there exists
\[
\lim_{x\rightarrow x_{0}^{+}}\frac{\partial_{x}g(s_{0},x)-\partial_{x}%
g(s_{0},x_{0})}{x-x_{0}}=\lim_{x\rightarrow x_{0}^{+}}\frac{\partial_{x}%
^{2}g(s_{0},x)}{1}=\ell_{3}.
\]
If $\varphi\neq0$ also in some interval $(x_{0}-\delta_{1},x_{0})$, then the
limit as $x\rightarrow x_{0}^{-}$ is still $\ell_{3}$, and so there exists
$\partial_{x}^{2}g(s_{0},x_{0})=\ell_{3}$ and $\partial_{x}^{2}g$ is
continuous at $(s_{0},x_{0})$. On the other hand, if $\varphi=0$ in some
interval $(x_{0}-\delta_{1},x_{0})$, then $x_{0}=a$, and $\varphi^{\prime}(
x_{0}) =\varphi^{\prime\prime}(x_{0})=0$.

Then by (\ref{limit partial x xi}), (\ref{limit partial x eta}),
(\ref{limit partial x g}),
\begin{align}
\partial_{x}\xi(T_{0}(s_{0},x_{0}),x_{0})  &  =1,\quad\partial_{x}\eta
(T_{0}(s_{0},x_{0}),x_{0})=w^{\prime}(x_{0}),\label{1000}\\
\quad\partial_{x}g(s_{0},x_{0})  &  =1.\nonumber
\end{align}
To calculate $\partial_{x}^{2}\xi(T_{0}(s_{0},x_{0}),x_{0})$ and $\partial
_{x}^{2}\eta(T_{0}(s_{0},x_{0}),x_{0})$, differentiate
(\ref{IVP for partial x}) with respect to $x$ to obtain%
\[%
\begin{cases}
\partial_{t}\left(  \partial_{x}^{2}\xi\right)  =-[w^{\prime\prime\prime}%
(\xi)\varphi(\xi)+w^{\prime\prime}(\xi)\varphi^{\prime}(\xi)-w^{\prime}%
(\xi)\varphi^{\prime\prime}(\xi)\\
\quad+\left(  \eta-w(\xi)\right)  \varphi^{\prime\prime\prime}(\xi
)](\partial_{x}\xi)^{2}\\
\quad-[w^{\prime\prime}(\xi)\varphi(\xi)+\left(  \eta-w(\xi)\right)
\varphi^{\prime\prime}(\xi)]\partial_{x}^{2}\xi-\varphi^{\prime}(\xi
)\partial_{x}^{2}\eta-2\varphi^{\prime\prime}(\xi)\partial_{x}\xi\partial
_{x}\eta, \smallskip\\
\partial_{t}\left(  \partial_{x}^{2}\eta\right)  =\varphi^{\prime\prime}%
(\xi)(\partial_{x}\xi)^{2}+\varphi^{\prime}(\xi)\partial_{x}^{2}\xi,
\smallskip\\
\partial_{x}^{2}\xi(0,x)=0, \ \partial_{x}^{2}\eta(0,x)=w^{\prime\prime}(x) .
\end{cases}
\]
Since $\xi(\cdot,x_{0})\equiv x_{0}$ and $\eta(\cdot,x_{0})\equiv w(x_{0})$ by
(\ref{solution when u=0}), we have that $\partial_{x}^{2}\xi(\cdot,x_{0})$ and
$\partial_{x}^{2}\eta(\cdot,x_{0})$ solve the system%
\[%
\begin{cases}
\partial_{t}\left(  \partial_{x}^{2}\xi(\cdot,x_{0})\right)  =0,\\
\partial_{t}\left(  \partial_{x}^{2}\eta(\cdot,x_{0})\right)  =0,\\
\partial_{x}\xi(0,x_{0})=0, \ \partial_{x}\eta(0,x_{0})=w^{\prime\prime}(
x_{0}) ,
\end{cases}
\]
and so
\begin{equation}
\partial_{x}^{2}\xi(t,x_{0})\equiv0,\quad\partial_{x}^{2}\eta(t,x_{0})\equiv
w^{\prime\prime}(x_{0}). \label{1001}%
\end{equation}

By (\ref{partial 2x of g (new)}), (\ref{limit I}), (\ref{limit IV}),
(\ref{II=IIa+IIb}), (\ref{limit IIa}), (\ref{limit IIb}), (\ref{limit V}),
(\ref{limit III}), (\ref{limit VI}), (\ref{1000}), and (\ref{1001}), we have
that
\begin{align*}
\lim_{(s,x)\rightarrow(s_{0},x_{0})^{+}}\partial_{x}^{2}g(s,x)  &
=-\frac{2w^{\prime}(x_{0})w^{\prime\prime}(x_{0})}{\left[  1+\left(
w^{\prime}(x_{0})\right)  ^{2}\right]  ^{2}}-\frac{k w^{\prime}(x_{0}%
)}{\left[  1+\left(  w^{\prime}(x_{0})\right)  ^{2}\right]  ^{2}}\ell_{2}+0\\
&  \quad+\frac{\left[  1-\left(  w^{\prime}(x_{0})\right)  ^{2}\right]
w^{\prime\prime}(x_{0})}{\left[  1+\left(  w^{\prime}(x_{0})\right)
^{2}\right]  ^{2}}w^{\prime}( x_{0}) +\frac{w^{\prime}(x_{0})k }{\left[
1+\left(  w^{\prime}(x_{0})\right)  ^{2}\right]  ^{2}}\ell_{2}\\
&  \quad+\frac{w^{\prime}(x_{0})w^{\prime\prime}( x_{0}) }{1+\left(
w^{\prime}(x_{0})\right)  ^{2}}=0,
\end{align*}
and so we can conclude, as before, that $\partial_{x}^{2}g$ exists and is
continuous at $(s_{0},x_{0})$.
\end{proof}

\section*{Acknowledgements}

The authors acknowledge the Center for Nonlinear Analysis (NSF PIRE
Grant No.\ OISE-0967140) where part of this work was carried out. The research
of I.~Fonseca was partially funded by the National Science Foundation under
Grant No.\ DMS-1411646 and that of G.~Leoni under Grant No.\ DMS-1412095. 
M.G.~Mora acknowledges support by the European Research Council under Grant No.\ 290888.
The authors would like to thank Bob Pego for his helpful insights.

\Addresses

\end{document}